\documentclass[12pt]{artformath}
\begin{document}
\title{Harmonic metrics of $\mathrm{SO}_{0}(n,n)$-Higgs bundles in the Hitchin section on non-compact hyperbolic surfaces}
\author{Weihan Ma \thanks{Chern Institute of Mathematics and LPMC, Nankai University, Tianjin 300071, China
\\ \email{1120210009@mail.nankai.edu.cn}}}
\date{}
\maketitle
\begin{abstract}
    Let $X$ be a Riemann surface. Hitchin constructed the $G$-Higgs bundles in the Hitchin section for a split real form $G$ of a complex simple Lie group,using the canonical line bundle $K$ and some holomorphic differentials $\symbfit{q}$. We study the case of ${\mathrm{SO}_0(n,n)}$. In our work, we establish the existence of harmonic metrics for these Higgs bundles, which are compatible with the ${\mathrm{SO}_0(n,n)}$-structure on any non-compact hyperbolic Riemann surface. Furthermore, these harmonic metrics weakly dominate $h_X$, the natural diagonal harmonic metric induced by the unique complete K\"ahler hyperbolic metric $g_X$ on $X$. Assuming these holomorphic differentials are all bounded with respect to $g_X$, we prove the uniqueness of such a harmonic metric.
    \keywords{Higgs bundles, harmonic metrics, non-compact surfaces, Hitchin section}
    \subjclass{53C07, 58E15, 14D21, 81T13.}
\end{abstract}

\tableofcontents

\section{Introduction}
Consider a Riemann surface $X$ with its canonical line bundle $K$. An $\mathrm{SL}(n,\mathbb{C})$-Higgs bundle on $X$ is a pair $\qty(E,\theta)$ consisting of a rank $n$ holomorphic vector bundle $E$ of trivial determinant and a Higgs field $\theta$ which is a trace-free holomorphic section of $\operatorname{End}\qty(E) \otimes K$. Given a Hermitian metric $h$ on $E$, we can obtain the Chern connection $\nabla_h$ associated with $h$. The metric $h$ is called a harmonic metric of the Higgs bundle $\qty(E,\theta)$ if it satisfies the following equation:
\begin{align*}
F\qty(h) + \comm{\theta}{\theta^{*h}} = 0,
\end{align*}
here $F\qty(h)$ is the curvature of Chern connection $\nabla_h$ and $\theta^{*h}$ is the adjoint of $\theta$ with respect to $h$. It is worth noting that this equation is equivalent to the condition that the connection $\mathbb{D}_h=\nabla_h + \theta + \theta^{*h}$ is flat.

This equation, introduced by Hitchin \cite{hitchin1987self}, is now called the Hitchin equation. For a compact Riemann surface $X$, the celebrated results of Hitchin \cite{hitchin1987self} and Simpson \cite{simpson1988constructing} show that a Higgs bundle $\qty(E,\theta)$ admits a harmonic metric $h$ such that $\det \qty(h)=1$ if and only if $\qty(E,\theta)$ is polystable. Moreover, when $\qty(E,\theta)$ is stable, such a harmonic metric is unique. Combining with the work of Donaldson \cite{donaldson1987twisted} and Corlette \cite{corlette1988flat}, we obtain the non-Abelian Hodge correspondence, which is a homeomorphism between the moduli space of polystable $\mathrm{SL}(n,\mathbb{C})$-Higgs bundles and the character variety of reductive representations from the fundamental group of $X$ to $\mathrm{SL}(n,\mathbb{C})$. For punctured Riemann surfaces, the theory of harmonic metrics was initiated by Simpson \cite{simpson1990harmonic}, who treated the tame case. The wild case was later developed by Biquard–Boalch in \cite{biquard2004wild}, and Mochizuki \cite{mochizuki2010wild, Mochizuki2021good} extended the theory to smooth quasi-projective varieties.

For a real reductive Lie group $G$, we also have the notion of $G$-Higgs bundles (see Section~\ref{Preliminaries on SO Higgs bundles}). An important class of $G$-Higgs bundles was constructed by Hitchin in \cite{hitchin1992lie} for a split real form $G$ of a complex simple Lie group $G^{\mathbb{C}}$. For such a $G$, Hitchin used the canonical line bundle $K$ and a tuple of holomorphic differentials $\symbfit{q}=(q_1,\dots,q_r)$, where $r$ is the rank of $G$, to define a $G$-Higgs bundle $\qty(\mathbb{K}_{G},\theta\qty(\symbfit{q}))$. When $X$ is compact and hyperbolic, this bundle satisfies the following fundamental properties:
\begin{itemize}
\item The $G$-Higgs bundle $\qty(\mathbb{K}_{G},\theta\qty(\symbfit{q}))$ is stable in some sense and admits a unique harmonic metric $h$ which is compatible with some extra structures,
\item the holonomy group of the flat connection $\mathbb{D}_{h}=\nabla_h + \theta + \theta^{*h}$ is in $G$, i.e., the flat connection $\mathbb{D}_{h}$ corresponds to a reductive representation $\rho$ from the fundamental group of $X$ to $G$.
\end{itemize}

Two key concepts associated with this class of Higgs bundles are the Hitchin fibration and the Hitchin section introduced by Hitchin \cite{hitchin1992lie}.
Let $r=\mathrm{rank}(G)$ and let $p_1,\dots, p_r$ be a basis of the ring of $G^{\mathbb{C}}$-invariant homogeneous polynomials. Their degrees are $d_1,\dots,d_r$ respectively, which depend only on $G$. Denote by $M(G^{\mathbb{C}})$ the moduli space of $G^{\mathbb{C}}$-Higgs bundles. The Hitchin fibration is the map
\begin{align*}
    h \colon M(G^{\mathbb{C}}) \rightarrow B(G^{\mathbb{C}}) \quad  (E, \theta) \to (p_1(\theta),\dots, p_r(\theta)), 
\end{align*}
where $B(G^{\mathbb{C}})=\bigoplus_{i=1}^r H^0\qty(X,K^{d_i})$ is called the Hitchin base.
The Hitchin section of this fibration is defined by 
\begin{align*}
    s \colon B(G^{\mathbb{C}}) \rightarrow  M(G^{\mathbb{C}}) \quad  \symbfit{q}=(q_1,\dots,q_r) \to  \qty(\mathbb{K}_{G},\theta\qty(\symbfit{q})),
\end{align*}
 where we regard a $G$-Higgs bundle naturally as a $G^{\mathbb{C}}$-Higgs bundle (see Remark~\ref{remark about G Higgs bundles}). 
Consequently, the Higgs bundle $(\mathbb{K}_{G},\theta\qty(\symbfit{q}))$ is called the $G$-Higgs bundle in the Hitchin section. 
We now recall the notion of Hitchin representations. Let $X$ be a compact hyperbolic Riemann surface. Let $\iota$ be the unique (up to conjugation) irreducible representation from $\mathrm{PSL}(2,\mathbb{R})$ to $G$. Recall that a representation $\rho \colon \pi_1(X) \to \mathrm{PSL}(2,\mathbb{R})$ is called Fuchsian if it is discrete and faithful. We also refer to the composition $\iota \circ \rho $ as a Fuchsian representation into $G$. Hitchin representations are defined as the continuous deformations of such Fuchsian representations. The set of conjugacy classes of Hitchin representations forms a connected component of the character variety known as the Hitchin component. Under the non-Abelian Hodge correspondence, $G$-Higgs bundles in the Hitchin section correspond to Hitchin representations.  

Now, we give an explicit construction of $G$-Higgs bundles in the Hitchin section for $G={\mathrm{SL}\qty(n,\mathbb{R})}$. Let $\symbfit{q} = \qty(q_2,\dots,q_n)$, where $q_j$ is a holomorphic $j$-differential on $X$. Recall that $K$ is the canonical line bundle of $X$. Choose a square root $K^{\frac{1}{2}}$ of $K$. The multiplication of $q_j$ induces the following holomorphic bundle maps:
$
K^{\nicefrac{\qty(n-2i+1)}{2}} \rightarrow K^{\nicefrac{\qty(n-2i+2(j-1)+1)}{2}} \otimes K  \, (j \leq i \leq n)
$.
We also have the identity map for $1 \leq i \leq n-1$:
$
K^{\nicefrac{\qty(n-2i+1)}{2}} \rightarrow K^{\nicefrac{\qty(n-2(i+1)+1)}{2}} \otimes K
$.
They define a Higgs field $\theta\qty(\symbfit{q})$ of $\mathbb{K}_{\mathrm{SL}\qty(n,\mathbb{R})} =  \mathop{\bigoplus} \limits_{i=1}^{n} K^{\nicefrac{\qty(n-2i+1)}{2}}$ by 
\begin{align} \label{Higgs field for sl}
\theta\qty(\symbfit{q})=
\begin{pNiceMatrix}
    0 & q_2 & \Cdots & q_n \\
    1 & \Ddots &\Ddots &\Vdots   \\
      & \Ddots & &q_2  \\
      & & 1&  0
\end{pNiceMatrix}.
\end{align}
The natural pairings
$
K^{\nicefrac{\qty(n+1-2i)}{2}} \otimes  K^{-\nicefrac{\qty(n+1-2i)}{2}} \rightarrow \mathcal{O}_{X} \, (1 \leq i \leq n)
$
induce a non-degenerate symmetric bilinear form $C_{\mathbb{K}_{\mathrm{SL}\qty(n,\mathbb{R})}}$ of $\mathbb{K}_{\mathrm{SL}\qty(n,\mathbb{R})}$. 

In particular, when $X$ is compact and $n=2$, these Higgs bundles parametrize
the Teichm\"uller space (see \cite{hitchin1987self}). The study of  harmonic metrics of $\qty(\mathbb{K}_{\mathrm{SL}\qty(n,\mathbb{R})},\theta\qty(\symbfit{q}))$ in the non-compact hyperbolic Riemann surfaces case is interesting. For the case $X= \overline{X}-D$ where $\overline{X}$ is a compact Riemann surface and $D$ is a finite set of points in $X$, let $q_j \, (j=2,\dots,n)$ be meromorphic differentials on $\overline{X}$ with possible poles at $D$ of pole order at most $j-1$. When $X$ is hyperbolic, using the work of Simpson \cite{simpson1990harmonic} on parabolic Higgs bundles and tame harmonic bundles, Biswas-Ar{\'e}s-Gastesi-Govindarajan \cite{biswas1997parabolic} showed that $\qty(\mathbb{K}_{\mathrm{SL}\qty(n,\mathbb{R})},\theta\qty(\symbfit{q}))$ can be regarded as a stable parabolic Higgs bundle of parabolic degree $0$ on $X$ with some chosen weights and thus admits a unique tame harmonic metric.
Recently, Li and Mochizuki studied the harmonic metrics of $\qty(\mathbb{K}_{\mathrm{SL}\qty(n,\mathbb{R})},\theta\qty(\symbfit{q}))$ for any non-compact hyperbolic Riemann surface in \cite{li2023higgs}. They obtained the following important theorem. 
\begin{theorem}\cite[Theorem~6.6]{li2023higgs}\label{existence and uniqueness theorem for sl}
Let $X$ be a non-compact hyperbolic Riemann surface and $g_X$ be the unique complete K\"ahler metric with Gau{\ss} curvature $-1$. Then there exists a harmonic metric $h$ of $(\mathbb{K}_{\mathrm{SL}\qty(n,\mathbb{R})}, \theta (\symbfit{q}))$ for any $\symbfit{q}$ such that 
\begin{itemize}
    \item $h$ is compatible with $C_{\mathbb{K}_{\mathrm{SL}\qty(n,\mathbb{R})}}$,
    \item $h$ weakly dominates $h_X$, where $h_X$ is the natural diagonal harmonic metric of $(\mathbb{K}_{\mathrm{SL}\qty(n,\mathbb{R})}, \theta (\symbf{0}))$ induced by $g_X$.
\end{itemize}
Moreover, when $\symbfit{q}$ are all bounded with respect to $g_X$, the harmonic metric which satisfies the above two conditions is unique.
\end{theorem}
 The compatibility condition in this theorem will be elaborated in Section~\ref{Preliminaries on compatibility of symmetric pairings and Hermitian metrics}. Let us describe the metric $h_X$ and explain the weak domination property here. Define $F_k=\mathop{\bigoplus} \limits_{i=1}^{k} K^{\nicefrac{\qty(n-2i+1)}{2}}$ for $1 \leq k \leq n$. They induce a filtration of $\mathbb{K}_{\mathrm{SL}\qty(n,\mathbb{R})}$. Let $h_X=\oplus_{i=1}^{n}a_{i,\,n}g_X^{\frac{n-2i+1}{2}}$, where $a_{i,\,n}$ are some constants. We can choose these constants such that $h_X$ is a harmonic metric of $(\mathbb{K}_{\mathrm{SL}\qty(n,\mathbb{R})}, \theta (\symbf{0}))$ and compatible with $C_{\mathbb{K}_{\mathrm{SL}\qty(n,\mathbb{R})}}$. We call $h$ weakly dominates $h_X$ if $\det (h|_{F_k}) \leq \det (h_X|_{F_k})$ for $1 \leq k \leq n$. Note that the definition of weak domination property depends on the filtration.

 They also proved the above theorem for $G= \mathrm{SO}_0\qty(n,n-1)$ and $G=\mathrm{Sp}\qty(2n,\mathbb{R})$. Indeed, for $G=\mathrm{SO}_0(n,n-1)$ or $\mathrm{SP}(2n,\mathbb{R})$, the $G$-Higgs bundles in the Hitchin section can be naturally identified with $\mathrm{SL}(m,\mathbb{R})$-Higgs bundles in the Hitchin section (e.g. one may refer \cite[Section 6]{li2023higgs}).
 In this paper, we study the remaining classical case, namely, $G=\mathrm{SO}_0(n,n)$. For the definition of $\mathrm{SO}_0(n,n)$-Higgs bundles and the precise construction of $(\mathbb{K}_{\mathrm{SO}_0(n,n)}, \theta (\symbfit{q}))$, we refer to Section~\ref{Preliminaries on SO Higgs bundles}. We note here that $\symbfit{q} = (q_{1},\cdots, q_{n-1}, q_{n}) \in \mathop{\bigoplus} \limits_{i=1}^{n-1} H^{0}(X,K^{2i}) \bigoplus H^{0}(X,K^{n}) $. Recall
 the Hitchin base for $G= \mathrm{SO}_0\qty(n,n-1)$ is 
 $\mathop{\bigoplus} \limits_{i=1}^{n-1} H^{0}(X,K^{2i})$ (see Section~\ref{Preliminaries on SO Higgs bundles}). The extra factor $H^{0}(X,K^{n})$ appearing in the Hitchin base for $\mathrm{SO}_0(n,n)$ corresponds to the Pfaffian polynomial of $\mathfrak{so}(2n,\mathbb{C})$, which is a distinguished invariant polynomial that exists only for $\mathfrak{so}(2n,\mathbb{C})$ and not for other classical complex simple Lie algebras.

 The proof of Theorem~\ref{existence and uniqueness theorem for sl} relies on the particular form of the Higgs field. More precisely, the elements on the sub-diagonal of the Higgs field must not be zero and the elements below the sub-diagonal of the Higgs field must be zero, just like~\eqref{Higgs field for sl}. Li and Mochizuki needed this particular form to obtain the weak domination property and $C^0$ estimates for harmonic metrics and eventually proved the existence of the harmonic metric. Combining with the compatibility condition, they proved the uniqueness of the harmonic metric in the case of bounded differentials. However, the Higgs fields in our case are not in this particular form (see~\eqref{natural Higgs field}). Fortunately, we can use the
 compatibility condition to overcome the difficulties and obtain the above theorem in this case. 
 \begin{theorem}[Theorem~\ref{existence and uniqueness of harmonic metric}] \label{existence and uniqueness of harmonic metric in introduction}
Let $X$ be a non-compact hyperbolic Riemann surface and $g_X$ be the unique complete K\"ahler metric with Gau{\ss} curvature $-1$. Then there exists a harmonic metric $h$ of $(\mathbb{K}_{\mathrm{SO}_0(n,n)}, \theta (\symbfit{q}))$ for any $\symbfit{q}$ such that 
\begin{itemize}
     \item $h$ is compatible with the $\mathrm{SO}_0(n,n)$-structure,
    \item $h$ weakly dominates $h_X$, where $h_X$ is the natural diagonal harmonic metric of $(\mathbb{K}_{\mathrm{SO}_0(n,n)}, \theta (\symbf{0}))$ induced by $g_X$.
\end{itemize}
Moreover, when $\symbfit{q}$ are all bounded with respect to $g_X$, the harmonic metric which satisfies the above two conditions is unique.
\end{theorem}
\begin{remark}
   \begin{itemize}
    \item See Section~\ref{Set-up Domination property} for the detailed description of $h_X$ and the filtration used in the definition of weak domination property. 
    \item The weak domination property gives an estimate for the Higgs field, i.e.,
    $
        \abs{\theta\qty(\symbfit{q})}^2_{h,g_X} \geq \abs{\theta\qty(\symbf{0})}^2_{h,g_X} = \frac{n(n-1)(2n-1)}{3}.
    $ For the details, see Theorem~\ref{dominate Higgs field}.
    \item   In \cite{fujioka2024harmonic}, Hitoshi Fujioka studied the existence of harmonic metrics of $\mathrm{SL}\qty(3,\mathbb{R})$-Higgs bundles in the Hitchin section on flat non-compact Riemann surfaces, i.e., $\mathbb{C}$ or $\mathbb{C}^*$, using spectral curves. He discovered that a harmonic metric may not exist in some cases. We also desire to obtain similar results in the $\mathrm{SO}_0(n,n)$ case.
   \end{itemize}
\end{remark}

Finally, we present two applications of our results. The first one is a reproof of the existence and uniqueness of the harmonic metric of
 $(\mathbb{K}_{\mathrm{SO}_0(n,n)}, \theta (\symbfit{q}))$ on a compact hyperbolic Riemann surface and we can get the weak domination property of this harmonic metric.
 \begin{theorem} [Theorem~\ref{existence and uniqueness in compact case}]
     Let $X$ be a compact hyperbolic Riemann surface. Then for any $\symbfit{q}$, there uniquely exists a harmonic metric $h$ of $(\mathbb{K}_{\mathrm{SO}_0(n,n)}, \theta (\symbfit{q}))$ such that 
     \begin{itemize}
    \item $h$ is compatible with the $\mathrm{SO}_0(n,n)$-structure,
    \item $h$ weakly dominates $h_X$,where $h_X$ is the natural diagonal harmonic metric of $(\mathbb{K}_{\mathrm{SO}_0(n,n)}, \theta (\symbf{0}))$ induced by $g_X$.
\end{itemize}
 \end{theorem}
 The second application is a rigidity result concerning the energy densities of Hitchin representations for  $\mathrm{PSO}_{0}\qty(n,n)$.
 Consider a closed surface $S$ with genus $g \geq 2$. A hyperbolic metric $g_0$ gives a representation from $\pi_1(S)$ to $\mathrm{PSL}(2,\mathbb{R})$. Composing with the unique irreducible representation of $\mathrm{PSL}(2,\mathbb{R})$ to $\mathrm{PSO}_0(n,n)$, we get the base $n$-Fuchsian representation $\rho_0$. The Hitchin representations are exactly the deformations of $\rho_0$. All Hitchin representations can be lifted to the representations to $\mathrm{SO}_0(n,n)$. The lifted representations are also called by Hitchin representations. 
 The conformal class of $g_0$ gives a Riemann surface structure on $S$, which is denoted by $X$. Then the $\mathrm{SO}_0(n,n)$-Higgs bundles $(\mathbb{K}_{\mathrm{SO}_0(n,n)}, \theta (\symbfit{q}))$ in the Hitchin section are one-one corresponding to the Hitchin representations $\rho$ via the non-Abelian Hodge correspondence. And the base $n$-Fuchsian representation $\rho_0$ corresponds to the Higgs bundle $(\mathbb{K}_{\mathrm{SO}_0(n,n)}, \theta (\symbf{0}))$. Note that a harmonic metric $h$ of $(\mathbb{K}_{\mathrm{SO}_0(n,n)}, \theta (\symbfit{q}))$ is equivalent to a equivariant harmonic map $f$ from the universal cover of $X$ to the symmetric space $\mathrm{SO}_{0}\qty(n,n)/\mathrm{SO}\qty(n)\times\mathrm{SO}\qty(n)$. Equip $\mathrm{SO}_{0}\qty(n,n)/\mathrm{SO}\qty(n)\times\mathrm{SO}\qty(n)$ with the Riemann metric induced by the Killing form. We have the following rigidity result about the energy density of $f$. This result is motivated by the work of Li \cite{li2019harmonic}. In fact, Li discovered the weak domination property of harmonic metrics for $\mathrm{SL(n,\mathbb{R})}$-Higgs bundles in the Hitchin section and used this property to prove the energy rigidity property. 
\begin{theorem} [Corollary~\ref{energy domination}]
    Suppose that $S$ is a closed orientable surface with genus $g\geq2$ and $g_0$ is a hyperbolic metric on $S$. Consider the universal cover $\widetilde{S}$ of $S$ and the lifted hyperbolic metric $\widetilde{g_0}$. Let $\rho$ be a Hitchin representation for $\mathrm{SO}_0(n,n)$ and $f$ be the unique $\rho$-equivariant harmonic map from $\qty(\widetilde{S},\widetilde{g_0})$ to $\mathrm{SO}_{0}\qty(n,n)/\mathrm{SO}\qty(n)\times\mathrm{SO}\qty(n)$. Then its energy density $e(f)$ satisfies 
    \begin{align*}
        e(f) \geq \frac{2n(n-1)^2(2n-1)}{3}.
    \end{align*}
    Moreover, the equality holds at a point if and only if $\rho$ is the base $n$-Fuchsian representation of $\qty(S,g_0)$.
\end{theorem}
\subsection*{Organization}
 In Section~\ref{Preliminaries}, we first recall the definition of $\mathrm{SO}_0(n,n)$-Higgs bundles and the construction of  $(\mathbb{K}_{\mathrm{SO}_0(n,n)}, \theta (\symbfit{q}))$. Then we present some results on the existence of harmonic metrics using smooth exhaustion family of harmonic metrics in \cite{li2023higgs} and the compatibility condition in \cite{Li2023generically}.
 In Section~\ref{Domination property}, we prove the key weak domination property. We hugely use the compatibility condition. In Section~\ref{Existence of harmonic metric}, we prove the existence part of Theorem~\ref{existence and uniqueness of harmonic metric in introduction}. The key point is Proposition~\ref{estimate P pro}. In Section~\ref{Uniqueness in bounded case}, we prove the uniqueness part of Theorem~\ref{existence and uniqueness of harmonic metric in introduction}. In Section~\ref{Applications}, we reprove the existence and uniqueness theorem of the harmonic metric in compact hyperbolic Riemann surface case which does not use the stability condition and gives the energy density domination. Using the Maximum principle, we obtain the rigidity result.
 \subsection*{Acknowledgment}
  I wish to express my sincere gratitude to my Ph.D. supervisor, Prof. Qiongling Li, for suggesting this problem and for many stimulating discussions. I am also grateful to Junming Zhang for helpful comments. I thank the reviewer for their careful reading and insightful suggestions, which have greatly improved the quality of this paper. This work was partially supported by the National Key R\&D Program of China No.2022YFA1006600 and the Fundamental Research Funds for the Central Universities (No.
63243067) and Nankai Zhide Foundation. 

\section{Preliminaries}\label{Preliminaries}
\subsection{Preliminaries on \texorpdfstring{$\mathrm{SO}_0(n,n)$}--Higgs bundles}
\label{Preliminaries on SO Higgs bundles}

Let $X$ be a Riemann surface and $K$ be its canonical line bundle.

\begin{definition} \label{SL Higgs bundles definition}
    An $\mathrm{SL}(n,\mathbb{C})$-Higgs bundle over $X$ is a pair $(E,\theta)$, where $E$ is a rank $n$ holomorphic vector bundle of trivial determinant and $\theta$ is a trace-free holomorphic endomorphism of $E$ twisted by $K$. 
\end{definition}
Let $G$ be a real reductive Lie group with Lie algebra $\mathfrak{g}$. We briefly introduce the notion of a $G$-Higgs bundle. See \cite{garciaprada2012hitchinkobayashi} and \cite{biquard2020parabolic} for more details.
Let $H \subset G$ be a maximal compact subgroup with Lie algebra $\mathfrak{h}$. The corresponding Cartan decomposition is $\mathfrak{g}=\mathfrak{h} \oplus \mathfrak{m}$, where 
\begin{align*}
    \comm{\mathfrak{h}}{\mathfrak{h}} \subset \mathfrak{h}, \quad \comm{\mathfrak{h}}{\mathfrak{m}} \subset \mathfrak{m}, \quad \comm{\mathfrak{m}}{\mathfrak{m}} \subset \mathfrak{h}.
\end{align*}
Hence the compact group $H$ acts on $\mathfrak{m}$ by the adjoint representation. We obtain the isotropy representation $\iota \colon H^{\mathbb{C}} \to \mathrm{GL}(\mathfrak{m}^{\mathbb{C}})$ by complexification. For a holomorphic $H^{\mathbb{C}}$-principal bundle $P$, let $P(\mathfrak{m^{\mathbb{C}}})=P \times_{H^{\mathbb{C}}} \mathfrak{m}^{\mathbb{C}}$ be the associated $\mathfrak{m}^{\mathbb{C}}$-bundle via the isotropy representation $\iota$.
\begin{definition} \label{G higgs bundles definition}
    A $G$-Higgs bundle on $X$ is a pair $(P,\theta)$, where $P$ is a holomorphic $H^{\mathbb{C}}$-principal bundle and $\theta$ is a holomorphic section of $P(\mathfrak{m^{\mathbb{C}}}) \otimes K$.
\end{definition}
\begin{remark} \label{remark about G Higgs bundles}
    \begin{itemize}
        \item  When $G$ is complex, let $H$ be a maximal compact group. Then $H^{\mathbb{C}}=G$ and the Cartan decomposition of the Lie algebra of $G$ is $\mathfrak{g} = \mathfrak{h} \oplus i \mathfrak{h}$, where $\mathfrak{h}$ is the Lie algebra of $H$. Hence a $G$-Higgs bundle $(P,\theta)$ consists of a holomorphic $G$-principle bundle $P$ and a holomorphic section $\theta \in P(\mathfrak{g}) \otimes K$. In particular, Definition~\ref{G higgs bundles definition} for $G=\mathrm{SL}(n,\mathbb{C})$ coincides with Definition~\ref{SL Higgs bundles definition}.
        \item For a real form $G$ of a complex reductive Lie group $G^{\mathbb{C}}$, a $G$-Higgs bundle can be naturally regarded as a $G^{\mathbb{C}}$-Higgs bundle.
    \end{itemize}
\end{remark}

For $G=\mathrm{SO}_0(n_1,n_2)$, the maximal compact subgroup $H$ is $\mathrm{SO}(n_1) \times \mathrm{SO}(n_2)$, whose complexification is $H^{\mathbb{C}}=\mathrm{SO}(n_1,\mathbb{C}) \times \mathrm{SO}(n_2,\mathbb{C})$. The Lie algebra $\mathfrak{g}$ of $G$ is
\begin{align*}
    \mathfrak{g}=\qty{
        \begin{pmatrix*}
            A & B \\
            C & D \\
        \end{pmatrix*} \middle| 
        \begin{array}{l}
             A \in \mathbb{R}^{n_1 \times n_1},\, D \in \mathbb{R}^{n_2 \times n_2},\, B \in \mathbb{R}^{n_1 \times n_2},\,C \in \mathbb{R}^{n_2 \times n_1},  \\
        \trans{A} + A =0, \,\trans{D}+D= 0 ,\, C=\trans{B}
        \end{array}
    }.
\end{align*}
Here  $\mathbb{R}^{m \times n}$ denotes the set of $m \times n$ matrices with entries in $\mathbb{R}$.
The Cartan decomposition $\mathfrak{g}=\mathfrak{h} \oplus \mathfrak{m}$ is given by
\begin{align*}
    \mathfrak{h} & =\qty{
 \begin{pmatrix*}
            A &  \\
             & D \\
        \end{pmatrix*} \middle| 
        \begin{array}{l}
             A \in \mathbb{R}^{n_1 \times n_1},\, D \in \mathbb{R}^{n_2 \times n_2},  \\
        \trans{A} + A =0, \,\trans{D}+D= 0
        \end{array}
    }, \\
  \mathfrak{m} & = \qty{
   \begin{pmatrix*}
             & B \\
            \trans{B} &  \\
        \end{pmatrix*} \middle|  B \in \mathbb{R}^{n_1 \times n_2} 
  }.
\end{align*}
Therefore, an $\mathrm{SO}_0(n_1,n_2)$-Higgs bundles can be described concretely in terms of vector bundles as follows:
An $\mathrm{SO}_0(n_1,n_2)$-Higgs bundle over $X$ is equivalent to a triple $(E,Q_E,\theta)$ of the form
\begin{align*}
\qty(E,Q_E,\theta) = \qty(V \oplus W,\, Q_V \oplus -Q_W,\,
\begin{pmatrix}
0 & \eta \\
\eta^{\dagger} & 0
\end{pmatrix}
),
\end{align*}
 where 
 \begin{itemize}
    \item $V$ and $W$ are holomorphic vector bundles of rank $n_1$ and $n_2$, respectively, each with trivial determinant,
    \item $Q_V,Q_W$ are non-degenerate holomorphic symmetric pairings on $V$ and $W$ respectively,
    \item $\eta \colon W \to V \otimes K$ is a holomorphic bundle morphism,
    \item $\eta ^{\dagger}$ is the adjoint of $\eta$ with respect to $Q_{V},Q_{W}$.
 \end{itemize}
\begin{remark}
    \begin{itemize}
        \item Note that $(V,Q_V)$ and $(W,Q_W)$ are equivalent, respectively, to holomorphic principal $\mathrm{SO}(n_1,\mathbb{C})$- and $\mathrm{SO}(n_2,\mathbb{C})$-bundles. Moreover, $(E,Q_E)$ corresponds to a holomorphic principal $\mathrm{SO}(n_1,\mathbb{C})\times \mathrm{SO}(n_2,\mathbb{C})$-bundle.
        \item The Higgs field $\theta$ is anti-symmetric with respect to 
        $
        \left(
        \begin{smallmatrix}
            Q_{V} & \\
                & -Q_{W}
        \end{smallmatrix}
        \right)
        $ and symmetric with respect to 
        $
        \left(
        \begin{smallmatrix}
            Q_{V} & \\
                & Q_{W}
        \end{smallmatrix}
        \right).
        $
        \item We make no distinction between the $\mathrm{SO}_0(n_1,n_2)$-Higgs bundle $\qty(E,Q_E,\theta)$ and the triple $\qty(\qty(V,Q_{V}),\qty(W,Q_{W}),\eta)$. For simplicity, we usually omit $Q_E$ and abbreviate $\qty(E,Q_E,\theta)$ as $\qty(E,\theta)$.
    \end{itemize}
\end{remark}
Fix a positive integer $n$, choose $\symbfit{q} = (q_{1},\dots, q_{n-1}, q_{n}) \in \mathop{\bigoplus} \limits_{i=1}^{n-1} H^{0}(X,K^{2i}) \bigoplus H^{0}(X,K^{n}) $, i.e., $q_{i}$ is a holomorphic $2i$-differential for $1\leq i \leq n-1$ and $q_{n}$ is a holomorphic $n$-differential. One can use these data to construct the triple $\qty(\qty(V,Q_{V}),\qty(W,Q_{W}),\eta \qty(\symbfit{q}))$ :

\begin{align}\label{Higgs field}
\begin{aligned} 
(V,Q_{V})=&
\qty( K^{n-1} \oplus K^{n-3} \oplus \dots \oplus K^{3-n} \oplus K^{1-n},\,
\begin{pmatrix} 
&         & 1  \\
& \iddots &  \\
1 &       &    \\
\end{pmatrix} 
), \\
(W,Q_{W})=&
\qty( K^{n-2} \oplus K^{n-4} \oplus \dots \oplus K^{4-n}\oplus K^{2-n} \oplus \mathcal{O}',\,
\begin{pmatrix}
&         & 1  &  \\
& \iddots &     & \\
1 &       &      &  \\
&        &     & 1 \\
\end{pmatrix} 
),   \\ 
\eta (\symbfit{q})= &
\begin{pmatrix}
q_1 & q_2 & \dots & q_{n-1} & q_n \\
1  & q_1 & \dots & q_{n-2} & 0    \\
& \ddots & \ddots & \vdots  & \vdots  \\
&        & 1      &  q_1    & 0      \\
&  &              &  1      & 0
\end{pmatrix} \colon 
W \longrightarrow V \otimes K.
\end{aligned}
\end{align}
Here $\mathcal{O}'$ is the trivial line bundle over $X$. We use this symbol to emphasize the difference between two trivial line bundles which another trivial line bundle $\mathcal{O}$ may appear in $V$ or $W$ depending on the parity of $n$.
\begin{definition} \label{definition of Higgs bundles in the Hitchin section}
For the triple $\qty(\qty(V,Q_{V}),\qty(W,Q_{W}),\eta \qty(\symbfit{q}))$, 
the corresponding $\mathrm{SO}_0(n,n)$-Higgs bundle is denoted by $(\mathbb{K}_{\mathrm{SO}_0(n,n)}, \theta (\symbfit{q}))$. This Higgs bundle is referred to as the $\mathrm{SO}_0(n,n)$-Higgs bundle in the Hitchin section. 
\end{definition} 

Concretely, the vector bundle $\mathbb{K}_{\mathrm{SO}_0(n,n)}$ is a direct sum 
\[\mathbb{K}_{\mathrm{SO}_0(n,n)}=K^{n-1} \oplus K^{n-3} \oplus \dots \oplus K^{3-n} \oplus K^{1-n} \oplus  K^{n-2} \oplus K^{n-4} \oplus \dots \oplus K^{4-n}\oplus K^{2-n} \oplus \mathcal{O}'.
\] 
After a suitable rearrangement of summands, it can be written as
\[
K^{n-1}\oplus K^{n-2} \oplus \dots K \oplus \mathcal{O} \oplus K^{-1} \oplus \dots \oplus K^{1-n}\oplus \mathcal{O}'.
\]
The Higgs bundle $(\mathbb{K}_{\mathrm{SO}_0(n,n)}, \theta (\symbf{0}))$ admits a diagrammatic description:
\begin{align} 
    \qty(K^{n-1} \xrightarrow{1} K^{n-2} \xrightarrow{1} \dots \xrightarrow{1} \dots \xrightarrow{1}
K \xrightarrow{1} \mathcal{O} \xrightarrow{1} K^{-1} \xrightarrow{1} \dots \xrightarrow{1} K^{2-n} \xrightarrow{1} K^{1-n} ) \oplus \qty(\mathcal{O}',0).
\end{align}.
\begin{remark}
    \begin{itemize}
    \item  For the explicit construction of $\mathrm{SO}_0(n,n)$-Higgs bundles in the Hitchin section,  we refer \cite[Chapter~6.4]{Garcia-Prada_2009} and \cite[Chapter~8.4]{Arroyo2009TheGO}.
    \end{itemize}
\end{remark}

Next, we introduce $\mathrm{SO}_0(n,n-1)$-Higgs bundles in the Hitchin section and explain the connections and differences between these two types of Higgs bundles: $\mathrm{SO}_0(n,n)$- and $\mathrm{SO}_0(n,n-1)$-Higgs bundles in the Hitchin section.

For $\symbfit{q}' = (q_{1},\dots, q_{n-1}) \in \mathop{\bigoplus} \limits_{i=1}^{n-1} H^{0}(X,K^{2i})$, the corresponding $\mathrm{SO}_0(n,n-1)$-Higgs bundle in the Hitchin section $(\mathbb{K}_{\mathrm{SO}_0(n,n-1)}, \theta (\symbfit{q}'))$ is given by the triple:
 \begin{align} \label{SO(n,n-1) Higgs field}
\begin{aligned} 
(V,Q_{V})=&
\qty( K^{n-1} \oplus K^{n-3} \oplus \dots \oplus K^{3-n} \oplus K^{1-n},\,
\begin{pmatrix} 
&         & 1  \\
& \iddots &  \\
1 &       &    \\
\end{pmatrix} 
), \\
(W,Q_{W})=&
\qty( K^{n-2} \oplus K^{n-4} \oplus \dots \oplus K^{4-n}\oplus K^{2-n},\,
\begin{pmatrix}
&         & 1    \\
& \iddots &      \\
1 &       &        \\
\end{pmatrix} 
),   \\ 
\eta (\symbfit{q}')= &
\begin{pmatrix}
q_1 & q_2 & \dots & q_{n-1} \\
1  & q_1 & \dots & q_{n-2} \\
& \ddots & \ddots & \vdots \\
&        & 1      &  q_1   \\     
&  &              &  1    \\  
\end{pmatrix} \colon 
W \longrightarrow V \otimes K.
\end{aligned}
\end{align}
Note that $\mathbb{K}_{\mathrm{SO}_0(n,n)}= \mathbb{K}_{\mathrm{SO}_0(n,n-1)} \oplus \mathcal{O}'$. In a diagrammatic form, the Higgs bundle $(\mathbb{K}_{\mathrm{SO}_0(n,n-1)},\theta (\symbf{0}))$ is represented by
\begin{align} 
    K^{n-1} \xrightarrow{1} K^{n-2} \xrightarrow{1} \dots \xrightarrow{1} \dots \xrightarrow{1}
K \xrightarrow{1} \mathcal{O} \xrightarrow{1} K^{-1} \xrightarrow{1} \dots \xrightarrow{1} K^{2-n} \xrightarrow{1} K^{1-n}.
\end{align}
The bundle $\mathbb{K}_{\mathrm{SO}_0(n,n-1)}$ carries a natural holomorphic filtration $ \symbfit{F}=\lbrace 0 =F_0 \subset F_1 \subset \dots \subset F_{2n-1}=\mathbb{K}_{\mathrm{SO}_0(n,n-1)} \rbrace $, which is defined by 
\[
  F_{k} = \mathop{\bigoplus} \limits_{i=1}^k K^{n-i}, \quad  1 \leq k \leq 2n-1.
\]
Furthermore, it follows from \eqref{SO(n,n-1) Higgs field} that the Higgs field $\theta (\symbfit{q}')$ sends $F_k$ to $F_{k+1} \otimes K$ and induces the following isomorphism
\[
 \theta (\symbfit{q}') \colon F_{k}/F_{k-1} \xrightarrow{1} F_{k+1}/F_{k} \otimes K 
\]
for $1 \leq k \leq 2n-2$. In other words, the Higgs field induces an isomorphism on every graded quotient of this filtration.

For $G= \mathrm{SL}(n,\mathbb{R})$ or $\mathrm{SP}(2n,\mathbb{R})$, the corresponding $G$-Higgs bundle in the Hitchin section also admits a holomorphic filtration 
\[
\symbfit{F}=\lbrace 0 =F_0 \subset F_1 \subset \dots \subset F_{m}=\mathbb{K}_{G} \rbrace ,
\]
where $m=n$ for $G= \mathrm{SL}(n,\mathbb{R})$ and $m=2n$ for $G=\mathrm{SP}(2n,\mathbb{R})$. Moreover, the Higgs field yields isomorphisms $ F_{k}/F_{k-1} \cong F_{k+1}/F_{k} \otimes K$ for $1 \leq k \leq m-1$ (see \cite[Section 6]{li2023higgs}). 

This property is crucial for the proof of Theorem~\ref{existence and uniqueness theorem for sl} by Li and Mochizuki (see \cite[Theorem~6.6]{li2023higgs}), but it fails for $G=\mathrm{SO}_0(n,n)$. In fact, if we define the holomorphic filtration $\symbfit{F}=\lbrace 0 =F_0 \subset F_1 \subset \dots \subset F_{2n}=\mathbb{K}_{\mathrm{SO}_0(n,n)} \rbrace $ of $\mathbb{K}_{\mathrm{SO}_0(n,n)}$ as 
\begin{align*} 
F_k=
    \begin{cases}
        \mathop{\bigoplus} \limits_{i=1}^k K^{n-i},   &1\leq k \leq 2n-1,       \\
         F_{2n-1} \oplus \mathcal{O}',   & k=2n,    \end{cases}
\end{align*}
then the Higgs field $\theta (\symbfit{q})$ induces isomorphisms $F_{k}/F_{k-1} \cong F_{k+1}/F_{k} \otimes K$ for $1 \leq k \leq 2n-2$ but gives the \textbf{zero} morphism on the last graded quotient $F_{2n-1}/F_{2n-2} \to F_{2n}/F_{2n-1} \otimes K$ by \eqref{Higgs field}. 

To overcome this difficulty, we introduce a new filtration for $\mathbb{K}_{\mathrm{SO}_0(n,n)}$ (see \eqref{filtration}). Although the Higgs field $\theta (\symbfit{q})$ still does not induce isomorphisms for all graded quotients of this new filtration, we are able to establish the weak domination property (see Section~\ref{Domination property}) and $C^0$ estimates for harmonic metrics (see Section~\ref{Existence of harmonic metric}) by exploiting the compatibility condition (see Section~\ref{Preliminaries on compatibility of symmetric pairings and Hermitian metrics}), which ultimately allows us to prove the existence of harmonic metrics. The proof of uniqueness of a harmonic metric under the assumption that all holomorphic differentials are bounded also relies on the compatibility condition (see Section~\ref{Uniqueness in bounded case}).

\begin{remark}
    \begin{itemize}
        \item  From the perspective of representation theory, we may also observe a key distinction between the  $G=\mathrm{SO}_0(n,n)$-Higgs bundles in the Hitchin section and those for the other classical cases, namely $\mathrm{SL}(n,\mathbb{R}), \mathrm{SO}_0(n,n-1)$ and $\mathrm{SP}(2n,\mathbb{R})$. 
   
   Let $X$ be a compact hyperbolic Riemann surface. By the Non-Abelian Hodge correspondence, the $G$-Higgs bundles in the Hitchin section on $X$ correspond bijectively to $G$-Hitchin representations from $\pi_{1}(X)$ to $G$. There is a natural embedding $ \iota_G \colon G \hookrightarrow \mathrm{SL}(m,\mathbb{R})$, where $m=2n-1$ for $G=\mathrm{SO}_0(n,n-1)$ and $m=2n$ for $G=\mathrm{SO}_0(n,n),\mathrm{SP}(2n,\mathbb{R})$. 

   Let $\rho
   $ be a $G$-Hitchin representation. If $G=\mathrm{SO}_0(n,n-1)$ or $\mathrm{SP}(2n,\mathbb{R})$, then the composition $\iota_G \circ \rho$ is also an $\mathrm{SL}(m,\mathbb{R})$-Hitchin representation (e.g. one may refer \cite[Section 9]{PozzettiSambarinoWienhard+2021+1+51}). However, this fails for $G=\mathrm{SO}_0(n,n)$. We briefly explain why.
   
   Consider the natural inclusion $\iota_0 \colon \mathrm{SO}_0(n,n-1) \hookrightarrow \mathrm{SO}_0(n,n)$. Let $\rho_0$ be the $\mathrm{SO}_0(n,n-1)$-Fuchsian representation. Then $\rho_0'=\iota_0 \circ \rho_0$ is the corresponding $\mathrm{SO}_0(n,n)$-Fuchsian representation. However, the composition $\iota_{\mathrm{SO}_0(n,n)} \circ \iota_0 \colon \mathrm{SO}_0(n,n-1) \to \mathrm{SL}(2n,\mathbb{R})$ is reducible. Consequently, $\iota_{\mathrm{SO}_0(n,n)} \circ \rho_0'$ is not $\mathrm{SL}(2n,\mathbb{R})$-Hitchin, as it fails to be irreducible. Recall that $G$-Hitchin representations are precisely the continuous deformations of the $G$-Fuchsian representation. Hence, under the induced map of the character varieties 
   \begin{align*}
    \mathfrak{X}(\iota_{\mathrm{SO}_0(n,n)}) :  \mathfrak{X}(\pi_1(X),\mathrm{SO}_0(n,n) ) \to \mathfrak{X}(\pi_1(X),\mathrm{SL}(2n,\mathbb{R})), 
   \end{align*}
    the $\mathrm{SO}_0(n,n)$-Hitchin component is sent into a component in the character variety $\mathfrak{X}(\pi_1(X),\mathrm{SL}(2n,\mathbb{R}))$ that lies entirely outside of the $\mathrm{SL}(2n,\mathbb{R})$-Hitchin component.
   \item   Indeed, for $G=\mathrm{SO}_0(n,n-1)$ or $\mathrm{SP}(2n,\mathbb{R})$, the $G$-Higgs bundles in the Hitchin section can be naturally identified with $\mathrm{SL}(m,\mathbb{R})$-Higgs bundles in the Hitchin section (e.g. one may refer \cite[Section 6]{li2023higgs}). This identification, however, fails for $G=\mathrm{SO}_0(n,n)$ due to the presence of the extra line bundle $\mathcal{O}'$ and the $n$-holomorphic differential $q_n$.
    \end{itemize}
\end{remark}

\subsection{Preliminaries on existence of harmonic metrics}

This subsection collects several existence results for harmonic metrics, largely following \cite{li2023higgs}.

Let $X$ be a Riemann surface and $\qty(E,\theta)$ be a $\mathrm{SL}(n,\mathbb{C})$-Higgs bundle on $X$. Recall that a Hermitian metric $h$ on $E$ is called a harmonic metric of the Higgs bundle $\qty(E,\theta)$ if it satisfies the Hitchin equation:
\begin{align*}
F\qty(h) + \comm{\theta}{\theta^{*h}} = 0,
\end{align*}
in which $F\qty(h)$ is the Chern curvature of $h$ and $\theta^{*h}$ is the adjoint of $\theta$ with respect to $h$. 

Let $Y \subset X$ be a relatively compact connected open subset with non-empty smooth boundary $\partial Y$. A Hermitian metric $h$ on $E|_{Y}$ is  harmonic if it satisfies the Hitchin equation in the interior of $Y$. Fix an arbitrary Hermitian metric $h_{\partial}$ on $E|_{\partial Y}$. The following proposition originates from Donaldson \cite{donaldson1992boundary}. 

\begin{proposition} \label{boundary harmonic metric} \cite[Proposition 3.3]{li2023higgs}
There exists a unique harmonic metric $h$ of $(E,\theta)|_{Y}$ such that $h|_{\partial Y}=h_{\partial Y}$.
\end{proposition}

Now, suppose $X$ is open and endowed with a Hermitian metric on $E$.

\begin{definition}
A smooth exhaustive family $\qty{X_i}$ of $X$ is an increasing sequence of relatively compact open subsets $X_1 \subset X_2 \subset \dots \subset X$ such that $X = \bigcup_i X_i$ and each $\partial X_i$ is smooth.
\end{definition}

Fix such a family $\qty{X_i}$. Denote by $h_{0,i}$ the restriction $h_0|_{X_i}$. For each $i$, let $h_i$ be a harmonic metric of $\left(E,\theta\right)|_{X_i}$. Define $s_i$ as the automorphism of $E|_{X_i}$ determined by 
\[
h_i\qty(\text{-},\text{-})=h_{0,i}\qty(s_i\qty(\text{-}),\text{-}).
\]
Let $f$ be a positive function on $X$ whose restriction $f|_{X_i}$ is bounded for every $i$.
The following proposition, proved in \cite{li2023higgs}, plays a fundamental role in the proof of our existence theorem.

\begin{proposition} \label{convergent to harmonic metric} \cite[Proposition 3.6]{li2023higgs}
Assume that $\abs{s_i}_{h_{0,i}} + \abs{s_i^{-1}}_{h_{0,i}} \leq f|_{X_i}$ for every $i$. Then, there exists a subsequence $s_{i\qty(j)}$ which is convergent to an automorphism $s_{\infty}$ of $E$ on any relatively compact subset $X$ in the $C^{\infty}$ sense. As a result, we obtain a harmonic metric $h_{\infty}=h_0 \cdot s_{\infty}$ as the limit of the sequence $h_{i\qty(j)}$. Moreover, we obtain $\abs{s_{\infty}}_{h_{0}} + \abs{s_{\infty}^{-1}}_{h_{0}} \leq f$. In particular, if $f$ is bounded, $h_0$ and $h_{\infty}$ are mutually bounded.
\end{proposition}


\subsection{Preliminaries on compatibility of symmetric pairing and metrics} \label{Preliminaries on compatibility of symmetric pairings and Hermitian metrics}

Recall the notion of compatibility of a non-degenerate symmetric pairing $C$ and a Hermitian metric $h$ on a complex vector space $V$. For more details, we refer to see \cite[Section 2]{Li2023generically}.

Denote by $V^{\ast}$ the dual space of $V$. The Hermitian metric $h$ induces a $\mathbb{C}$-anti-linear isomorphism $\varPhi _{h} \colon V \cong V^{\ast}$ by $\varPhi _{h}(u)(v) = h(v,u)$. The dual Hermitian metric $h^{\ast}$ on $V^{\ast}$ is then defined by 
\[
h^{\ast}(u^{\ast}, v^{\ast}) = h(\varPhi_{h}^{-1}(v^{\ast}), \varPhi_{h}^{-1}(u^{\ast})).
\]
 Analogously, $C$ also give rise to an isomorphism $\varPhi_{C} \colon V \cong V^{\ast}$ and a non-degenerate symmetric pairing $C^{\ast}$ on $V^{\ast}$.
\begin{definition}
We say that $h$ is compatible with $C$ if $\varPhi_{C}$ is isometric with respect to $h$ and $h^{\ast}$; that is
\begin{align*}
    h^{\ast}(\varPhi_{C}(u), \varPhi_{C}(v)) = h(u, v)
\end{align*}
for all $u,v \in V$.
\end{definition}

We have the following lemma.
\begin{lemma} \cite[Lemma 2.2]{Li2023generically}
The following conditions are equivalent.
\begin{itemize}
\item The Hermitian metric $h$ is compatible with $C$.
\item $C(u,v) = \overline{C^{\ast}(\varPhi_{h}(u), \varPhi_{h}(v))}$ holds for any $u,v \in V$.
\item  $\varPhi_{C^{\ast}} \circ \varPhi_{h} = \varPhi_{h^{\ast}} \circ \varPhi_{C}$.
\end{itemize}
\end{lemma}

For convenience, we use $C$ and $H$ to represent the matrix of $C$ and $h$ with respect to a basis $\symbfit{v}$ of $V$ respectively. By the above lemma, $h$ is compatible with $C$ if and only if 
\begin{equation} \label{comp}
H \overline{C^{-1}} H^{T}= C. 
\end{equation} 

Let $X$ be a Riemann surface. Consider an $\mathrm{SO}_0(n,n)$-Higgs bundle $\qty(E,\theta)$ over $X$. The corresponding triple is $((V,Q_{V}),(W,Q_{W}),\eta)$ 
 For a Hermitian metric on $E$,we introduce the following compatibility notion.
\begin{definition}\label{SO(n,n) compatible condition}
A Hermitian metric $h$ on $E=V \oplus W$ is called compatible with the $\mathrm{SO}_0(n,n)$-structure if $h=h|_V \oplus h|_W$, where $h|_V,h|_W$ are compatible with $Q_V, Q_W$ on every complex vector space fiber respectively.
\end{definition}

Take a relatively compact connected open subset $Y \subset X$ with non-empty smooth boundary $\partial Y$. Let $h_{\partial Y}$ be a Hermitian metric on $E|_{\partial Y}$ which is compatible with the $\mathrm{SO}_0(n,n)$-structure. By Proposition~\ref{boundary harmonic metric}, there exists a unique harmonic metric $h$ of $\qty(E,\theta)$ on $Y$ satisfying the boundary condition $h|_{\partial Y}=h_{\partial Y}$. Moreover, the harmonic
metric $h$ is compatible with the $\mathrm{SO}_0(n,n)$-structure by the following lemma.
\begin{lemma}\cite[Lemma 3.15]{li2023higgs} \label{h is compatible with so}
The harmonic metric $h$ is compatible with the $\mathrm{SO}_0(n,n)$-structure. 
\end{lemma}

\section{Weak domination property}\label{Domination property}

Let $X$ be a Riemann surface and $K$ be its canonical line bundle. For each set of holomorphic differentials $\symbfit{q} = (q_{1},\dots, q_{n-1}, q_{n}) \in \mathop{\bigoplus} \limits_{i=1}^{n-1} H^{0}(X,K^{2i}) \bigoplus H^{0}(X,K^{n})$, we can construct an $\mathrm{SO}_0(n,n)$-Higgs bundle in the Hitchin section, which is denoted by $(\mathbb{K}_{\mathrm{SO}_0(n,n)}, \theta (\symbfit{q}))$. 
In this section, we will introduce the weak domination property of a Hermitian metric on $\mathbb{K}_{\mathrm{SO}_0(n,n)}$. Suppose $X$ is hyperbolic and $Y \subset X$ is a relatively compact connected open subset with smooth boundary. We prove that the unique harmonic metric of $(\mathbb{K}_{\mathrm{SO}_0(n,n)}, \theta (\symbfit{q}))|_Y$ with a particular boundary value has this property. Note that Proposition~\ref{boundary harmonic metric} ensures the existence of such a harmonic metric.
 Li discovered this property for $\mathrm{SL}(n,\mathbb{R})$-Higgs bundles in the Hitchin section in \cite{li2019harmonic}. In \cite[Proposition 4.14]{li2023higgs}, Li and Mochizuki generalized this result in a more broader set of conditions.
\subsection{Set-up}\label{Set-up Domination property}
Let $X$ be a Riemann surface and $K$ be its canonical line bundle.
Consider a $\mathrm{SL}(n,\mathbb{C})$-Higgs bundle $(E,\theta)$ over $X$ which admits a holomorphic filtration
\[
\mathbf{F}=\{0=F_0\subset F_1\subset F_2\subset\dots\subset F_m=E\}.
\]

Let $h$ be a Hermitian metric on $E$. Let $F_k(h)$ denote the induced metric of $h$ on $F_k$.
\begin{definition}
Suppose $h$ and $h_0$ are both Hermitian metrics on $E$. Call $h$ weakly dominates $h_0$ with respect to the filtration $\symbfit{F}$ if
\begin{align}
    \det(F_k(h))&\leq\det(F_k(h_0))
\end{align}
holds on $X$ for $1\leq k\leq m$.
\end{definition}

We introduce the following lemma to simplify the computation of the determinant of the restriction of a Hermitian metric on some subspaces.

Let $V$ be a finite dimension complex vector space. Suppose that $V$ has a flag $\mathbf{F}=\{0=F_0\subset F_1\subset F_2\subset\dots\subset F_m=E\}$. We do not require that the flag is complete. For $k= 1, \dots , m$, we define $Gr_k^F(V)=F_k(V)/F_{k-1}(V)$. There exists a natural isomorphism
\begin{align*}
\rho_k \colon \det\qty(Gr_k^F(V))\otimes\det\qty(F_{k-1})\cong\det\qty(F_k).
\end{align*}

Let $h$ be a Hermitian metric on $V$. The induced metric on $F_k(V)$ is denoted by $F_k(h)$. This induces a Hermitian metric $\det(F_k(h))$ of $\det(F_k(V))$. Let $G_k(V,h)$ be the orthogonal complement of $F_{k- 1}( V) $ in $F_k(V)$. The $h$-orthogonal projection $F_k(V)\to Gr_k^F(V)$ induces an isomorphism $G_k(V,h)\cong Gr_k^F(V)$. We obtain the metric $Gr_k^F(h)$ of $Gr_k^F(V)$ which is induced by $h|_{Gr_k(V,h)}$ and the isomorphism $G_k(V,h)\cong Gr_k^F(V)$.

\begin{lemma} \label{isometric} The morphism $\rho_k$ is isometric with respect to $\det\qty(F_k(h))$ and $\det\qty(Gr_k^F(h))\otimes\det(F_{k-1}(h))$.
\end{lemma}
\begin{proof}
    Given an orthonormal basis $\symbfit{v}$ of $F_{k-1}\qty(V)$, it can be extended to an orthonormal basis $\symbfit{v}'$ of $F_{k}\qty(V)$. Notice that $\symbfit{v}' \setminus \symbfit{v}$ is an orthonormal basis of $G_k  (V,h)$, which induces a basis of $Gr_k^F(V)$. With this basis, we can easily verify this lemma.
\end{proof}

Now, we introduce a filtration of the bundle $\mathbb{K}_{\mathrm{SO}_0(n,n)}$ and a Hermitian metric $h_X$ on it. 

The bundle $\mathbb{K}_{\mathrm{SO}_0(n,n)}$ is a direct sum of some holomorphic line bundles, i.e., $\mathbb{K}_{\mathrm{SO}_0(n,n)}=K^{n-1} \oplus K^{n-3} \oplus \dots \oplus K^{3-n} \oplus K^{1-n} \oplus  K^{n-2} \oplus K^{n-4} \oplus \dots \oplus K^{4-n}\oplus K^{2-n} \oplus \mathcal{O}'$. It can also be rearranged as 
\[
K^{n-1} \oplus K^{n-2} \oplus \dots \oplus K^{1} \oplus \mathcal{O}\oplus \mathcal{O}' \oplus  K^{-1} \oplus \dots \oplus K^{2-n}\oplus K^{1-n}. 
\] 
Define 
\begin{align} \label{filtration}
F_k=
    \begin{cases}
        \mathop{\bigoplus} \limits_{i=1}^k K^{n-i},   &1\leq k \leq n,       \\
         F_{n}\oplus \mathcal{O}', &k=n+1,\\
         F_{n+1} \oplus \mathop{\bigoplus} \limits_{i=1}^{k-n-1} K^{-i},   & n+2 \leq k \leq 2n. 
    \end{cases}
\end{align}
Then $\mathbb{K}_{\mathrm{SO}_0(n,n)}$ admits a holomorphic filtration
\[
0=F_0 \subset F_1 \subset \dots \subset F_{2n}=\mathbb{K}_{\mathrm{SO}_0(n,n)}.
\]
This filtration is denoted by $\mathcal{F}$. Notably, in the new arrangement of line bundles, 
\begin{align}\label{natural Higgs field}
    \theta\qty(\symbfit{q})=
    \begin{pNiceMatrix}[first-row,last-col,columns-width=0.4cm,nullify-dots]
        1 &\dots & n-1&n &n+1 & n+2&\dots &2n-1  & 2n \\
         * & \NotEmpty& \NotEmpty  & \NotEmpty &\NotEmpty  & \NotEmpty &\NotEmpty & \NotEmpty   & * & 1 \\
         1 & \NotEmpty     &  &  &  &  &  &   & \NotEmpty & 2\\
         &      &  &  &  &  &  &   &\NotEmpty & \Vdots\\
         &      & 1&  &  &  &  &   & \NotEmpty & n\\
         &      &  & 0 &  &  &  &  & \NotEmpty & n+1\\
         &      &  & 1 & 0 &   &  & &\NotEmpty & n+2\\
         &      &  &  &  & 1 &  &  & \NotEmpty & n+3\\
         &      &  &  &  &  &   &   &\NotEmpty & \Vdots\\
         &      &  &  &  &  &   & 1     & * & 2n
         \CodeAfter 
         \line{2-1}{4-3} \line{7-6}{9-8} 
         \line{1-1}{9-9} \line{1-2}{8-9}
         \line{1-3}{7-9} \line{1-4}{6-9}
         \line{1-5}{5-9} \line{1-6}{4-9} 
         \line{1-7}{3-9} \line{1-8}{2-9} 
         \line{1-1}{1-9} 
         \line{1-9}{9-9}
         \end{pNiceMatrix},
\end{align}
 where the holomorphic differentials $q_{1},\dots, q_{n-1}, q_{n}$ are on the upper triangle and we omit them since it is not critical in our proof.
 
Suppose $q_i$ vanishes for all $i$. The Higgs bundle $(\mathbb{K}_{\mathrm{SO}_0(n,n)}, \theta (\symbf{0}))$ admits a natural diagonal harmonic metric, and this metric is compatible with the compatible with the $\mathrm{SO}_0(n,n)$-structure when $X$ is hyperbolic.

\begin{proposition}\label{h_X}
For a hyperbolic Riemann surface $X$,
$(\mathbb{K}_{\mathrm{SO}_0(n,n)}, \theta (\symbf{0}))$ admits a diagonal harmonic metric $h_X$ which is compatible with the $\mathrm{SO}_0(n,n)$-structure.
\end{proposition}
\begin{proof}
Let $g_X$ be the unique complete K\"ahler hyperbolic metric on $X$, locally written as $g_X=g(\dif x^2+ \dif y^2)$. The induced Hermitian metric on $K_X^{-1}$ is $g \dif z\otimes \dif \bar{z}$, also denoted by $g_X$. It also induces a K\"ahler form $\omega = \frac{\sqrt{-1}}{2}g \dif z \wedge \dif \bar{z}$. Let $F(g_X)$ be the curvature of the Chern connection of $g_X$. Since $g_X$ is hyperbolic, we have 
\begin{align*}
    F(g_X)=\sqrt{-1} \omega.
\end{align*}
Recall that $\mathbb{K}_{\mathrm{SO}_0(n,n)}=\mathop{\bigoplus} \limits_{i=1}^{2n-1} K^{n-i} \oplus \mathcal{O}'$.
Now, let's define the diagonal Hermitian metric $h_X$ as follows:
\begin{align*}
h_X= \mathop{\oplus} \limits_{k=1}^{2n-1} a_{k,\,n} \cdot g_X^{k-n} \oplus h_0,
\end{align*}
where $h_0$ is the constant metric $1$ on trivial line bundle $\mathcal{O}'$ and 
\begin{align*}
a_{k,\,n}= \prod_{l=1}^{n-1}\qty(\frac{l(2n-1-l)}{2})^{-1}\prod_{l=1}^{k-1}\qty(\frac{l(2n-1-l)}{2}).
\end{align*}

Now, one can easily verify that $h_X$ is a diagonal harmonic metric of $(\mathbb{K}_{\mathrm{SO}_0(n,n)}, \theta (\symbf{0}))$ and is compatible with the $\mathrm{SO}_0(n,n)$-structure. 
\end{proof}
\begin{remark}
 Let
\[
b_{k,N}=\prod_{l=1}^{k-1} \qty(\frac{l(N-l)}{2})^{1/2} \prod_{l=k}^{N-1} \qty(\frac{l(N-l)}{2})^{-1/2}.
\]
 The metric given by $\mathop{\oplus} \limits_{k=1}^{2n-1} a_{k,\,n} \cdot g_X^{k-n}$ is exactly the metric 
\[
\widetilde{h}_{X,N}= \mathop{\bigoplus} \limits_{k=1}^{N} b_{k,\,N} \cdot g_X^{-\frac{N+1-2k}{2}},
\]
which appears in \cite[(37)]{li2023higgs} for $N=2n-1$.
\end{remark}
From now on, we call a Hermitian metric $h$ on $\mathbb{K}_{\mathrm{SO}_0(n,n)}$ weakly dominates $h_X$ if $h$  weakly dominates $h_X$ with respect to the filtration $\mathcal{F}$.

\subsection{The weak domination property of harmonic metrics}
The following weak domination property, inspired by \cite[Proposition 4.14]{li2023higgs}, plays an important role in showing the convergence of harmonic metrics in the exhaustion process. It should be noted that the result of \cite[Proposition 4.14]{li2023higgs} cannot be applied directly since the Higgs field $\theta(\symbfit{q})$ has $0$ on the sub-diagonal and has an nonzero element below the sub-diagonal (see \eqref{natural Higgs field}).  The definitions of the subbundle $F_k$ and the filtration $\mathcal{F}$ are presented in \eqref{filtration}.

\begin{proposition}\label{h dominates}
On a hyperbolic Riemann surface $X$, let $Y \subset X$ be a relatively compact connected open subset with non-empty smooth boundary $\partial Y$. Suppose that $h$ is a harmonic metric of $(\mathbb{K}_{\mathrm{SO}_0(n,n)}, \theta (\symbfit{q}))$ on $Y$ satisfying $h=h_X$ on $\partial Y$. Then, 
$h$ weakly dominates $h_X$ with respect to the filtration $\mathcal{F}$, i.e.,
\begin{align*}
    \det\qty(h|_{F_k}) \leq \det\qty(h_X|_{F_k})
\end{align*}
holds on $Y$ for $1\leq k \leq 2n$.

Moreover, if there exist a point $y \in Y+$ and $1 \leq k \leq 2n-1$ such that the equality holds, i.e.,
\begin{align*}
     \det\qty(h|_{F_k})(y) = \det\qty(h_X|_{F_k})(y),
\end{align*}
then $\symbfit{q}=0$ and $h=h_X$ on $Y$.
\end{proposition}
\begin{proof}
Restrict $(\mathbb{K}_{\mathrm{SO}_0(n,n)}, \theta (\symbfit{q}))$ to $Y$. The first step is to derive the Hitchin equation with respect to a holomorphic subbundle $F$ of $\mathbb{K}_{\mathrm{SO}_0(n,n)}$. Let $F^\perp$ be the  orthogonal complement of $F$ with respect to the harmonic metric $h$. Equip $F^\perp$ the quotient holomorphic structure induced from $\mathbb{K}_{\mathrm{SO}_0(n,n)}/F$. With respect to the smooth orthogonal decomposition
\[
\mathbb{K}_{\mathrm{SO}_0(n,n)}=F\oplus F^\perp,
\]
 the holomorphic structure $\bar{\partial}_{\mathbb{K}_{\mathrm{SO}_0(n,n)}}$, the Higgs field $\theta(\symbfit{q})$ and the harmonic metric $h$ are as follows:
\[
\bar{\partial}_{\mathbb{K}_{\mathrm{SO}_0(n,n)}}=
\begin{pmatrix}
\bar{\partial}_F&  \beta    \\
0               &  \bar{\partial}_{F^{\perp}}
\end{pmatrix},
\quad\theta(\symbfit{q})=
\begin{pmatrix}
\alpha_1 & \alpha\\
B & \alpha_2
\end{pmatrix},
\quad 
h=\begin{pmatrix}
h_{F}&0\\0&h_{F^{\perp}}
\end{pmatrix},
\]
where $B\in \Omega^{1, 0}( X, \operatorname{Hom}( F, F^\perp ) ) , \alpha\in \Omega^{1, 0}( X, \operatorname{Hom}( F^\perp , F) )$, and $\beta\in \Omega^{0, 1}( X, \operatorname{Hom}( F^\perp , F) )$. 

The Chern connection $\nabla_h$ and the adjoint $\theta(\symbfit{q})^{*h}$ of the Higgs field take the following form:
\[
\nabla_h=
\begin{pmatrix}
\nabla_{h_F}&\beta\\
-\beta^{*h}&\nabla_{h_{F^{\perp}}}
\end{pmatrix},
\quad
\theta(\symbfit{q})^{*h}=
\begin{pmatrix}
\alpha_1^{*h_F}&B^{*h}\\
\alpha^{*h}&\alpha_2^{*h_{F^{\perp}}}
\end{pmatrix}.
\]

Substituting these expressions into the Hitchin equation and restricting to $\operatorname{Hom}(F,F)$, we obtain
\[
F(\nabla_{h_{F}})-\beta\wedge\beta^{*h}+\alpha\wedge\alpha^{*h}+B^{*h}\wedge B+\commutator{\alpha_{1}}{\alpha_{1}^{*h_{F}}}=0.
\]
Taking trace and noting that $ \tr ( \commutator{\alpha_1}{\alpha_1^{* h_F}} ) = 0$, we have
\[
\tr(F(\nabla_{h_F}))-\tr(\beta\wedge\beta^{*h})+\tr(\alpha\wedge\alpha^{*h})+\tr(B^{*h}\wedge B)=0.
\]

Let $g_X$ be the unique complete K\"ahler hyperbolic metric on $X$. Denote by $\Lambda_{g_{X}}$ the contraction with respect to the associated K\"ahler form $\omega$ of $g_X$. Therefore,
\begin{align} 
& -\sqrt{-1}\Lambda_{g_{X}} \tr(F(\nabla_{h_{F}}))-\sqrt{-1}\Lambda_{g_{X}} \tr(B^{*h}\wedge B) \nonumber \\
=& -\sqrt{-1}\Lambda_{g_{X}}\tr(\beta\wedge\beta^{*h})+
\sqrt{-1}\Lambda_{g_{X}}\tr(\alpha\wedge\alpha^{*h})  \nonumber \\
=&\abs{\beta}_{h,g_{X}}^2+\abs{\alpha}_{h,g_{X}}^2 \nonumber \\
\geq& 0. \label{HItchin equation respect to decomposition}
\end{align}

It is necessary to recall the definition of the subbundle $F_k$ as presented in \eqref{filtration}. 
For $1 \leq k \leq 2n$, take $L_k$ as the $h$-orthogonal bundle inside $F_k$ with respect to $F_{k-1}$. In this way, $\mathbb{K}_{\mathrm{SO}_0(n,n)}$ has a $C^{\infty}$-decomposition as
\begin{align}
& L_1 \oplus L_2 \oplus \dots \oplus L_{2n},
\end{align} 
where  $L_k \cong K^{n-k}$ for $1 \leq  k \leq n$, $L_{n+1} \cong \mathcal{O}'$ and $L_k \cong K^{n-k+1}$ for $n+2 \leq k \leq 2n$.

From \eqref{Higgs field} or \eqref{natural Higgs field}, we can easily see that $\theta(\symbfit{q}) \colon F_{k} \to F_{k+1} \otimes K$ except for $k=n$ and $\theta(\symbfit{q}) \colon F_{n} \to F_{n+2} \otimes K$. So under the above smooth decomposition, we have the following:
\begin{itemize}
\item the metric $h$ is given by 
\begin{align*}
    h=
    \begin{pNiceMatrix}
    h_1 &          &     \\
    &  \Ddots  &      \\    
    &           & h_{2n}
    \end{pNiceMatrix},
\end{align*}
where $h_i$ is the restriction of $h$ on $L_i$ and $h_i=\det\qty(h|_{F_i})/\det\qty(h|_{F_{i-1}})$ (assume $\det\qty(h|_{F_0})=1$);
\item  the holomorphic structure on $\bar{\partial}_{\mathbb{K}_{\mathrm{SO}_0(n,n)}}$ is given by 
\begin{align*}
    \bar{\partial}_{\mathbb{K}_{\mathrm{SO}_0(n,n)}}=
    \begin{pNiceMatrix}
    \bar{\partial}_1 & \beta_{1,\,2} & \beta_{1,\,3} & \dots & \beta_{1,\,2n} \\
    & \bar{\partial}_2 & \beta_{2,\,3} & \dots & \beta_{2,\,2n} \\
    & & \bar{\partial}_3 & \dots & \beta_{3,\,2n} \\
    & & & \Ddots & \Vdots \\
    & & & & \bar{\partial}_{2n}
    \end{pNiceMatrix},
\end{align*}
where $\bar{\partial}_k$ is the $\bar{\partial}$-operators\ defining the holomorphic structure on $L_k$, and $\beta_{i,\,j} \in \Omega^{0,1}\left(\operatorname{Hom}\left(L_j, L_i \right)\right)$; 
\item the Higgs field is given by
\begin{align} \label{Higgs field under orthogonal basis}
    \theta(\symbfit{q})=
    \begin{pNiceMatrix}[columns-width=0.2cm]
        a_{1,\,1} & \NotEmpty& \NotEmpty  & \NotEmpty &\NotEmpty  & \NotEmpty &\NotEmpty & \NotEmpty   &a_{1,\,2n} \\
        \theta_1 & \NotEmpty     &  &  &  &  &  &   & \NotEmpty\\
        &      &  &  &  &  &  &   &\NotEmpty\\
        &      & \theta_{n-1}&  &  &  &  &   & \NotEmpty\\
        &      &  & \gamma &  &  &  &  & \NotEmpty \\
        &      &  & \theta_n & \gamma' &   &  & &\NotEmpty\\
        &      &  &  &  & \theta_{n+1} &  &  & \NotEmpty\\
        &      &  &  &  &  &   &    &\NotEmpty \\
        &      &  &  &  &  &   & \theta_{2n-2}& a_{2n,\,2n} 
        \CodeAfter 
        \line{2-1}{4-3} \line{7-6}{9-8} 
        \line{1-1}{9-9} \line{1-2}{8-9}
        \line{1-3}{7-9} \line{1-4}{6-9}
        \line{1-5}{5-9} \line{1-6}{4-9} 
        \line{1-7}{3-9} \line{1-8}{2-9} 
        \line{1-1}{1-9} 
        \line{1-9}{9-9}
        \end{pNiceMatrix},
\end{align}
where $a_{i,\,j} \in \Omega^{1,0}\left(\operatorname{Hom}\left(L_j, L_i\right)\right)$, $\theta_k$ is the identity morphism, $\gamma \in \Omega^{1,0}\left(\operatorname{Hom}\left(L_n, L_{n+1}\right)\right)$ and $\gamma'\in \Omega^{1,0}\left(\operatorname{Hom}\left(L_{n+1}, L_{n+2}\right)\right)$.
\end{itemize}

For any Hermitian metric $h$ on $\mathbb{K}_{\mathrm{SO}_0(n,n)}$, which is compatible with the $\mathrm{SO}_0(n,n)$-structure, we have the following useful properties: 
\begin{align}\label{det equality} 
  \det \qty(h|_{F_{k}}) = \det \qty(h|_{F_{2n-k}}),  \quad  1 \leq k \leq 2n-1, 
\end{align}
which we will prove in Lemma~\ref{det identities}. So we only need to prove the following inequalities: 
\begin{align*}
\det\qty(h|_{F_k}) \leq \det\qty(h_X|_{F_k}), \quad 1\leq k \leq n.
\end{align*}

We now apply the preceding formulas to the specific subbundles $F= F_k$ for $k= 1, \dots,n+1$. The associated factor $B$ for each case is as follows:
\begin{align*}
    \intertext{
    \begin{itemize}
        \item $1 \leq k \leq n-1$
    \end{itemize}
    }
    B & =
\begin{pmatrix}
0&0&\dots&0&\theta_k\\
0&0&\dots&0&0\\
\vdots&\vdots&\dots&\vdots&\vdots\\
0&0&\dots&0&0
\end{pmatrix}
\colon F_k\to(L_{k+1}\oplus\dots\oplus L_{2n})\otimes K,
\intertext{
    \begin{itemize}
        \item $k=n$
    \end{itemize}
    }
    B&=
\begin{pmatrix}
0&0&\dots&0&\gamma\\
0&0&\dots&0&\theta_n\\
\vdots&\vdots&\dots&\vdots&\vdots\\
0&0&\dots&0&0
\end{pmatrix}
\colon F_{n}\to(L_{n+1}\oplus\dots\oplus L_{2n})\otimes K,
\intertext{
    \begin{itemize}
        \item $k=n+1$
    \end{itemize}
    }
    B&=
\begin{pmatrix}
0&0&\dots&\theta_n&\gamma'\\
0&0&\dots&0&0\\
\vdots&\vdots&\dots&\vdots&\vdots\\
0&0&\dots&0&0
\end{pmatrix}
\colon F_{n+1}\to(L_{n+2}\oplus\dots\oplus L_{2n})\otimes K.
\end{align*}
Then 
\begin{alignat}{10}
    \sqrt{-1}\Lambda_{g_{X}} \tr(B^{*h}\wedge B)&=
-\abs{\theta_k}_{h,g_X}^2, & \quad  1\leq k &\leq n-1, \nonumber \\
\sqrt{-1}\Lambda_{g_{X}} \tr(B^{*h}\wedge B)&=
-\abs{\theta_n}_{h,g_X}^2-\abs{\gamma}_{h,g_X}^2, & \quad k &=n, \nonumber\\
\sqrt{-1}\Lambda_{g_{X}} \tr(B^{*h}\wedge B)&=
-\abs{\theta_n}_{h,g_X}^2-\abs{\gamma'}_{h,g_X}^2, & \quad k&=n+1. \nonumber
\intertext{By \eqref{HItchin equation respect to decomposition},}
 -\sqrt{-1}\Lambda_{g_{X}} \tr(F(h|_{F_k})) & \geq  
 -\abs{\theta_k}_{h,g_X}^2, & \quad 1\leq k &\leq n-1, \label{Hitchinequa k} \\
 -\sqrt{-1}\Lambda_{g_{X}} \tr(F(h|_{F_n})) & \geq  
-\abs{\theta_n}_{h,g_X}^2-\abs{\gamma}_{h,g_X}^2, & \quad  k &= n, \label{Hitchinequa n}\\
 -\sqrt{-1}\Lambda_{g_{X}} \tr(F(h|_{F_{n+1}})) & \geq  
 -\abs{\theta_n}_{h,g_X}^2-\abs{\gamma'}_{h,g_X}^2, & \quad  k &=n+1. \label{Hitchinequa n+1}
\intertext{Note that the equation \eqref{HItchin equation respect to decomposition} for  $\qty(\mathbb{K}_{\mathrm{SO}_0(n,n)}, \theta (\symbf{0}), h_X)$ and $F=F_k$ becomes}
    -\sqrt{-1}\Lambda_{g_{X}}  \tr(F(h_X|_{F_k}))&=  
    -\abs{\theta_k}_{h_X,g_X}^2,  & \quad 1\leq k & \leq n. \label{ Hitequa for 0 case}
\end{alignat}
The subsequent step is to calculate $\abs{\theta_k}_{h,g_X}^2$ for $1 \leq k \leq n$.
\begin{align}
\abs{\theta_k}_{h,g_X}^2& =\frac{h_{k+1}}{h_k}/g_X= \frac{\det\qty(h|_{F_{k-1}})\det\qty(h|_{F_{k+1}})}{\det\qty(h|_{F_{k}})^2}/g_X, \quad 1 \leq k \leq n-1,\label{theta k}\\
  \abs{\theta_{n}}_{h,g_X}^2&=\frac{h_{n+2}}{h_{n}}/g_X=\frac{\det\qty(h|_{F_{n+2}})}{\det\qty(h|_{F_{n+1}})}\cdot
\frac{\det\qty(h|_{F_{n-1}})}{\det\qty(h|_{F_{n}})}/g_X=\frac{\det\qty(h|_{F_{n-2}})}{\det\qty(h|_{F_{n}})}/g_X. \label{theta n}
\end{align}
In the last line we used the equalities $\det\qty(h|_{F_{n+1}})=\det\qty(h|_{F_{n-1}})$ and $\det\qty(h|_{F_{n+2}})=\det\qty(h|_{F_{n-2}})$.

The error terms $\gamma$ and $\gamma'$ appear difficult to handle, but fortunately they vanish depending on the parity of $n$.

\textbf{When $n$ is even}, the error term $\gamma$ is $0$, as we will prove in Lemma~\ref{error term gamma is zero}. Hence the inequality \eqref{Hitchinequa n} becomes:
\begin{align} \label{Hitchinequa n+1 for even n}
     -\sqrt{-1}\Lambda_{g_{X}} \tr(F(h|_{F_n})) & \geq  
-\abs{\theta_n}_{h,g_X}^2.
\end{align}

\textbf{When $n$ is odd}, Lemma~\ref{error term gamma is zero} implies that the error term $\gamma'$ vanishes. Furthermore, the  key equality 
\begin{align} \label{det n = det n}
    \det \qty(h|_{F_{n}})=
    \det \qty(h|_{F_{n+1}}),
\end{align}
which is established in Lemma~\ref{det identities}, holds.
 Hence the inequality \eqref{Hitchinequa n+1} becomes:
\begin{align} \label{Hitchinequa n+1 for odd n}
     -\sqrt{-1}\Lambda_{g_{X}} \tr(F(h|_{F_n})) & \geq  
-\abs{\theta_n}_{h,g_X}^2.
\end{align}
 From \eqref{Hitchinequa k}, \eqref{Hitchinequa n+1 for even n} and \eqref{Hitchinequa n+1 for odd n}, it follows that,regardless of the parity of $n$,
 \begin{align}
     -\sqrt{-1}\Lambda_{g_{X}} \tr(F(h|_{F_k})) & \geq  
 -\abs{\theta_k}_{h,g_X}^2,  \quad 1\leq k \leq n. \label{Hitchinequa full}
 \end{align}

 Now, we can use the method in \cite[Proposition 4.14]{li2023higgs} to prove this theorem. 
 
 Set $v_k= \log \frac {\det ( h|_{F_k})} {\det ( h_X|_{F_k}) }$ for $1\leq k\leq n$ and $v_0=0$. The Laplacian with respect to $g_X$ is $2\sqrt{-1}\Lambda_{g_X}\partial\bar{\partial}$, denoted by $\Delta_{g_X}$. From \eqref{theta k}, \eqref{theta n} and \eqref{Hitchinequa full}, we obtain
\begin{alignat*}{10}
\frac12 \Delta_{g_X}v_k+(e^{v_{k-1}+v_{k+1}-2v_k}-1) \cdot \frac{\det(h_X|_{F_{k-1}})\det(h_X|_{F_{k+1}})} {\det(h_X|_{F_k})^2}/g_X & \geq 0, \quad k=1,\dots,n-1.  \\
\frac12 \Delta_{g_X}v_{n}+(e^{v_{n-2}-v_{n}}-1) \cdot
\frac{\det(h_X|_{F_{n-2}})} {\det(h_X|_{F_{n}})}/g_X &\geq 0.
\end{alignat*}
Let $c_k : X \to \mathbb{R}$ be functions, for $1 \leq k \leq n$, given by 
\begin{alignat*}{10}
c_k&=\frac{\det(h_X|_{F_{k-1}})\det(h_X|_{F_{k+1}})} {\det(h_X|_{F_k})^2}/g_X\int_0^1e^{(1-t)(v_{k-1}+v_{k+1}-2v_k)} \dif t,  \quad k=1,\dots,n-1, \\
c_{n}&=\frac{\det(h_X|_{F_{n-2}})} {\det(h_X|_{F_{n}})}/g_X\int_0^1e^{(1-t)(v_{n-2}-v_{n})} \dif t. 
\end{alignat*}
Then $v_k$'s satisfy
\begin{alignat}{10}
\frac12\Delta_{g_X}v_k+c_k(v_{k-1}-2v_k+v_{k+1}) &\geq 0,\quad k=1,\dots,n-1,  \label{v_k equation linear}\\
\frac12\Delta_{g_X}v_{n}+c_n(v_{n-2}-v_{n}) &\geq 0.  \label{v_n equation linear}
\end{alignat}
 By the assumption on the boundary value of $h$, $v_k=0$ on $\partial Y$ for $k=1,\dots,n$. It is easy to check that the above system of equations satisfies the assumptions in Lemma ~\ref{boundary maximum}. Moreover, $(1,1,\dots,1)$ is indeed a super-solution of this system. Then we can apply Lemma ~\ref{boundary maximum} and obtain $v_k\leq 0$ on $Y$ for $k=1,\dots,n$. 

 Now suppose for some $1\leq k \leq n-1$ and $y \in Y$, $v_k\qty(y)=0$. Since $v_{k-1}, v_{k+1} \leq 0$, we have
 \begin{align*}
     \frac12\Delta_{g_X}v_k-2c_k\cdot v_k \geq 0.
 \end{align*}
 Notice that the constant function $0$ satisfies 
 \begin{align*}
     \frac12\Delta_{g_X}0-2c_k\cdot 0 = 0.
 \end{align*}
 Using the strong maximum principle in \cite[Theorem 3.3.1 in page 49]{jost2012partial}, we conclude the zero set of $v_k$ is open. Since $Y$ is connected, $v_k=0$ holds on $Y$. The equation \eqref{v_k equation linear} and $v_{k-1}, v_{k+1} \leq 0$ imply 
 $v_{k-1}=v_{k+1}=0$ holds on $Y$. Repeating the process, we find that $v_k=0$ holds on $Y$ for all $k$. 

 For $k=n$, the method is similar. In summary, if there exist a point $y \in Y$ and $1 \leq k \leq 2n-1$ such that $\det\qty(h|_{F_k})(y) = \det\qty(h_X|_{F_k})(y)$, then this equality holds on $Y$ for all $k$. From \eqref{HItchin equation respect to decomposition}, \eqref{Hitchinequa k}, \eqref{Hitchinequa n},\eqref{ Hitequa for 0 case}, \eqref{theta k} and \eqref{theta n}, we can deduce that the extra terms $\alpha$ and $\beta$ in the smooth decomposition are both $0$ for all $F_{k}$, i.e., $\symbfit{q}=0$ and $h=h_X$.
\end{proof}

\begin{lemma} \cite[Theorem 1]{maximum} \label{boundary maximum}
 Let $\left(X,g\right)$ be a compact Riemannian manifold with boundary. For each $1\leq i\leq n$, let $u_i$ be a $C^2$ real valued function on X satisfying:
\[
\Delta_gu_i+\sum_{j=1}^nc_{i,\,j}u_j\geq0,\quad1\leq i\leq n,
\]
where $c_{i,\,j}$ are continuous functions on X, $1\leq i,j\leq n$, satisfying
\begin{itemize}
    \item  cooperative: $c_{i,\,j} \geq 0$ for $i \neq j$,
    \item  fully coupled: the index set $\{1,\dots,n\}$ cannot be split up in two disjoint non-empty sets $\alpha, \beta$ such that $c_{i,\,j}\equiv0$ for $i\in\alpha,j\in\beta$.
\end{itemize}
Suppose that there exists a super-solution $(\psi_1,\psi_2,\dots,\psi_m)$ satisfying $\psi_i \geq 1$ of the above system, i.e.,
\[
\Delta_g\psi_i+\sum_{j=1}^nc_{ij}\psi_j\leq0,\quad 1\leq i\leq n.
\]
Then
\[
\sup_{X}u_i\leq\sup_{\partial X}u_i,\quad1\leq i\leq n.
\]
\end{lemma}
\subsection{Proof of some lemmas}
In this subsection, we will demonstrate how several lemmas that were previously left unproven can be established using the compatibility condition.

Let $V$ be a $2n$ dimensional complex linear space equipped with a non-degenerate symmetric pairing $C$. Suppose $V=\mathop{\bigoplus}\limits_{k=1}^{2n} V_k$, where the dimension of each $V_k$ is equal to $1$.

We have a natural filtration $\symbfit{F}=\lbrace 0 =F_0 \subset F_1 \subset \dots \subset F_{2n}=V \rbrace $ such that $F_{k}=\mathop{\bigoplus} \limits_{i=1}^{k} V_{i}$. Notice that we have a natural isomorphism by $V_{k} \simeq F_k \big /{F_{k-1}}$. 

Let $h$ be a Hermitian metric of $V$, which is compatible with $C$. For each $1 \leq k \leq 2n $, it induces a metric $F_{k}(h)$ on $F_k$. We define $L_{k}$ as the orthogonal complement of $F_{k-1}$ in $F_{k}$. Consequently, we have an isomorphism given by $V_k \cong F_{k} \big/F_{k-1} \cong L_{k}$. Denote the image of $v \in V_k$ under this isomorphism by $v^{\perp} \in L_k$. 

In our case, we can assume that there exists an element $v_k$ for each $V_k$ such that $C$ is represented by the following matrix with respect to the basis $\symbfit{v}=\qty{v_1,\dots,v_{2n}}$ :
\begin{align}\label{C natural base}
C=
\begin{pNiceMatrix}[first-row,last-col,columns-width=0.4cm,nullify-dots]
    1&\dots &n-1 &n&n+1&n+2&\dots&2n \\
     & & & & & & & 1 &1\\
     & & & & & &\Iddots &  &\Vdots\\
     & & & & & 1 & &  &n-1\\
     & & &1 &0 & & &  &n\\
     & & &0 &1 & & & &n+1 \\
     & &1 & & & & &  &n+2\\
     &\Iddots & & & & & & &\Vdots \\
     1& & & & & & &  &2n \\
     \end{pNiceMatrix},
\end{align}
i.e., $C_{i,\,j}=\delta_{i,\,2n-j+1}$ for $\{i,j\} \nsubseteq \{n, n+1\}$ and $C_{n,\,n}=C_{n+1,\,n+1}=1$, $C_{n,\,n+1}=C_{n+1,\,n}=0$.

Using $\symbfit{v} ^{\perp}$ to denote the $h$-orthogonal basis $\lbrace v_1^{\perp},\dots,v_{2n}^{\perp}\rbrace$, we can obtain the following lemma.
\begin{lemma}\label{C in orthogonal base}
    For $v_i \in V_i$ and $v_j \in V_j$, if $i+j \leq 2n+1$ and $\qty{i,j}\neq\qty{n,n+1}$, we have
\begin{align*}
C(v_i^{\perp}, v_j^{\perp}) &=C(v_i, v_j).
\end{align*}
\end{lemma}
\begin{proof}
Notice that $v_i= \hat{v_i} + v_i^{\perp}$ where $\hat{v_i} \in F_{i-1}$ and $v_i^{\perp} \in L_{i}$. From \eqref{C natural base}, we know that 
\[
C(u,v)=0
\]
when $u \in F_k,v \in F_l$ and $k+l \leq 2n$ except for $k=l=n$. Hence
\begin{align*}
C(v_i,v_j) & =C(\hat{v_i}+v_i^{\perp} ,\hat{v_j}+ v_j^{\perp}) \\
 & = C(\hat{v_i},\hat{v_j})+C(\hat{v_i},v_j^{\perp}) \\
& + C(v_i^{\perp},\hat{v_j}) + C(v_i^{\perp},v_j^{\perp})         \\
& = C(v_i^{\perp},v_j^{\perp}).
\end{align*}
\end{proof}
According to Lemma~\ref{C in orthogonal base}, with respect to the $h$-orthogonal basis $\symbfit{v}^{\perp}$, the matrix of $C$ takes the following form:
\begin{align}\label{partitioning of C}
    C=
    \left(
    \begin{NiceArray}[columns-width=0.4cm]{c|c|c|c c|c| c|c} 
    & & & & & & & 1 \\
    \hline
    & & & & & &\iddots & \vdots \\
    \hline
    & & & & & 1 &\dots & * \\
    \hline
    & & &1 &C_{n,\,n+1} &* &\dots & *  \\
    & & &C_{n+1,\,n}&C_{n+1,\,n+1} &*&\dots &*  \\
    \hline
    & &1 &*&* &* & \dots&*\\
    \hline
    &\iddots & \vdots&\vdots &\vdots &\vdots & \iddots &\vdots  \\
    \hline
    1 &\dots &* &* &* &* &\dots & *
\end{NiceArray}
    \right).
\end{align}
 In the same display, we also indicate a partition of $C$ into blocks. 
With respect to the $h$-orthogonal basis $\symbfit{v^{\perp}}$, the matrix form of $h$ is
\begin{align}\label{partitioning of H}
    H=
    \left(
    \begin{NiceArray}{c|c|c c|c|c}
         H_1 &  & & & & \\
         \hline
           &\ddots& & & & \\
           \hline
           &      & H_{n} & & & \\
           &      &       & H_{n+1} & & \\
           \hline
           & & & & \ddots & \\
           \hline
           & & &     &     & H_{2n}
      \end{NiceArray}
    \right).
\end{align}
 As in the above case, we also provide a partitioning of $H$.

We will now prove a more general lemma which can be applied to our specific case. 
\begin{lemma}\label{simplify C H in orthogonal base}
 Let 
\begin{align*}
C=
\begin{pmatrix}
& & C_{1,\,s} \\
&\iddots &\vdots \\
C_{s,\,1}&\dots & C_{s,\,s}\\
\end{pmatrix} \quad
H=
\begin{pmatrix}
H_1 &      &        \\
&  \ddots  &        \\ 
&          & H_{s}  
\end{pmatrix}
\end{align*}
be two block matrices. 
Make the following assumptions:
\begin{itemize}
    \item  $C_{i,\,j}=\trans{C_{j,\,i}}$,
    \item  the matrices $C$ and $H$ are both non-degenerate,
    \item  the matrices $C$ and $H$ satisfy the equation \eqref{comp}, i.e., $\trans{H}=\overline{C}H^{-1}C$.
\end{itemize}
Then
    \begin{align*}
\trans{H_k}&=\overline{C}_{k,\,s-k+1}H_{s-k+1}^{-1}C_{s-k+1,\,k}   
\intertext{and for $i+j > s+1$}
C_{i,\, j}&=0.
\end{align*}
\end{lemma}
\begin{proof}
Since $C$ and $H$ are both non-degenerate, the elements on the skew-diagonal of $C$ and on the diagonal of $H$ are all non-degenerate.

Notice that
\begin{align*}
\overline{C}H^{-1} &=
\begin{pmatrix}
& & \overline{C}_{1,\,s} \\
&\iddots &\vdots \\
\overline{C}_{s,\,1}&\dots & \overline{C}_{s,\,s}\\
\end{pmatrix}
\begin{pmatrix}
H_1^{-1} &      &        \\
&  \ddots  &        \\ 
&          & H_s^{-1}  
\end{pmatrix} \\
&=  
\begin{pNiceMatrix}[columns-width=auto]
    &        &       &  \overline{C}_{1,\,s} H_{s}^{-1}\\
    &        & \overline{C}_{2,\,s-1} H_{s-1}^{-1} &   \overline{C}_{2,\,s} H_{s}^{-1} \\  
    & \Iddots   & \Vdots   &   \Vdots           \\
    &          &     &        \\
\overline{C}_{s,\,1} H_1^{-1}   & \dots &\overline{C}_{s,\,s-1} H_{s-1}^{-1} &  \overline{C}_{s,\,s}H_{s}^{-1}
   \end{pNiceMatrix}.
\end{align*}
Considering the $(1,1)$ element of $\overline{C}H^{-1}C$ and $\trans{H}$, we get 
\[
\trans{H_1}= \overline{C}_{1,\,s} H_{s}^{-1}C_{s,\,1}.
\]
Similarly, considering the $(1,k)$ element of $\overline{C}H^{-1}C$ and $\trans{H}$ for $k \geq 2$, we get 
\begin{align*}
0&=\overline{C}_{1,\,s}H_{s}^{-1}C_{s,\,k}, \\
\intertext{i.e.,}
C_{s,\,k}&=C_{k,\,s}=0
\end{align*}
since $C_{1,\,s}$ and $H_{s}$ are both non-degenerate.

By induction, we can easily verify that
\begin{align*}
\trans{H_k}&=\overline{C}_{k,\,s-k+1}H_{s-k+1}^{-1}C_{s-k+1,\,k}   
\intertext{and for $i+j > s+1$}
C_{i,\, j}&=0.  
\end{align*}
\end{proof}

Recalling that the partitionings of $C$ and $H$ which are given by \eqref{partitioning of C} and \eqref{partitioning of H}, we get 
\begin{alignat}{10}
    H_{k}&=H_{2n-k+1}^{-1}, &\quad  k &\neq n,n+1  \\
       C_{i,\,j}&=0, & i+j & >2n+1 \quad \text{except for} \quad i=j=n+1.      
\end{alignat}
Notably, $\det\qty(H)=1$ since $H$ is compatible with $C$. So we have $H_{n}=H_{n+1}^{-1}$. Consequently, for any $1\leq k \leq 2n$, we have
\begin{align}\label{h equality}
    H_{k}=H_{2n-k+1}^{-1}.
\end{align}

There is a simple but useful lemma, which we shall use repeatedly. Note that the direct sum $V=W_1\oplus W_2 \oplus U_1 \oplus U_2$ in this lemma is \textbf{not} assumed $h$-orthogonal, but just a vector space direct sum.
\begin{lemma}\label{easy but useful lemma}
    Let $\qty(V,h)$ be a Hermitian space. Suppose that $V=W_1\oplus W_2 \oplus U_1 \oplus U_2$ and $W_1\oplus U_1$ is $h$-orthogonal to $W_2\oplus U_2$. Then for any $u_2 \in U_2$, $u_2-u_2^{\perp} \in W_2$, where $u_2^{\perp}$ is the projection of $u_2$ to the orthogonal complement of $W_1\oplus W_2\oplus U_1$.
\end{lemma}
The proof is straightforward. Now we can calculate $H_n$ and $H_{n+1}$ under certain assumptions. Recall that $V=\mathop{\bigoplus}\limits_{k=1}^{2n} V_k$ and $F_{k}=\mathop{\bigoplus} \limits_{i=1}^{k} V_{i}$. The space $V$ is endowed with a non-degenerate symmetric form $C$ and a Hermitian metric $h$ that is compatible with $C$. Furthermore, we select a basis vector $v_k$ in each $V_k$ such that the matrix of $C$  relative to the basis takes the form \eqref{C natural base}. Let $v_k^{\perp}$ be the $h$-orthogonal projection of $v_k$ onto the $h$-orthogonal complement of $F_{k-1}$ in $F_k$. Under the $h$-orthogonal basis $\lbrace v_1^{\perp},\dots,v_{2n}^{\perp}\rbrace$, the matrices of $C$ and $h$ are given by \eqref{partitioning of C} and \eqref{partitioning of H} respectively.
\begin{lemma}\label{hn=1}
 Suppose that there exists a splitting $F_{n-1}=F_{n-1}'\oplus F_{n-1}''$ such that $F_{n-1}'\oplus V_{n}$ is $h$-orthogonal to $F_{n-1}''\oplus V_{n+1}$. Then
\[
H_n=H_{n+1}=1.
\]
\end{lemma}
\begin{proof}
Lemma~\ref{easy but useful lemma} implies that $\hat{v}_{n+1}=v_{n+1}-v_{n+1}^{\perp} \in F_{n-1}''\subset F_{n-1}$. The same process in proof of Lemma ~\ref{C in orthogonal base} implies $C_{n,\,n+1}=C(v_{n}^{\perp}, v_{n+1}^{\perp}) =C(v_{n}, v_{n+1})=0$. From the compatibility condition, we know that
\begin{align*}
    \begin{pmatrix}
        H_n &  \\
         & H_{n+1}
    \end{pmatrix}
    &=
    \begin{pmatrix}
      1   &     \\
         & \overline{C}_{n+1,\,n+1}
    \end{pmatrix}
    \begin{pmatrix}
      H_n^{-1}   &     \\
         &   H_{n+1}^{-1}
    \end{pmatrix}
    \begin{pmatrix}
     1    &     \\
         & C_{n+1,\,n+1}
    \end{pmatrix} \\
    &=
    \begin{pmatrix}
      H_n^{-1}   &     \\
         &   \abs{C_{n+1,\,n+1}}^2 H_{n+1}^{-1}
    \end{pmatrix}.
\end{align*}
Hence $H_n=H_{n+1}=1$ by \eqref{h equality}.
\end{proof}

Now we can prove the lemmas left to this subsection. 

Let $X$ be a Riemann surface. Recall the construction of $(\mathbb{K}_{\mathrm{SO}_0(n,n)}, \theta (\symbfit{q}))$ in Definition~\ref{definition of Higgs bundles in the Hitchin section}. The filtration $\mathcal{F}=\lbrace 0 =F_0 \subset F_1 \subset \dots \subset F_{2n}=V \rbrace$ is defined by \eqref{filtration}.
For a point $p \in X$, let $V$ be the fiber of $\mathbb{K}_{\mathrm{SO}_0(n,n)}=K^{n-1} \oplus K^{n-2} \oplus \dots \oplus K^{1} \oplus \mathcal{O}\oplus \mathcal{O}' \oplus  K^{-1} \oplus \dots \oplus K^{2-n}\oplus K^{1-n}$ on the point $p$. Choosing a local holomorphic coordinate $z$ of $X$ around $p$, we have the basis $\symbfit{v}=\qty{v_1=\dif z^{n-1},\dots,v_{n-1}=\dif z,v_{n}=1,v_{n+1}=1,v_{n+2}=\dif z^{-1},\dots,v_{2n}=\dif z^{1-n}}$. Then $F_{k}$ is the linear span of\{ $v_1,\dots,v_k$\}. 

Let $C$ be the natural non-degenerate symmetric pairing. With respect to the basis $\symbfit{v}$, the matrix form of $C$ is given by \eqref{C natural base} and the Higgs field $\theta (\symbfit{q})$ takes the form \eqref{natural Higgs field}. Consider a Hermitian metric $h$ on $V$ which is compatible with the $\mathrm{SO}_0(n,n)$-structure. For each $k$, let $L_k$ be the $h$-orthogonal complement of $F_{k-1}$ in $F_k$. Denote by $v_k ^{\perp}$ the orthogonal projection of $v_k$ onto $L_k$. Then $\symbfit{v} ^{\perp}=\lbrace v_1^{\perp},\dots,v_{2n}^{\perp}\rbrace$ is an $h$-orthogonal basis. Moreover, $F_{k}$ is also the linear span of $\{ v_1^{\perp},\dots,v_k^{\perp}\}$.
\begin{lemma} \label{det identities}
   For $1 \leq k \leq 2n-1$, we have
    \begin{align} \label{general n det identities}
        \det \qty(h|_{F_{k}}) = \det \qty(h|_{F_{2n-k}}). 
    \end{align}
    Moreover, when $n$ is odd, 
    \begin{align} \label{odd n det identities}
         \det \qty(h|_{F_{n}})=
    \det \qty(h|_{F_{n+1}}). 
    \end{align}
\end{lemma}
\begin{proof}
 First, we establish equations \eqref{general n det identities}. By the symmetry of index $k$ and $2n-k$, we only need to prove these equalities for $1 \leq k \leq n$. Under the $h$-orthogonal basis $\symbfit{v} ^{\perp}=\lbrace v_1^{\perp},\dots,v_{2n}^{\perp}\rbrace$, the matrix form of $h$ is \eqref{partitioning of H}. We have already established the equalities \eqref{h equality}, namely, 
  $H_{k}=H_{2n-k+1}^{-1}$ for any $1 \leq k \leq 2n$. 
 By Lemma~\ref{isometric}, for $1 \leq k \leq n$, we have
\begin{align*}
    \det \qty(h|_{F_{2n-k}})= \prod_{i=1}^{k} H_i \cdot \prod_{i=k+1}^{2n-k} H_i = \prod_{i=1}^{k} H_i=\det \qty(h|_{F_{k}}),
\end{align*}
where the last equality follows because the factors $H_{k+1},\dots,H_{2n-k}$ cancel pairwise due to the relations 
$H_{i}=H_{2n-i+1}^{-1}$ for $1 \leq i \leq 2n$.

Finally, we prove the equation \eqref{odd n det identities} for odd $n$. Define
\begin{align*}
    F_{n-1}'&= K^{n-1} \oplus K^{n-3} \oplus \dots \oplus K^4 \oplus K^2, \\
    F_{n-1}'' &= K^{n-2} \oplus K^{n-4} \oplus \dots \oplus K^3 \oplus K^{1}. 
\end{align*}
 
Then $F_{n-1}=F_{n-1}' \oplus F_{n-1}''$ and $F_{n-1}' \oplus \mathcal{O}$ is $h$-orthogonal to $F_{n-1}'' \oplus \mathcal{O}'$ (see Definition~\ref{definition of Higgs bundles in the Hitchin section} and Definition~\ref{SO(n,n) compatible condition}). By Lemma~\ref{isometric} and Lemma~\ref{hn=1}, we have
\begin{align*}
     \det \qty(h|_{F_{n+1}})= \prod_{k=1}^{n} H_k \cdot H_{n+1} = \prod_{k=1}^{n} H_k=\det \qty(h|_{F_{n}}).
\end{align*}
\end{proof}
Let $\theta (\symbfit{q})= \theta \dd{z} $.
Under the $h$-orthogonal basis $\symbfit{v} ^{\perp}=\lbrace v_1^{\perp},\dots,v_{2n}^{\perp}\rbrace$, the matrix representation of $\theta$ is
\begin{align*}
   \begin{pNiceMatrix}[columns-width=0.2cm]
        a_{1,\,1} & \NotEmpty& \NotEmpty  & \NotEmpty &\NotEmpty  & \NotEmpty &\NotEmpty & \NotEmpty   &a_{1,\,2n} \\
        \theta_1 & \NotEmpty     &  &  &  &  &  &   & \NotEmpty\\
        &      &  &  &  &  &  &   &\NotEmpty\\
        &      & \theta_{n-1}&  &  &  &  &   & \NotEmpty\\
        &      &  & \gamma &  &  &  &  & \NotEmpty \\
        &      &  & \theta_n & \gamma' &   &  & &\NotEmpty\\
        &      &  &  &  & \theta_{n+1} &  &  & \NotEmpty\\
        &      &  &  &  &  &   &    &\NotEmpty \\
        &      &  &  &  &  &   & \theta_{2n-2}& a_{2n,\,2n} 
        \CodeAfter 
        \line{2-1}{4-3} \line{7-6}{9-8} 
        \line{1-1}{9-9} \line{1-2}{8-9}
        \line{1-3}{7-9} \line{1-4}{6-9}
        \line{1-5}{5-9} \line{1-6}{4-9} 
        \line{1-7}{3-9} \line{1-8}{2-9} 
        \line{1-1}{1-9} 
        \line{1-9}{9-9}
        \end{pNiceMatrix},
\end{align*}
\begin{lemma} \label{error term gamma is zero}
    When $n$ is even, $\gamma=0$; when $n$ is odd, $\gamma'=0$.
\end{lemma}
\begin{proof}
   Note that $\gamma =0$ iff the projection of $\theta(v_n^{\perp})$ onto $L_{n+1}$ is zero. Let $v_n'=v_n-v_n^{\perp} \in F_{n-1}$. Then $\theta (v_n^{\perp})= \theta (v_n)-\theta(v_n')$. From \eqref{natural Higgs field}, we have $\theta(v_n)-v_{n+2} \in F_{n}$ and $\theta (v_n') \in F_{n}$. Since the projection onto $L_{n+1}$ kills any vector in $F_n$, the projection of $\theta (v_n^{\perp})$ coincides with that of $v_{n+2}$. Consequently, it suffices to verify that the projection of $v_{n+2}$ onto $L_{n+1}$ is zero. 
   
   Suppose $n$ is even. Then $K^{n-1} \oplus K^{n-3} \oplus \dots \oplus K\oplus K^{-1} \oplus \dots \oplus K^{3-n} \oplus K^{1-n}$ is $h$-orthogonal to $K^{n-2} \oplus K^{n-4} \oplus \dots \mathcal{O} \oplus \dots \oplus K^{4-n}\oplus K^{2-n} \oplus \mathcal{O}'$, which is a direct consequence of the compatibility condition (see Definition~\ref{definition of Higgs bundles in the Hitchin section} and Definition~\ref{SO(n,n) compatible condition}). 
   Define
   \begin{align*}
    W_1 &= K^{n-2} \oplus K^{n-4} \oplus \dots \oplus \mathcal{O}, \\
    W_2 &= K^{n-1} \oplus K^{n-3} \oplus \dots \oplus K, \\
    U_1 &= \mathcal{O}',\\
    U_2 &= K^{-1}.
   \end{align*}
Note that $F_{n+1}= W_1 \oplus W_2 \oplus U_1$. We have shown that $W_1 \oplus U_1$ is $h$-orthogonal to $W_2\oplus U_2$ for even $n$. By Lemma~\ref{easy but useful lemma}, for $v_{n+2} \in K^{-1} =U_2$, we have $v_{n+2}-v_{n+2}^{\perp} \in W_2 \subset F_{n-1}$. Therefore, $v_{n+2}$ has zero projection onto $L_{n+1}$.

For odd $n$, $\gamma'=0$ can be proved by a similar argument. Therefore we omit the proof.
\end{proof}

\section{Existence of harmonic metrics}\label{Existence of harmonic metric}

In this section, we prove the existence of harmonic metrics on Higgs bundle $(\mathbb{K}_{\mathrm{SO}_0(n,n)}, \theta (\symbfit{q}))$ for any $\symbfit{q}$ on a non-compact hyperbolic Riemann surface $X$. See Proposition~\ref{h_X} for the detailed description of $h_X$ in this theorem.

\begin{theorem}\label{existencs of harmonic metric}
    Let $X$ be a non-compact hyperbolic Riemann surface and $K$ be its canonical line bundle. Then for any $\symbfit{q} = (q_{1},\dots, q_{n-1}, q_{n}) \in \mathop{\bigoplus} \limits_{i=1}^{n-1} H^{0}(X,K^{2i}) \bigoplus H^{0}(X,K^{n}) $, there exists a harmonic metric $h$ of   $(\mathbb{K}_{\mathrm{SO}_0(n,n)}, \theta (\symbfit{q}))$ such that 
    \begin{itemize}
        \item $h$ is compatible with the $\mathrm{SO}_0(n,n)$-structure,
        \item $h$ weakly dominates $h_X$.
    \end{itemize}
\end{theorem}

\subsection{Notation}
For an $m \times m $ matrix $A=(A_{i,j})$, define its Euclidean norm and sup‑norm respectively by
\[
\abs{A}= \sqrt{\sum_{i,j=1}^{m} \abs{A_{i,j}}}, \quad \abs{A}_{\infty}= \sup_{1 \leq i,j \leq m}\abs{A_{i,j}}.
\]
Let $V$ be a complex vector space  with a fixed basis $\mathbf{e}=\left(e_1, \dots, e_{m}\right)$. Given a Hermitian metric $h$ of $V$, apply the Gram-Schmidt process to the basis $\mathbf{e}$ to obtain an orthonormal basis $\mathbf{v}(h)=\left(v_1(h), \dots, v_{m}(h)\right)$. 
Let $P(h)=\left(P(h)_{j,\,k}\right)$ be the matrix determined by $\mathbf{v}=\mathbf{e} P(h)$. Then $P(h)$ is upper triangular, i.e.,
\[
v_k(h)=\sum_{j \leq k} P(h)_{j,\,k} e_j.
\]
The inverse matrix of $P(h)$ is denoted by $P^{-1}(h)=\left(P^{-1}(h)_{j,\,k}\right)$. With respect to the frame $\symbfit{e}$, the metric $h$ is represented by the matrix $h(\symbfit{e})=\trans{\qty(P(h)^{-1})} \cdot \overline{P(h)^{-1}}$. Consider an endmorphism $f$ of $V$. Let $A=(A_{i,j})$ be the matrix determined by $f(e_j)= \sum_{i}A_{i,\,j}e_{i}$. 
Then the Eucliden norm of $f$ with respect to the metric $h$ is
\[
 \abs{f}_h=\abs{P(h)^{-1}AP(h)}.
\] 
We use a similar notation for a vector bundle equipped with a frame and a Hermitian metric.

\subsection{Proof of the main theorem}

We first recall a fundamental result about harmonic metrics, i.e., a priori estimate for Higgs fields. 

For $R > 0$, let $\Delta(R)$ be the open disk $\{z \in \mathbb{C} \mid \abs{z} < R\}$. Consider a Higgs bundle $(E,\theta)$ of rank $m$ on $\Delta(R)$ with a harmonic metric $h$. The Higgs field $\theta$ can be expressed as $f \dif z$, where $f$ is a holomorphic endomorphism of $E$. 
\begin{proposition}\cite[Proposition 2.1]{Mochizuki_2016} \label{bounds of theta}
    Let $M$ be the maximal of the eigenvalues of $f$ over $\Delta(R)$. Fix $0 < r < R$. Then there exist $C_1,C_2 >0$ depending only on $m, r, R$ such that
    \begin{align}
        \abs{f}_{h} \leq C_1 M + C_2
    \end{align}
    holds on $\Delta(r)$.
\end{proposition}

The following proposition, which is motivated by \cite[Proposition 2.4]{li2023higgs}, is the key point in the proof of the existence theorem.

\begin{proposition}\label{estimate P pro}
Let $V$ be a $2n$ dimensional complex vector space, equipped with a non-degenerate symmetric pairing $C$. Fix a basis $\symbfit{e}$. Assume the matrix form of $C$ with respect to $\symbfit{e}$ is the same as \eqref{C natural base}. Given a Hermitian metric $h$ which is compatible with $C$. Set $P \coloneq P(h)$.

Let $A$ be a $2n \times 2n$ matrix with
\begin{alignat*}{10}
A_{j,\,k}&=0,  & j &>k+1,    & &\textit{except for} \quad A_{n+2,\,n}, \\
A_{k+1,\,k} &\neq 0, \quad & 1 \leq k &\leq 2n-1,  \quad & &\textit{except for} \quad A_{n+1,\,n}, A_{n+2,\,n+1}.
\end{alignat*}
Set
\[
\left| A \right|_{\infty} \coloneqq \max _{j,\,k}\left|A_{j,\,k}\right|, \quad \widetilde{\left| A \right|_{\infty}} \coloneqq \max \{\max _{k \neq n,\, n+1}\left|\left(A_{k+1,\,k}\right)^{-1}\right|,\abs{\left(A_{n+2,\,n}\right)^{-1}}\}.
\]
Suppose $\left|P^{-1} A P\right| \leq c,\left|P_{1,\,1}\right| \geq d$. Then there exists a constant $C=C(|A|,|\widetilde{A}|, c, d)$ such that
\[
\left|\left(P^{-1}\right)_{i,\,j}\right|+\left|P_{i,\,j}\right| \leq C .
\]
 \end{proposition}

The proof of the above proposition is postponed to the following subsection since it's too cumbersome. We now employ this proposition to obtain the crucial $C^0$ estimate.
 
Let $X$ be a non-compact hyperbolic Riemann surface.  Recall that the unique complete K\"ahler hyperbolic metric $g_X$ induces a natural harmonic metric $h_X$ of 
$(\mathbb{K}_{\mathrm{SO}_0(n,n)}, \theta (\symbf{0}))$. Fix $\symbfit{q} = (q_{1},\dots, q_{n-1}, q_{n}) \in \mathop{\bigoplus} \limits_{i=1}^{n-1} H^{0}(X,K^{2i}) \bigoplus \allowbreak H^{0}(X,K^{n})$.
For any Hermitian metric $h$ of $\mathbb{K}_{\mathrm{SO}_0(n,n)}$ which is compatible with the $\mathrm{SO}_0(n,n)$-structure, let $s(h_X,h)$ be the automorphism of $\mathbb{K}_{\mathrm{SO}_0(n,n)}$ determined by $h = h_X \cdot s(h_X,h)$. 
\begin{proposition} \label{crucial estimate}
    Consider two relatively compact connected open subsets $K,K'$ of $X$ such that $\overline{K} \subset K'$. Then there exists a positive constant $C$ depending only on $K,K'$ and $\symbfit{q}$ such that 
    \begin{align*}
        \abs*{s(h_X,h)}_{h_X} + \abs*{s^{-1}(h_X,h)}_{h_X} \leq C
    \end{align*}
    holds on $K$ for any harmonic metric $h$ on $(\mathbb{K}_{\mathrm{SO}_0(n,n)}, \theta (\symbfit{q}))|_{K'}$ which weakly dominates $h_X$ and is compatible with the $\mathrm{SO}_0(n,n)$-structure.
\end{proposition}
\begin{proof} 
 Take a holomorphic coordinate chart $(U,z)$ around a point $p \in K$. On $U$, the bundle $\mathbb{K}_{\mathrm{SO}_0(n,n)}$ admits the natural local holomorphic frame $\symbfit{e}=\qty{\dif z^{n-1},\dots,1,1 \,\dots,\dif z^{1-n}}$. Express $\theta(\symbfit{q})$ as $f \dif z$. Let $A$ be the matrix determined by $f(e_j)= \sum_{i}A_{i,\,j}e_{i}$. For any harmonic metric $h$ on $(\mathbb{K}_{\mathrm{SO}_0(n,n)}, \theta (\symbfit{q}))|_{K'}$, by applying the Gram–Schmidt process to $\symbfit{e}$, we obtain an orthonormal basis $\mathbf{v}=\mathbf{e} P(h)$. Since $K$ and $K'$ are both relatively compact, local estimates extend globally after passing to a finite open cover. For this reason, we will make no distinction between local and global assertions in the remainder of the proof.
   
   Proposition~\ref{bounds of theta} implies that there exists a positive constant $C_1$ depending only on $K, K'$ and $\symbfit{q}$ such that 
   \begin{align*} 
       \abs{f}_{h}= \abs{P^{-1}\qty(h)AP\qty(h)}\leq C_1 
   \end{align*}
   holds on $K$ for any harmonic metric $h$ of $(\mathbb{K}_{\mathrm{SO}_0(n,n)}, \theta (\symbfit{q}))|_{K'}$.
   
  Suppose $h$ weakly dominates $h_X$. Then we have
   \begin{align*}
       \abs{P\qty(h)_{1,\,1}} \geq \abs{P{\qty(h_X)}_{1,\,1}}.
   \end{align*}

   Suppose $h$ is compatible $\mathrm{SO}_0(n,n)$-structure. Recalling the particular form \eqref{natural Higgs field} of $\theta(\symbfit{q})$, it's easy to verify that $A$ and $P(h)$ satisfy the conditions in Proposition~\ref{estimate P pro}. So we obtain that there exists a positive constant $C_3$ depending only on $K, K'$ and $\symbfit{q}$ such that
    \begin{align*}
        \abs{P^{-1}\qty(h)} + \abs{P\qty(h)} \leq C_3
    \end{align*}
    holds on $K$ for any harmonic metric $h$ of $(\mathbb{K}_{\mathrm{SO}_0(n,n)}, \theta (\symbfit{q}))|_{K'}$ which weakly dominates $h_X$ and is compatible with the $\mathrm{SO}_0(n,n)$-structure.
   With respect to the frame $\symbfit{e}$, $s(h_X,h)$ is expressed as matrix 
   \begin{align*}
   S=
   P\qty(h_X)\overline{\trans{P\qty(h_X)}}
   \overline{\trans{P^{-1}\qty(h)}}P^{-1}\qty(h).
   \end{align*}
   Then 
   \begin{align*}
       \abs{s(h_X,h)}_{h_X} & = \abs{P^{-1}\qty(h_X)SP\qty(h_X)} \leq C_4, \\
       \abs{s^{-1}(h_X,h)}_{h_X} & = \abs{P^{-1}\qty(h_X)S^{-1}P\qty(h_X)} \leq C_4
   \end{align*}
   holds on $K$, where $C_4$ is a positive constant depending only on $K, K'$ and $\symbfit{q}$.
\end{proof}
\begin{proof}[Proof of Theorem ~\ref{existencs of harmonic metric}]
    Let $X_{i}$ be a smooth exhaustion family of non-compact hyperbolic Riemann surface $X$. Proposition~\ref{boundary harmonic metric} ensures that there exists the harmonic metric $h_i$ of $(\mathbb{K}_{\mathrm{SO}_0(n,n)}, \theta (\symbfit{q}))|_{X_i}$ such that $h_i|_{\partial X_i} = h_X|_{\partial X_i}$. By Proposition~\ref{h is compatible with so} and Proposition~\ref{h_X}, $h_i$ is compatible with the $\mathrm{SO}_0(n,n)$-structure.
    By Proposition~\ref{h dominates}, $h_i$ weakly dominates $h_X$. Proposition~\ref{convergent to harmonic metric} and Proposition~\ref{crucial estimate} imply that the sequence 
    $h_i$ contains a convergent subsequence. So we obtain a harmonic metric $h$ of $(\mathbb{K}_{\mathrm{SO}_0(n,n)}, \theta (\symbfit{q}))$. It is obvious that $h$ satisfies the compatibility condition and the weak domination property.
\end{proof}

\subsection{Proof of Proposition ~\ref{estimate P pro}}

Let's start with an easy lemma, which can help us deduce some equalities about the elements of $P$.
\begin{lemma}\label{block form identities}
Let $\widetilde{P}$ and $\widetilde{Q}$ be non-degenerate upper and lower triangular block square matrices respectively and 
\begin{align*}
C =
\begin{pmatrix}
& & & & E_{m_1} \\
& & &\iddots & \\
& &E_{m_n} & & \\
&\iddots& & &  \\
E_{m_{2n-1}}& && &
\end{pmatrix},
\end{align*}
where $E_{m_k}$ is the identity matrix with rank $m_k$ for $1 \leq k \leq 2n-1$. Assume the partitionings of $\widetilde{P}$ and $\widetilde{Q}$ are the same with $C$ and $m_k=m_{2n-k}$.
If $T \coloneqq \widetilde{Q}^{-1}C\widetilde{P}$ is a unitary matrix, then
\begin{align*}
\widetilde{P}_{i,\,j} \qty(\widetilde{P}_{k,\,j})^{*}=\widetilde{Q}_{2n-i,\,2n-j}\qty(\widetilde{Q}_{2n-k,\,2n-j})^{*}.
\end{align*}
\end{lemma}
\begin{proof}
Using the special form of $C$, $\widetilde{Q}$ and $\widetilde{P}$, it's easy to verify that $T_{i,\,j}=0$ for $i+j<2n$, i.e., 
\begin{align*}
T = 
\begin{pmatrix}
& & & & T_{1,\,2n-1}  \\
& & &\iddots & \\
& &T_{n,\,n}& &\vdots \\
&\iddots&& &  \\
T_{2n-1,\,1}& &\dots& &T_{2n-1,\,2n-1}   
\end{pmatrix}.
\end{align*}

Since $T$ is unitary,
\begin{align*}
T T^{*} & = 
\begin{pmatrix}
& & & & T_{1,\,2n-1}  \\
& & &\iddots & \\
& &T_{n,\,n}& &\vdots \\
&\iddots& & &  \\
T_{2n-1,\,1}& &\dots& &T_{2n-1,\,2n-1}   
\end{pmatrix}
\begin{pmatrix}
& & & & T_{2n-1,\,1}^{*}  \\
& & &\iddots & \\
& &T_{n,\,n}^{*}& &\vdots \\
&\iddots& & &  \\
T_{1,\,2n-1}^{*}& &\dots& &T_{2n-1,\,2n-1}^{*}   
\end{pmatrix}  \\
& = \begin{pmatrix}
E_{m_1} & & & & \\
& \ddots & & & \\
& & E_{m_n} & & \\
& & & \ddots & \\
& & & & E_{m_{2n-1}}
\end{pmatrix}.
\end{align*}
Using the similar method in Lemma~\ref{simplify C H in orthogonal base}, it's easy to prove that $T_{k,\,2n-k}$ is unitary for $1 \leq k \leq 2n-1$ and other elements are zero. Recalling that $\widetilde{Q}T=C\widetilde{P}$, we get $\widetilde{P}_{i,\,j}=\left( C \widetilde{P}\right)_{2n-i,\,j}=\left( \widetilde{Q}T \right)_{2n-i,\,j}= \widetilde{Q}_{2n-i,\,2n-j} T_{2n-j,\,j}$. Hence
\begin{align*}
\widetilde{P}_{i,\,j} \qty(\widetilde{P}_{k,\,j})^{*}=\widetilde{Q}_{2n-i,\,2n-j} T_{2n-j,\,j}\qty(T_{2n-j,\,j})^{*}\qty(\widetilde{Q}_{2n-k,\,2n-j})^{*} = \widetilde{Q}_{2n-i,\,2n-j}\qty(\widetilde{Q}_{2n-k,\,2n-j})^{*}.
\end{align*}
\end{proof}
Next, we use Lemma~\ref{block form identities} to derive the following equations. Firstly, we partition the matrix $P$ in Proposition~\ref{estimate P pro} as indicated in \eqref{partitioning of C}, and obtain some equations by Lemma~\ref{block form identities}. Finally, we convert them into equations about the original matrix elements.
\begin{lemma} \label{lem3}
Let $P$, $C$ satisfy the conditions in Proposition ~\ref{estimate P pro}, then 
\begin{align}
\abs{P_{i,\,i}} & = \abs{P^{-1}_{2n+1-i,\,2n+1-i}}, \quad 1 \leq i \leq 2n,  \label{eq1}\\
\abs{P_{n,\,j}} & = \abs{\qty(P^{-1})_{2n+1-j,\,n}}, \quad j \neq n,\,n+1, \label{eq2} \\
\abs{\qty(P^{-1})_{n,\,n}}^2 & 
= \abs{P_{n,\,n}}^{2} + \abs{P_{n,\,n+1}}^2. \label{eq3}
\end{align}
\end{lemma}
\begin{proof}
By \eqref{comp}, we get 
\[
P \overline{\trans{P}} C = C \trans{\qty(P^{-1})} \overline{P^{-1}}.
\]
To simplify the notation, we use $Q$ to replace lower triangular matrix $\trans{\qty(P^{-1})}$. Then the above equality becomes 
\[
P P^{*} C=C Q Q^{*},
\]
i.e.
\[
CP P^{*}C^{*}= Q Q^{*}.
\]
Here we use $C=C^{-1}$.

This just means that $Q^{-1} CP$ is a unitary matrix. Recall the matrix $C$ presented in \eqref{C natural base}. We partition $C$ as described in \eqref{partitioning of C}. Let $\widetilde{C}$ be the corresponding block matrix, i.e.,
\[
\widetilde{C}=
\begin{pNiceMatrix}[
  first-row,
  last-col,
  columns-width = auto,
]
   1 & \dots & n-1 & n & n+1 & \dots  & 2n-1 \\
  & & & & & &1 & 1 \\
  & & & & & \Iddots & & \Vdots \\
  & & & & 1 & & & n-1 \\
  & & & E_2 & & & & n \\
  & & 1 & & & & & n+1 \\
  & \Iddots & & & & & & \Vdots \\
  1 & & & & & & & 2n-1 \\
\end{pNiceMatrix}
\]
We can partition $P$ and $Q$ similarly. Denote the corresponding block matrices as $\widetilde{P}$ and $\widetilde{Q}$ respectively.
By Lemma~\ref{block form identities}, we have
\begin{align}\label{PQ relations}
    \widetilde{P}_{i,\,j} \qty(\widetilde{P}_{k,\,j})^{*}=\widetilde{Q}_{2n-i,\,2n-j}\qty(\widetilde{Q}_{2n-k,\,2n-j})^{*}.
\end{align}

To prove the three equalities in the lemma sequentially, we convert the above block matrix expressions into explicit matrix expressions. Recall that $Q=\trans{\qty(P^{-1})}$. For an integer $i \neq n$, define a new integer $i'$ as follows: 
$i'=i$ if $i < n$, and $i'=i+1$ if $i >n$.
\begin{itemize}
\item For $i \neq n$, $\widetilde{P}_{i,\,i}=P_{i',\,i'}$ and $\widetilde{Q}_{i,\,i}=Q_{i',\,i'}=\qty(P^{-1})_{i',i'}$. By \eqref{PQ relations}, we get
\begin{align*}
P_{i,\,i} \overline{P}_{i,\,i}=\left(P^{-1}\right)_{2n+1-i,\,2n+1-i}\overline{\left(P^{-1}\right)}_{2n+1-i,\,2n+1-i}
\end{align*}
for $i \neq n,\,n+1$. Note that $\left(P^{-1}\right)_{i,\,i}=P^{-1}_{i,\,i}$, since $P$ is an upper triangular matrix. Then 
\begin{align*}
    \abs{P_{i,\,i}}^2 & = \abs{P^{-1}_{2n+1-i,\,2n+1-i}}^2
\end{align*}
for $i\neq n,\,n+1$. 
So \eqref{eq1} holds for $i \neq n, n+1$. From the compatibility condition in Proposition~\ref{estimate P pro}, we know that $\left|\det \left( P \right) \right|=1$. Using the equality \eqref{eq1} for $i\neq n,n+1$, which we have proved , we obtain
\begin{align*}
  1=\left|\det \left( P \right) \right|=\prod_{i=1}^{2n} \abs{P_{i,\,i}}=\abs{P_{n,\,n}}\abs{P_{n+1,\,n+1}}.  
\end{align*}
The second equality follows from the fact that $P$ is upper triangular.
This proves \eqref{eq1} for both $i=n$ and $n+1$.
\item For $j \neq n$, $\widetilde{P}_{n,\,j}=\left(\begin{smallmatrix}
      P_{n,\,j+1} \\
P_{n+1,\,j+1}
\end{smallmatrix}
\right)$, $\widetilde{Q}_{n,\,j}=\left(\begin{smallmatrix}
     Q_{n,\,j+1} \\
Q_{n+1,\,j+1}
\end{smallmatrix}
\right)=\left(\begin{smallmatrix}
     P_{j+1,\,n} \\
P_{j+1,\,n+1}   
\end{smallmatrix}
\right)$. 
Let $i=k=n$ in the equality \eqref{PQ relations}. Then we get 
\begin{align*}
\begin{pmatrix}
P_{n,\,j+1} \\
P_{n+1,\,j+1}
\end{pmatrix}
\begin{pmatrix}
\overline{P}_{n,\,j+1} & \overline{P}_{n+1,\,j+1}
\end{pmatrix}
=
\begin{pmatrix}
\left(P^{-1}\right)_{2n-j,\,n} \\
\left(P^{-1}\right)_{2n-j,\,n+1}
\end{pmatrix}
\begin{pmatrix}
\overline{\left(P^{-1}\right)}_{2n-j,\,n} &
\overline{\left(P^{-1}\right)}_{2n-j,\,n+1}
\end{pmatrix}. 
\end{align*}
Extracting $(1,1)$ entries from both sides of the above matrix equality, we obtain the equality \eqref{eq2} for $j>n+1$. For $j<n$, $P_{n,\,j}= \left(P^{-1}\right)_{2n+1-j,\,n}=0$ since $P$ is an upper triangular matrix. The equality \eqref{eq2} holds trivially in this case.
\item Note that 
\begin{align*}
    \widetilde{P}_{n,\,n}=
 \begin{pmatrix}
P_{n,\,n} & P_{n,\,n+1}   \\
    & P_{n+1,\,n+1}
\end{pmatrix},\widetilde{Q}_{n,\,n}=
\begin{pmatrix}
\left(P^{-1}\right)_{n,\,n} &    \\
\left(P^{-1}\right)_{n,\,n+1} & \left(P^{-1}\right)_{n+1,\,n+1}
\end{pmatrix}.
\end{align*}
Let $i=j=k=n$ in the equality \eqref{PQ relations}. Then 
\begin{align*}
 \! \! \! \! \! \! \!  
&\begin{pmatrix}
P_{n,\,n} & P_{n,\,n+1}   \\
    & P_{n+1,\,n+1}
\end{pmatrix}
\begin{pmatrix}
\overline{P}_{n,\,n}     &   \\
\overline{P}_{n,\,n+1}   & \overline{P}_{n+1,\,n+1}
\end{pmatrix} \! \! 
= \! \!
\begin{pmatrix}
\left(P^{-1}\right)_{n,\,n} &    \\
\left(P^{-1}\right)_{n,\,n+1} & \left(P^{-1}\right)_{n+1,\,n+1}
\end{pmatrix}
\begin{pmatrix}
\overline{\left(P^{-1}\right)}_{n,\,n} &   \overline{\left(P^{-1}\right)}_{n,\,n+1}  \\
& \overline{\left(P^{-1}\right)}_{n+1,\,n+1}
\end{pmatrix}.
\end{align*}
 Extracting $(1,1)$ entries from both sides of the above matrix equality, we obtain the equality \eqref{eq3}.
\end{itemize} 
\end{proof}

\begin{proof}[Proof of Proposition ~\ref{estimate P pro}]
We use the method in \cite[Proposition 2.4]{li2023higgs} to prove this proposition.

 The proof will be divided into three steps.
\begin{enumerate}
\setlength{\itemindent}{2em} 
\item[Step 1] We first estimate diagonal elements of $P$ and $P^{-1}$.

When $i \neq n,\,n+1$, $\left( P^{-1}AP \right)_{i+1,\,i}=P^{-1}_{i+1,\,i+1}A_{i+1,\,i}P_{i,\,i}$ and $\left( P^{-1}AP \right)_{n+2,\,n}=P^{-1}_{n+2,\,n+2}A_{n+2,\,n}P_{n,\,n}$. Thus 
\begin{align*}
\abs{P_{i,\,i}} & \leq C_1 \abs{P_{i+1,\,i+1}}, \\
\abs{P_{n,\,n}} & \leq C_2 \abs{P_{n+2,\,n+2}}.
\end{align*}
So we get 
\begin{align*}
C_3^{-1}\abs{P_{1,\,1}} \leq \abs{P_{i,\,i}} \leq C_3\abs{P_{2n-i+1,\,2n-i+1}} = C_3 \abs{P^{-1}_{i,\,i}}.
\end{align*}
Here we use the equality \eqref{eq1} in Lemma~\ref{lem3}. Hence
\begin{align*}
C_{4}\leq \abs{P_{i,\,i}} \leq C_{5},   
\end{align*}
here $C_4$ and $C_5$ are positive number which only depends $|\widetilde{A}|$, $c$ and $d$.

The cases $i=n,\, n+1$, follow from the inequalities $\abs{P_{n-1,\,n-1}} \leq C_1 \abs{P_{n,\,n}}$, $\abs{P_{n,\,n}} \leq C_2 \abs{P_{n+2,\,n+2}}$ and the equality \eqref{eq1} in Lemma ~\ref{lem3}.
\item[Step 2] We have the following lemmas.
\begin{lemma}\cite[Lemma 2.5]{li2023higgs}\label{expression for P inverse}
 For an invertible upper triangular matrix $P$ of rank $r$, the following properties hold.
 \begin{itemize}
    \item $P^{-1}$ is an upper triangular matrix.
    \item $\qty(P^{-1})_{i,\,i}=P_{i,\,i}^{-1}$ for $1 \leq i \leq r$.
    \item For $1 \leq i< j \leq r$ and $m \in \mathbb{Z}_{\geq 1}$, let $\mathcal{S}_m(i,j)$ denote the set of $\mathbf{i}=\left(i_0, i_1, \dots, i_m\right) \in \mathbb{Z}_{\geq 1}^{m+1}$ such that $i_0=i<i_1<\dots<i_m=j$. Then
\[
\left(P^{-1}\right)_{i,\,j}=\sum_{m \geq 1} \sum_{\mathbf{i} \in \mathcal{S}_m(i,\,j)}(-1)^m \prod_{p=0}^m\left(P_{i_p,\, i_p}\right)^{-1} \prod_{p=0}^{m-1} P_{i_p,\,i_{p+1}} .
\]
 \end{itemize}
\end{lemma} 
\begin{lemma}\cite[Lemma 2.6]{li2023higgs} \label{lem4}
Let $P$ be an invertible upper triangular matrix of rank $r$. Suppose
\begin{align*}
    \abs{\left(P_{i,\,i}\right)^{-1}} &\leq B_1 \quad \text{for all} \quad 1 \leq i \leq r, \\
    \left|P_{i,\,i+t}\right| &\leq B_2 \quad \text{for all}0 \leq t \leq t_0 \quad \text{and} \quad 1 \leq i \leq r-t.
\end{align*}
Then
\[
\left|\left(P^{-1}\right)_{i,\,i+t}\right| \leq 2^r \sum_{m=1}^{t+1} B_1^{m+1} B_2^m
\]
for all $0 \leq t \leq t_0$ and $1 \leq i \leq r-t$.
\end{lemma}

\item[Step 3] Estimate other elements.

It is enough to estimate $P_{i,\,j}$ with $i \leq j$ by Lemma~\ref{lem4}. We proceed by induction on $t=j-i$. That is, assuming the estimates hold for all pairs $(i,j)$ with $0\leq j-i \leq t$, we need to prove them for all pairs $(i,j)$ with $j-i=t+1$. This requires a stronger inductive hypothesis. Namely, we also assume  $\abs{\left(P^{-1}\right)_{n-t+1,\,n+2}}\leq C'$ for some positive constant $C'$. Moreover, the proof for the case
$t+1$ itself involves an induction on the index $i$.

We begin by verifying that all estimates hold at 
$t=0$. We have proved the estimates for all diagonal entries $P_{i,\,i} \, (1 \leq i \leq 2n)$ in Step 1. For $\qty(P^{-1})_{n+1,\,n+2}$, we have the equation $\left( P^{-1}AP \right)_{n+1,\,n}=\qty(P^{-1})_{n+1,\,n+2}A_{n+2,\,n}P_{n,\,n}$. Note that $A_{n+2,\,n} \neq 0$. Together with the known upper bounds on $\left( P^{-1}AP \right)_{n+1,\,n}$ and $P_{n,\,n}^{-1}$, we obtain the desired upper bound on $\qty(P^{-1})_{n+1,\,n+2}$. Thus, we have proved
\begin{align*}
    &\left|P_{i,\,i}\right|  \leq C_6 \quad \text{for all} \quad 1 \leq i \leq 2n, \\
&\abs{\qty(P^{-1})_{n+1,\,n+2}}  \leq C_7
\end{align*}
for some positive constants $C_6,C_7$.
Now, assume that
\begin{equation}\label{ass}
\begin{aligned} 
\left|P_{i,\,i+t}\right| & \leq C_8, \\
\abs{\qty(P^{-1})_{n-t+1,\,n+2}} & \leq C_9
\end{aligned}
\end{equation}
for all $0 \leq t \leq t_0,\,1 \leq i \leq 2n-t$.
We are going to show that there exist constants $C_{10}$ and $C_{11}$ such that $\left|P_{i,\,i+t_0+1}\right| \leq C_{10}$ for all $1 \leq i \leq 2n-t_0-1$ and $\abs{\qty(P^{-1})_{n-t_0,\,n+2}} \leq C_{11} $.

Set
\begin{align*}
    \mathcal{T}\left(i, t_0\right) 
    & =
    \left\{(l,k) \mid i-1 \leq l-1 \leq k \leq i+t_0\right\}, \\
    \mathcal{T}^{\prime}\left(i, t_0\right)
    & =
    \mathcal{T}\left(i, t_0\right) \backslash\left\{(i, i-1),\left(i+t_0+1, i+t_0\right)\right\}.
\end{align*}
For any $(k,l) \in \mathcal{T}^{\prime}\left(i,t_0\right)$, we have $\ell-i \leq t_0$ and $i+t_0-k \leq t_0$, we obtain
\begin{align*}
&\left( P^{-1}AP \right)_{i,\,i+t_0} \\
= & \sum_{(l,k) \in \mathcal{T}{\left(i,t_0\right)}}\left(P^{-1}\right)_{i,l} A_{l,\,k} P_{k,\,i+t_0}  + \qty(P^{-1})_{i,\,n+2}A_{n+2,\,n}P_{n,\,i+t_0}\\
= & \sum_{(l,k) \in \mathcal{T}^{\prime}\left(i\, t_0\right)}\left(P^{-1}\right)_{i,l} A_{l,\,k} P_{k,\,i+t_0} +\left(P^{-1}\right)_{i,\,i+t_0+1} A_{i+t_0+1,\,i+t_0} P_{i+t_0,\, i+t_0}+\left(P^{-1}\right)_{i,\,i} A_{i,\,i-1} P_{i-1,\,i+t_0}\\
&+ \qty(P^{-1})_{i,\,n+2}A_{n+2,\,n}P_{n,\,i+t_0} \\
= & \sum_{(l,k) \in \mathcal{T}^{\prime}\left(i, t_0\right)}\! \! \! \! \! \! \left(P^{-1}\right)_{i,\,l} A_{l,\,k} P_{k,\,i+t_0}\!+\!\sum_{m \geq 2} \sum_{\mathbf{i} \in S_m\left(i,i+t_0+1\right)}(-1)^m A_{i+t_0+1,\,i+t_0} P_{i+t_0,\,i+t_0} \prod_{p=0}^m\left(P_{i_0,\,i_0}\right)^{-1} \!\prod_{p=0}^{m-1} P_{i_p,\,i_{p+1}} \\
& -P_{i,\,i+t_0+1} P_{i,\,i}^{-1} P^{-1}_{i+t_0+1,\,i+t_0+1} A_{i+t_0+1,\, i+t_0} P_{i+t_0,\,i+t_0}\!+\!\left(P^{-1}\right)_{i,\,i} A_{i,\,i-1} P_{i-1,\, i+t_0}\!+\!\left(P^{-1}\right)_{i,\,n+2}A_{n+2,\,n}P_{n,\,i+t_0}.
\end{align*}
Here we agree that $A_{1,\,0}=A_{2n+1,\,2n}=0$, $\left( P^{-1} \right)_{i,\,2n+1}=P_{i,\,2n+1}=P_{0,\,1+t_0}=0$ and $P_{2n+1,\,2n+1}=1$.The third equality follows by replacing $\left(P^{-1}\right)_{i,\,i+t_0+1}$ in the second term of the second equality with the expression from Lemma~\ref{expression for P inverse}.

To make the notation more concise, we use $S\left(i,t_0\right)$ to represent the sum of the first two terms of the right-hand side of the last equality, i.e.,
\begin{equation}\label{long formula}
     \begin{aligned} 
     \left( P^{-1}AP \right)_{i,\,i+t_0} 
     = & S\left(i,t_0\right) - P_{i,\,i+t_0+1} P_{i,\,i}^{-1} P^{-1}_{i+t_0+1,\,i+t_0+1} A_{i+t_0+1,\, i+t_0}  P_{i+t_0,\,i+t_0}  \\
     &+ \left(P^{-1}\right)_{i,\,i} A_{i,\,i-1} P_{i-1,\, i+t_0}+ \qty(P^{-1})_{i,\,n+2}A_{n+2,\,n}P_{n,\,i+t_0}. 
     \end{aligned}
   \end{equation}
Notably, $S\left(i,t_0\right)$ only involves $A_{l,\,k}$, $\left(P^{-1}\right)_{i,\,i+t}$ and $P_{s,\,s+t}$  for $0 \leq t \leq t_0$ and $i \leq s \leq i+t_0$.

By the induction hypothesis \eqref{ass} and Lemma~\ref{lem4}, $\left|\left(P^{-1}\right)_{i,\,i+t}\right| \leq C_{12}$ for all $i$ and $0 \leq t \leq t_0$, i.e.,
\begin{align*}
\abs{S\left(i,t_0\right)} \leq C_{13}.
\end{align*}

Note that the last term $ P^{-1}_{i,\,n+2}A_{n+2,\,n}P_{n,\,i+t_0}$ appears only when $n-t_0 \leq i \leq n+2$ since $P$ and $P^{-1}$ are both upper triangular matrices.  We therefore prove the estimate by cases on 
$i$.
\begin{itemize}
\item $1\leq i \leq n-t_0-1$

If $t_0 \geq n-1$, there is nothing to prove. We may thus assume $t_0 \leq n-2$. Since $P^{-1}_{i,\,n+2}A_{n+2,\,n}P_{n,\,i+t_0}$ vanishes in this case, the long formula \eqref{long formula} becomes 
\begin{equation}\label{long formula s}
     \begin{aligned} 
     \left( P^{-1}AP \right)_{i,\,i+t_0} 
     = & S\left(i,t_0\right) - P_{i,\,i+t_0+1} P_{i,\,i}^{-1} P^{-1}_{i+t_0+1,\,i+t_0+1} A_{i+t_0+1,\, i+t_0}  P_{i+t_0,\,i+t_0}  \\
     &+ \left(P^{-1}\right)_{i,\,i} A_{i,\,i-1} P_{i-1,\, i+t_0}. 
     \end{aligned}
   \end{equation}
From the conditions in Proposition~\ref{estimate P pro}, we know that $A_{i+t_0+1,\, i+t_0}$ is nonzero. Using the formula \eqref{long formula s} for the case $i=1$, we first obtain an estimate for $P_{1,\,t_0+2}$. We now extend this estimate to $\abs{P_{i,\,i+t_0+1}}$ for all $1\leq i \leq n-t_0-1$. For the inductive step, assuming we have a bound on $P_{i-1,\, i+t_0}$, we obtain a bound on the third term in the formula \eqref{long formula s}. Hence we have a bound on the second term in the formula \eqref{long formula s}. Since $A_{i+t_0+1,\, i+t_0}$ is nonzero, we obtain a bound on $P_{i,\,i+t_0+1}$. Thus, by induction we obtain the bounds on $P_{i,\,i+t_0+1}$ for all $1\leq i \leq n-t_0-1$.

\item $i=n-t_0$

In this case, $A_{i+t_0+1,\, i+t_0}=A_{n+1,\,n}=0$. Hence the second term in the long formula \eqref{long formula} vanishes. Then
\begin{equation}\label{s long formula}
    \begin{aligned}
    \left( P^{-1}AP \right)_{n-t_0+1,\,n+1} 
    = & S\left(n-t_0+1,t_0\right) + \left(P^{-1}\right)_{n-t_0+1,\,n-t_0+1} A_{n-t_0+1,\,n-t_0} P_{n-t_0,\, n+1} \\
    & + \qty(P^{-1})_{n-t_0+1,\,n+2}A_{n+2,\,n}P_{n,\,n+1}. 
    \end{aligned}
\end{equation}
When $t_0=0$, $A_{n-t_0+1,\,n-t_0}=A_{n+1,\,n}$ also vanishes. Consequently, the formula used previously is no longer applicable. However, we can employ the equality
\begin{align*}
\qty(P^{-1}AP)_{n+2,\,n+1}=\qty(P^{-1})_{n+2,\,n+2}A_{n+2,\,n}P_{n,\,n+1}.
\end{align*}
Note that $A_{n+2,\,n} \neq 0 $. Combining the known bounds on $\qty(P^{-1}AP)_{n+2,\,n+1}$ and the diagonal entry $\qty(P^{-1})_{n+2,\,n+2}$, it follows that 
\begin{align*}
\abs{P_{n,\,n+1}} \leq C_{14}.
\end{align*}

When $1 \leq t_0 \leq n-1 $, $A_{n-t_0+1,\,n-t_0}$ never vanishes. The estimate for $P^{-1}_{n-t_0+1,\,n+2}A_{n+2,\,n}P_{n,\,n+1}$ is easily obtained from our induction hypothesis and the estimate for $P_{n,\,n+1}$. Then we get an estimate for $P_{n-t_0,\, n+1}$ from the formula \eqref{s long formula}.

When $t_0 \geq n$, there is nothing to prove.
\item $i=n-t_0+1$

By the induction hypothesis \eqref{ass} and Lemma ~\ref{lem4},
\begin{align*}
\abs{P_{n-t_0+1,\,n+2}} \leq C_{15}.
\end{align*}

\item $n-t_0+2 \leq i \leq n$

The last term of the formula \eqref{long formula} only involves $P_{s,\,s+t}$ for $t\leq t_0$, and $A_{i+t_0+1,\,i+t_0} \neq 0$. Hence we can obtain an estimate for $P_{i,\,i+t_0+1}$ by a similar induction method which has been used in the case $1 \leq i \leq n-t_0-1$.

\item $i=n+1$

The last term of \eqref{long formula} is $P^{-1}_{n+1,\,n+2}A_{n+2,\,n}P_{n,\,n+1+t_0}$. We already have estimates for $P_{n,\,n+1+t_0}$ (when $n-t_0+2 \leq i \leq n$) and $P_{n+1,\,n+2}$ (by the induction hypothesis \eqref{ass} and Lemma~\ref{lem4}). This gives a bound on the last term. If $t_0=0$, $A_{n+t_0+2,\,n+t_0+1}=A_{n+2,\,n+1}=0$. But the estimate for $P_{n+1,\,i+2}=P_{n+1,\,n+2}$ is still available from the above. For $t_0\geq 1$, $A_{n+t_0+2,\,n+t_0+1}$ never vanishes. Using \eqref{long formula}, we can estimate $P_{i,\,i+t_0+1}$. This proposition is now established via induction.

\item $i=n+2$

Since $A_{n+t_0+3,\,n+2+t_0}$ never vanishes, we only need to estimate
the last term of \eqref{long formula}. By the equality \eqref{eq2} in Lemma~\ref{lem3}, 
\begin{align*}
\abs{P_{n,\,n+2+t_0}} = \abs{\left(P^{-1}\right)_{n-t_0-1,\,n}}.
\end{align*} 
By the estimate for $P_{n-t_0-1,\,n}$ (obtained when $1 \leq i \leq n-t_0-1$) and Lemma~\ref{lem4}, we obtain a bound on $\left(P^{-1}\right)_{n-t_0-1,\,n}$.

\item $n+3 \leq i \leq 2n-t_0-1$

The last term of \eqref{long formula} vanishes, so we obtain the estimate by an inductive argument analogous to the one used for the range $1 \leq i \leq n-t_0-1$.
\end{itemize}
\end{enumerate}
Referring back to the induction hypothesis $\eqref{ass}$, it remains to estimate $\qty(P^{-1})_{n-t_0,\,n+2}$. Return to the formula \eqref{long formula} for $i=n-t_0$. The element $A_{i+t_0+1,\, i+t_0}={}A_{n+1,\,n}=0$, hence the second term in \eqref{long formula} vanishes. Then we have
\begin{align} \label{last long formula}
\left( P^{-1}AP \right)_{n-t_0,\,n} 
=  S\left(n-t_0,t_0\right) 
+ \left(P^{-1}\right)_{n-t_0,\,n-t_0} A_{n-t_0,\,n-t_0-1} P_{n-t_0-1,\, n}+ P^{-1}_{n-t_0,\,n+2}A_{n+2,\,n}P_{n,\,n}. 
\end{align}
By the induction hypothesis \eqref{ass}, we have an upper bound on $\qty(P^{-1})_{n-t_0-1,\,n}$. Since $A_{n-t_0,\,n-t_0-1}$ and $A_{n+2,\,n}$ never vanish, using \eqref{last long formula}, we obtain 
\[
\abs{\qty(P^{-1})_{n-t_0,\,n+2}} \leq C_{16}
\]
for some positive constant $C_{16}$.

We have proved that there exist constants $C_{17}$ and $C_{18}$ such that 
\begin{align*}
    \left|P_{i,\,i+t_0+1}\right| &\leq C_{10}, \\ \abs{\qty(P^{-1})_{n-t_0,\,n+2}} &\leq C_{11}
\end{align*}
for all $1 \leq i \leq 2n-t_0-1$. This completes the proof by mathematical induction.
\end{proof}

\section{Uniqueness in a bounded case}\label{Uniqueness in bounded case}
Let $X$ be a hyperbolic Riemann surface (maybe non-compact) and $g_X$ be the unique complete K\"ahler hyperbolic metric on $X$. The canonical line bundle of $X$ is denoted by $K$.

For each $\symbfit{q} = (q_{1},\dots, q_{n-1}, q_{n}) \in \mathop{\bigoplus} \limits_{i=1}^{n-1} H^{0}(X,K^{2i}) \bigoplus H^{0}(X,K^{n}) $, we can construct an $\mathrm{SO}_0(n,n)$-Higgs bundle in the Hitchin section, which is denoted by $(\mathbb{K}_{\mathrm{SO}_0(n,n)}, \theta (\symbfit{q}))$. Recall that $g_X$ induces a diagonal harmonic metric $h_X$ of $(\mathbb{K}_{\mathrm{SO}_0(n,n)}, \theta (\symbf{0}))$ (see Proposition~\ref{h_X}). In this section, we shall prove the uniqueness of the harmonic metric that is compatible with the $\mathrm{SO}_0(n,n)$-structure on any non-compact hyperbolic Riemann surface under certain bounded assumptions. Our previous work has already demonstrated the existence of such a harmonic metric.

\begin{theorem} \label{uniqueness in mutually bounded case}
Let $h_1$ and $h_2$ be two harmonic metrics of $(\mathbb{K}_{\mathrm{SO}_0(n,n)}, \theta\qty(\symbfit{q}))$. Suppose that both $h_i$ are compatible with the $\mathrm{SO}_0(n,n)$-structure and mutually bounded with $h_X$. If $\theta\qty(\symbfit{q})$ is bounded with respect to $h_X$ and $g_X$, then $h_1=h_2$.
\end{theorem}

Let us explain the mutually bounded condition. 
\begin{definition}
For a Hermitian metric $h$ on $\mathbb{K}_{\mathrm{SO}_0(n,n)}$, let $s(h_X,h)$ be the automorphism of $\mathbb{K}_{\mathrm{SO}_0(n,n)}$ determined by $h = h_X \cdot s(h_X,h)$. We say $h$ is mutually bounded with $h_X$ if there exists a constant $C > 0$ such that 
\begin{align*}
    \abs{s(h,h_X)}_{h_X}+\abs{s(h,h_X)^{-1}}_{h_X}\leq C
\end{align*}
holds on $X$.
\end{definition}

Invoking the above theorem, we can prove the uniqueness of the harmonic metric which weakly dominates $h_X$ when holomorphic differentials are bounded. The following proof of Theorem~\ref{uniqueness in bounded differential case} given Theorem~\ref{uniqueness in mutually bounded case} is almost identical to the one in \cite[Proposition 5.3]{li2023higgs}.

\begin{theorem}\label{uniqueness in bounded differential case}                                             
    When $q_1,\dots ,q_n$ are all bounded with respect to $g_X$, the harmonic metric of  
    $(\mathbb{K}_{\mathrm{SO}_0(n,n)},\theta\qty(\symbfit{q}))$ which is compatible with the $\mathrm{SO}_0(n,n)$-structure and weakly dominates $h_X$ is unique. 
\end{theorem}
\begin{proof}
Let $h$ be a harmonic metric of $(\mathbb{K}_{\mathrm{SO}_0(n,n)},\theta\qty(\symbfit{q}))$ which satisfies the two conditions in this theorem. By Theorem~\ref{uniqueness in mutually bounded case}, we only need to check that $h$ is mutually bounded with $h_X$. Recall that $h_X= \mathop{\oplus} \limits_{k=1}^{2n-1} a_{k,\,n} \cdot g_X^{k-n} \oplus h_0$, where $a_{k,\,n}$ are some positive constants and $h_0$ is the constant metric $1$ on the trivial line bundle.

The boundedness of $\symbfit{q}$ implies that the eigen-forms of $\theta\qty(\symbfit{q})$ are all bounded with respect to $g_X$. By \cite[Proposition~3.12]{li2020complete}, there exists a constant
$C >0$ such that
\[
\abs{\theta\qty(\symbfit{q})}_{h,g_X} \leq C
\]
holds on $X$, i.e., $\theta\qty(\symbfit{q})$ is bounded with respect to $h$ and $g_X$. Note that $\theta\qty(\symbfit{q})$ is also bounded with respect to $h_X$ and $g_X$.

For any point $x \in X$, choose $\tau_x \in K|_x$ such that $\abs{\tau_x}_{g_X}=1$. Set
\begin{align*}
    e_k= 
        \begin{cases}
            \frac{1}{a_{k,\,n}} \tau_x^{n-k}, &1 \leq k \leq n, \\
              1 & k=n+1, \\
              \frac{1}{a_{k-1,\,n}} \tau_x^{n-k+1}, &
              n+2 \leq k \leq 2n.
        \end{cases}
\end{align*}
 We get an orthonormal basis $\symbfit{e}=\qty(e_1,\dots,e_{2n})$ of $\mathbb{K}_{\mathrm{SO}_0(n,n)}|_x$ with respect to $h_X$. Let $A$ be the matrix determined by $\theta(\symbfit{q})|_x(\symbfit{e})=\symbfit{e}\cdot A\tau_x$. Because $\theta(\symbfit{q})$ is bounded with respect to $h_X$ and $g_X$, there exists $B_1> 0$ which is independent of $x$ such that $\abs{A}_{\infty}\leq B_1$. Additionally, $A$ satisfies the conditions in Proposition~\ref{estimate P pro}. Recall that
 \[
  \widetilde{\left| A \right|_{\infty}} \coloneqq \max \{\max _{k \neq n,\, n+1}\left|\left(A_{k+1,\,k}\right)^{-1}\right|,\abs{\left(A_{n+2,\,n}\right)^{-1}}\}.
 \]
 So we have $\widetilde{\abs{A}_{\infty}}=1$.
 
Applying the Gram-Schmidt process to the basis $\symbfit{e}$ and the metric $h|_x$, we obtain an orthonormal basis $\symbfit{v}$ of $\mathbb{K}_{\mathrm{SO}_0(n,n)}|_x$ with respect to $h|_x$. Let $P$ be the upper triangular matrix determined by $\symbfit{v}=\symbfit{e}\cdot P$. Since $\theta(\symbfit{q})$ is bounded with respect to $h$ and $g_X$, there exists $B_2>0$, which is independent of $x$, such that$\abs{P^{-1}AP}\leq B_2$. Because $h$ weakly dominates $h_X$, we obtain $P_{1,1}\geq c$, where $c$ is a positive constant independent of $x$. By Proposition~\ref{estimate P pro}, there exists $B_3>0$, which is also independent of $x$, such that $\abs{P}+\abs{P^{-1}}\leq B_3$. Therefore, there exists a positive constant $B_{4}$ such that
\[
\left|s(h,h_X)\right|_{h_X}+\left|s(h,h_X)^{-1}\right|_{h_X}\leq B_4
\]
holds on $X$.
\end{proof}

\subsection{Preliminaries on Linear Algebra}

\subsubsection{Quasi-cyclic vectors}
Consider an $m$-dimensional complex vector space $E$ equipped with an endomorphism $f$. A vector $e\in E$ is called an $f$-cyclic vector if $e,f(e),\dots,f^{m-1}(e)$ generate $E$. The following proposition is well known. 
\begin{proposition}\label{cyclic vector dimension 1}
 The endomorphism $f$ has a cyclic vector if and only if the eigenspace corresponding to each eigenvalue of 
 $f$ has dimension $1$.
\end{proposition}

We need the following weaker condition that can be applied to our case.
\begin{definition}
  Let $E$ be an $m$-dimensional complex vector space equipped with an endomorphism $f$. A vector $e\in E$ is called an $f$-quasi-cyclic vector if $e,f(e),\dots,f^{m-2}(e)$ spans a vector subspace of dimension $m-1$.
\end{definition}

Consider a Hermitian metric $h$ on $E$. For any $e \in E$, set $\omega(f,e) = e\wedge f(e) \wedge \dots \wedge f^{m-2}(e)$. Then, $e$ is an $f$-quasi-cyclic vector of $E$ if and only if $\omega(f,e)\neq0$. Obviously, $\abs{\omega(f,e)}_h \leq \abs{f}_h^{(m-2)(m-1)/2}\abs{e}_h^{m-1}$. The following lemma is similar to \cite[Lemma~2.16]{li2023higgs} and the proofs of the two lemmas are almost identical. 
\begin{lemma}\label{quasicyclic}
Let $A$ and $\rho$ be positive constants.
Assume that
\begin{align*}
\abs{f}_h \leq A, \quad  \rho\abs{e}_h^{m-1} \leq \abs{\omega(f,e)}_h
\end{align*}
for some non-zero element $e\in E$. Then 
there exists a constant $\epsilon_0(m,A,\rho)>0$ depending only on $m$, $A$ and $\rho$, such that for every endomorphism $g$ of $E$ with $\abs{f-g}_h \leq \epsilon_0(m,A,\rho)$, 
\begin{align*}
    \frac{1}{2}\rho\abs{e}_h^{m-1}<\abs{\omega(g,e)}_h.
\end{align*}
Consequently, $e$ is a $g$-quasi-cyclic vector.
\end{lemma}
\begin{proof}
We can assume $\abs{f-g}_h<1$. Then 
\begin{align*}
\left|f^j(e)-g^j(e)\right|_h\leq\sum_{k=1}^j\dbinom{j}{k}\abs{f}_h^{j-k}\abs{f-g}_h^k|e|_h\leq\abs{f-g}_h(1+\abs{f}_h)^j\abs{e}_h.
\end{align*}
We get
\begin{align*}
& \left| e\wedge f( e)\wedge\dots\wedge f^{j-1}(e)\wedge(f^j(e)-g^j(e))\wedge g^{j+1}(e)\wedge\dots g^{m-2}(e)\right|_h \\
\leq & \abs{e}_h^{m-1}\cdot \abs{f}_h^{j(j-1)/2}\cdot
\abs{f^j-g^j}_h\cdot \abs{g}_h^{(m-1)(m-2)/2-j(j+1)/2}\\
\leq & \abs{e}_h^{m-1}\cdot \abs{f-g}_h\cdot(1+|f|_h)^{(m-1)(m-2)/2}.
\end{align*}
Hence,
\begin{align*}
\left|\omega(f,e)-\omega(f_1,e)\right|_h\leq (m-1)(1+|f|_h)^{(m-1)(m-2)/2}|f-f_1|_h\cdot|e|_h^{m-1}.
\end{align*}
Now, this lemma is clear.
\end{proof}
\subsubsection{Real structures and self-adjoint endomorphisms} 
Let $E$ be a finite dimensional complex vector space which is equipped with a non-degenerate symmetric pairing $C$. Consider a self-adjoint endomorphism $f$ of $E$. Let $E_\alpha$ be the generalized eigenspace corresponding to the eigenvalue $\alpha$ of $f$. We have $E=\bigoplus_{\alpha\in\mathbb{C}}E_\alpha$. The following lemma is well known, so we omit the proof.
\begin{lemma}\label{C orthogonal}
If $\alpha\neq\beta$, then $E_{\alpha}$ and $E_{\beta}$ are orthogonal with respect to $C$.
\end{lemma} 
 
Equip $E$ with a Hermitian metric $h$ which is compatible with $C$ (see Section~\ref{Preliminaries on compatibility of symmetric pairings and Hermitian metrics}). Then the metric $h$ and the pairing $C$ induce a real structure $\kappa$, which is determined by  
\begin{align*}
 h\qty(\cdot,\cdot)=C\qty(\cdot,\kappa(\cdot)).
\end{align*}
Note that the compatibility condition ensures $\kappa^2=1$. For more details, we refer readers to see \cite[Section 2]{Li2023generically}. 
 
To prove the key Theorem~\ref{estimate for s}, we first require several propositions.
\begin{proposition} \cite[Proposition~2.11]{li2023higgs} \label{U=k(U)=0}
Let $U$ be a subspace of $E$ which satisfies the following conditions: 
\begin{itemize}
    \item $f\qty(U) \subset U$ and $f\qty(\kappa\qty(U)) \subset \kappa\qty(U)$,
    \item $U \cap \kappa(U) =0$.
\end{itemize}
 Then either $U=0$, or $f|_U$ and $f|_{\kappa(U)}$ have a common eigenvalue.
\end{proposition}

We adapt Corollary~2.12 of \cite{li2023higgs} to our setting.
\begin{corollary}\label{0}
Let $U$ be a subspace of $E$ which satisfies the two conditions in Proposition~\ref{U=k(U)=0}.
 
Suppose that there exists a $f$-quasi-cyclic vector. Then we have 
either $U=\kappa\qty(U)=0$, or $\dim U=\dim \kappa\qty(U)=1$. Moreover, in the second case, $U \oplus \kappa\qty(U)$ is an eigenspace of $f$.
\end{corollary}
\begin{proof}  
 For $\dim E \leq 1$, it is trivial. So we assume $\dim E \geq 2$. Let $e$ be a quasi-cyclic vector of $f$. The linear span of $\qty{e,f(e),\dots,f^{\dim E-2}(e)}$ is denoted by $E'$. 
\begin{itemize} 
    \item Suppose $f$ has a cyclic vector. We want to show that $U=0$. If not, Proposition~\ref{U=k(U)=0} tells us that $f$ has a eigenspace whose dimension is at least $2$. By Proposition~\ref{cyclic vector dimension 1}, this is impossible.
    \item Suppose $f$ has no cyclic vector. The vector $e$ is just a quasi-cyclic vector of $f$, i.e., $f\qty(E') \subset E'$ and $\dim \qty(E')=\dim \qty(E)-1$. 
    Let $L$ be a one dimensional vector space such that  $E= E' \oplus L$. 
    Since $f(E') \subset E'$, with respect to this decomposition, the matrix of $f$ has the block upper triangular form
    \begin{align} \label{matrix of f}
        \begin{pmatrix}
            A & \alpha \\
            0 & \lambda
        \end{pmatrix},
    \end{align}    
   where $A$ is a $\dim(E') \times \dim(E')$ square matrix representing $f|_{E'}$, $\alpha \in \mathbb{C}^{\dim
   (E')}$ and $\lambda \in \mathbb{C}$.
    It's straightforward to check that $\lambda$ is an eigenvalue of $f$.
   Our next objective is to choose the complementary line $L$ such that $\alpha=0$.

    Let $\Lambda=\qty{\mu_1,\dots,\mu_{l}}$ be the set of all eigenvalues of $f|_{E'}$. For $\mu \in \Lambda$, denote by $V'_{\mu}$ and $\mathcal{V}'_{\mu}$ the eigenspace and the generalized eigenspace of $f|_{E'}$ corresponding to the eigenvalue $\mu$, respectively.
    Notice that $e\in E'$ is a $f|_{E'}$-cyclic vector. The Proposition~\ref{cyclic vector dimension 1} implies that $\dim (V_{\mu_i}')=1$ for each $i$. The full spectrum of $f$ is $\Lambda \cup \{\lambda\}$. Denote by $V_{\mu}$ and $\mathcal{V}_{\mu}$ denote the eigenspace and the generalized eigenspace of $f$ corresponding to the eigenvalue $\mu \in \Lambda \cup \{\lambda\}$.
    
     We first show that $\lambda \in \Lambda$.

    We have the direct sum 
    \[
    E'=\mathop{\oplus} \limits_{i=1}^l \mathcal{V}'_{\mu_i}.
    \]
    
    Assume $\lambda \notin \Lambda$. Then $ E' \cap V_{\lambda}=0 $.
    Since $\dim (E')=\dim(E)-1$ and $V_{\lambda} \neq 0$, we have $\dim(V_{\lambda
    })=1$ and 
    \[
    E=\mathop{\oplus} \limits_{i=1}^l \mathcal{V}'_{\mu_i} \oplus V_{\lambda}.
    \]
    Hence $V_{\mu_i}'$ (resp. $\mathcal{V}'_{\mu_i}$) is also the eigenspace (resp. generalized eigenspace) of $f$ corresponding to the eigenvalue $\mu_i$ for $1 \leq i \leq l$.
    Then we obtain that each eigenspace of $f$ has dimension $1$. By Proposition~\ref{cyclic vector dimension 1}, this would imply $f$ has a cyclic vector, contradicting our assumption. Thus $\lambda \in \Lambda$.

 Next we verify that $V_{\lambda} \nsubseteq E'$. 
 
 We have proved $\lambda \in \Lambda$, therefore $\Lambda=\qty{\mu_1,\dots,\mu_{l}}$ is exactly the spectrum of $f$.
From\eqref{matrix of f}, we know that 
\[
 \mathrm{a}(\lambda,f) = \mathrm{a}(\lambda, f|_{E'})+1,
\]
where $\mathrm{a}(\lambda,f)$ (resp. $\mathrm{a}(\lambda, f|_{E'})$) denotes the algebraic multiplicity of $\lambda$ with respect to $f$ (resp. $f|_{E'}$). Then
\[
\dim(\mathcal{V}_{\lambda}) = \dim(\mathcal{V}'_{\lambda}) +1.
\]
Recall 
\[
E=\mathop{\oplus} \limits_{i=1}^l \mathcal{V}_{\mu_i}, \quad E'=\mathop{\oplus} \limits_{i=1}^l \mathcal{V}'_{\mu_i}.
\]
Since $\mathcal{V}'_{\mu_i} \subset \mathcal{V}_{
\mu_i}$ and $\dim (E)=\dim(E')+1$, we have 
\[
\mathcal{V}'_{\mu_i} =  \mathcal{V}_{\mu_i} \quad \text{for} \quad \mu_i \neq \lambda.
\]
Hence also $ V_{\mu}' =  V_{\mu}$ for $\mu \neq \lambda$.
Since $f$ admits no cyclic vector, Proposition~\ref{cyclic vector dimension 1} implies the existence of an eigenvalue $\mu \in \Lambda$ with $ \dim(V_{\mu}) \geq 2 $.
For $\mu \in \Lambda$, we already have 
$\dim(V_{\mu}) =  \dim(V_{\mu}'=1$. Therefore, the only possibility is
\[\dim(V_{\lambda}) \geq 2 .
\]
In fact, the dimension count $\dim (E)=\dim(E')+1$ forces 
\[
\dim(V_{\lambda}) = 2.
\]
If $V_{\lambda} \subset E'$, then $f|_{E'}$ would have an eigenspace of dimension $2$, contradicting Proposition~\ref{cyclic vector dimension 1} (since $f|_{E'}$ admits a cyclic vector). Hence $V_{\lambda} \nsubseteq E'$.
   
Choose $u \in V_{\lambda}$ but not in $E'$. Let $L$ be the line spanned by $u$. So we obtain a decomposition $E=E' \oplus L $ such that $f\qty(E') \subset E', f(L) \subset L$. Consequently, with respect to this decomposition, the matrix of $f$ takes the block diagonal form
    \begin{align*}
        \begin{pmatrix}
            A & 0 \\
            0 & \lambda
        \end{pmatrix},
    \end{align*}    
   where $A$ is a $\dim(E') \times \dim(E')$ square matrix representing $f|_{E'}$. Thus $f$ has a $1 \times 1$ Jordan block for the eigenvalue $\lambda$. 
     
     Let $V_{\mu}$ be the eigenspace corresponding to the eigenvalue $\mu \in \Lambda$ of $f$. We have proved
     \begin{align*}
         \dim(V_{\mu}) = 
         \begin{cases}
             1 , \, & \mu\neq \lambda,\\
             2 , & \mu=\lambda.
         \end{cases}
     \end{align*}
    
     With the preliminaries settled, we now prove this proposition. Suppose
     $U$ and $\kappa\qty(U)$ are both nonzero. Then $f|_{U}$ and $f|_{\kappa\qty(U)}$ have only one common eigenvalue $\lambda$, and the eigenspaces corresponding to $\lambda$ of $f|_{U}$ and $f|_{\kappa\qty(U)}$ are both one dimensional. Let $u \in U$ be a generalized eigenvector of $f|_{U}$ for $\mu \neq \lambda$. Then $h\qty(u,u)=C\qty(u,\kappa\qty(u))=0$ by Lemma~\ref{C orthogonal}. 
     Hence $U$ is the space of generalized eigenvectors of $f|_{U}$ for eigenvalue $\lambda$. The subspace $\kappa\qty(U)$ is similar. Now using the uniqueness of the Jordan normal form of $f$, we obtain $\dim U=\dim \kappa\qty(U)=1$ and $U \oplus \kappa\qty(U)$ is the eigenspace of $f$ corresponding to $\lambda$.
\end{itemize}
\end{proof}

\subsubsection{The key theorem}

Consider two complex vector spaces $V,W$ with dimension $n$. Let $C_V,C_W$ be non-degenerate symmetric pairings on $V,W$ respectively. Consider a linear map $\theta \colon W\to V$. Denote by $\theta^{\dagger}$ the adjoint of $\theta$ with respect to $C_V,C_W$. 
 Let 
\begin{align*}
    E= V\oplus W, \quad
    C=
    \begin{pmatrix}
        C_V &  \\
           & C_W
    \end{pmatrix}, \quad
    f=
    \begin{pmatrix}
      0  & \theta \\
       \theta^{\dagger} &  0
    \end{pmatrix}.
\end{align*}
Note that $C$ is a non-degenerate symmetric pairing of $E$ and $f$ is self adjoint with respect to $C$.
    
Suppose that $E$ is equipped with a basis $e=(e_1,\dots,e_{2n})$ such that 
\begin{align*}
    f(e_k)=e_{k+1}+\sum_{j\leq k}\mathcal{A}(f)_{j,\,k}e_j \quad k=1,\dots,2n-2,
\end{align*}
i.e., $e_1$ is a $f$-quasi-cyclic vector.

For $\rho> 0$, let $\mathcal{H} ( E, \symbfit{e}, C_V,C_W; \rho)$ be the space of Hermitian metrics $h$ of $V$ such that 
\begin{itemize}
    \item $h=h|_V \oplus h|_W$ where $h|_V, h|_W$ are compatible with $C_V,C_W$ respectively,
    \item $\abs{e_1\wedge \dots \wedge e_{2n-1}}_h \geq \rho\abs{e_1}^{2n-1}$.
\end{itemize}
For $h,h^{\prime}\in\mathcal{H} ( E, \symbfit{e}, C_V,C_W; \rho)$, recall that the automorphism $s( h, h^{\prime})$ of $E$ is determined by 
\[
h^{\prime}(u,v)=h(s(h,h^{\prime})u,v)
\]
for any $u,v\in E$.

The following theorem is motivated by \cite[Proposition 2.15]{li2023higgs}.
\begin{theorem}\label{estimate for s} 
Let $A>0$ and $\rho>0$. There exist $\epsilon\qty(n,\,A,\,\rho)>0$ and $C\qty(n,\,A,\,\rho)>0$ depending only on $n$, $A$ and $\rho$ such that the following holds for any $0<\epsilon<\epsilon\qty(n,\,A,\,\rho)$.

Suppose $\abs{f}_{h}\leq A$ for a Hermitian metric $h \in \mathcal{H} ( E, \symbfit{e}, C_V,C_W; \rho)$. Then for any $h' \in \mathcal{H} ( E, \symbfit{e}, C_V,C_W; \rho)$ such that $\abs{\comm{s\qty(h,h')}{f}}_{h}\leq \epsilon$, we have
\begin{align*}
\abs{s\qty(h,h')-\id_{E}}_h \leq C\qty(n,\,A,\,\rho)\epsilon.
\end{align*}
\end{theorem}
\begin{proof}
Recall that the automorphism $s( h, h^{\prime})$ of $E$ is self-adjoint with respect to both $h$ and $h^{\prime}$. We obtain the eigen-decomposition $E= \bigoplus _{a> 0}E_a$ of $s( h, h^{\prime})$, which is orthogonal with respect to $h$ and $h'$ simultaneously.

Let $\kappa$ be the real structure induced by $C$ and $h$. It is worth noting that $\kappa(E_a)=E_{a^{-1}}$ by \cite[Lemma~2.13]{Li2023generically}. Set $\mathcal{S}(h,h^{\prime})\coloneq \{a>1\mid E_a\neq0\}$. If $\mathcal{S}(h, h^{\prime}) = \emptyset$, we get $s(h,h^{\prime})=\mathrm{id}_E$. So we consider the case of $\mathcal{S}(h,h^{\prime})\neq\emptyset$.

Let $\nu$ be any positive number such that $\nu\leq\min\{1,\max\mathcal{S}(h,h^{\prime})-1\}$. Let $c_1<c_2<\dots<c_m$ denote the elements of $\mathcal{S}(h,h^{\prime})$. Set $c_0=1$. Because $\abs{\mathcal{S}(h,h^{\prime})}\leq n$, there exists $1\leq l \leq m$ such that 
\begin{itemize} 
    \item $ c_i- c_{i-1}\leq \frac 12n^{-1}\nu$ for $i<l$,
    \item $ c_l-c_{l-1}>\frac{1}{2}n^{-1}\nu$.
\end{itemize}

We divide $\mathcal{S}(h,h^{\prime})$ into two subsets. Set $\mathcal{S} ( h, h^{\prime}; \nu)_0= \{ c_1, \dots , c_{l - 1}\} $ and $\mathcal{S} ( h, h^\prime; \nu) _1= \{ c_{l}, \dots , c_m\}$. Note that $\mathcal{S} ( h, h^{\prime}; \nu)_0$ maybe empty but $\mathcal{S} ( h, h^\prime; \nu) _1$ must be non-empty.
\begin{lemma}\label{the difference of eigenvalues} For any $a_0\in\mathcal{S}(h,h^{\prime};\nu)_0\cup\{1\}$ and $a_1\in\mathcal{S}(h,h^{\prime};\nu)_1$, we have $a_0^{-1}-a_1^{-1}\geq\frac1{12}n^{-1}\nu$.
\end{lemma}
\begin{proof}
    Because $\nu\leq1$, we obtain the first claim. The second claim is clear. For any $a_0\in S(h,h^{\prime};\nu)_0\cup\{1\}$ and
 $a_1\in\mathcal{S}(h,h'\nu)_1$, we obtain
\begin{align*}
    a_0^{-1}-a_1^{-1} \geq 
    a_0^{-1}-(a_0+n^{-1}\nu/2)^{-1} = a_0^{-1}\qty(a_0+n^{-1}\nu/2)^{-1}\frac12n^{-1}\nu\geq\frac12\cdot\frac13\cdot\frac12n^{-1}\nu=\frac1{12}n^{-1}\nu.
\end{align*}
\end{proof}

Set
\[
U^{(\nu)}=\bigoplus_{a\in\mathcal{S}(h,h^{\prime};\nu)_1}E_a,\quad E^{(\nu)}=E_1\oplus\bigoplus_{a\in\mathcal{S}(h,h^{\prime};\nu)_0}E_a\oplus\bigoplus_{a^{-1}\in\mathcal{S}(h,h^{\prime};\nu)_0}E_a.
\]
Since $\mathcal{S}(h,h^{\prime};\nu)_1\neq\emptyset$, we have $U^{(\nu)}\neq 0$. Notice that $U^{(\nu)}\cap\kappa(U^{(\nu)})=0$ since $\kappa(E_a)=E_{a^{-1}}$. We obtain the decomposition
\[
E=E^{(\nu)}\oplus U^{(\nu)}\oplus\kappa(U^{(\nu)}).
\]
For the endomorphism $f$ of $E$, we have the following decomposition
\[
f=\sum_{U_1,U_2 \in \{E^{(\nu)},U^{(\nu)},\kappa(U^{(\nu)})\}}f_{U_1,U_2},
\]
where $f_{U_1, U_2} \in \Hom( U_2, U_1)$. Set
\[
\widetilde{f}^{(\nu)}=f_{E^{(\nu)},E^{(\nu)}}+f_{U^{(\nu)},U^{(\nu)}}+f_{\kappa(U^{(\nu)}),\kappa(U^{(\nu)})}.
\]
Roughly speaking, this is the diagonal part of $f$ under the above decomposition.

Notice that $V,W$ are orthogonal with respect to $h$ and $h'$ simultaneously. Hence $s\qty(h,h')$ can be decomposed as $s_V \oplus s_W$,
 here $s_V \colon V\to V$ and $s_W \colon W \to W$. So we get
\begin{align*}
V&=V \cap E^{(\nu)} \oplus V \cap U^{(\nu)} \oplus V \cap \kappa(U^{(\nu)}), \\
W&=W \cap E^{(\nu)} \oplus W \cap U^{(\nu)} \oplus W \cap \kappa(U^{(\nu)}).
\end{align*}
Recall that 
\begin{align*}
    f=
    \begin{pmatrix}
      0  & \theta \\
       \theta^{\dagger} &  0
    \end{pmatrix},
\end{align*}
where $\theta \colon W \to V $ and $\theta^{\dagger} \colon V\to W$. By using the decompositions of $V$ and $W$ as presented above, we can define $\widetilde{\theta}^{(\nu)}\colon W\to V$ and $ \widetilde{\qty(\theta^{\dagger})}^{(\nu)}\colon V\to W$ which are analogous to $\widetilde{f}^{(\nu)}$. 
\begin{lemma}\label{theta}
We have 
\begin{align}
    \widetilde{f}^{(\nu)}=\widetilde{\theta}^{(\nu)} + \qty(\widetilde{\theta}^{(\nu)})^{\dagger}.
\end{align}
\end{lemma}
\begin{proof}
It is easy to verify that $\widetilde{f}^{(\nu)}= \widetilde{\theta}^{(\nu)} + \widetilde{\qty(\theta^{\dagger})}^{(\nu)}$. So we only need to check that 
\[
    \widetilde{\qty(\theta^{\dagger})}^{(\nu)} =\qty(\widetilde{\theta}^{(\nu)})^{\dagger}.
\]
We first prove that $\widetilde{f}^{(\nu)}$ is self adjoint with respect to $C$.

We abbreviate $U=U^{(\nu)}$ and put $\widetilde{U}=U\oplus\kappa(U)$. So we have the decomposition $E=E^{(\nu)}\oplus\widetilde{U}$, which is orthogonal with respect to $C$ by \cite[Lemma 2.10]{Li2023generically}. Consequently, the operator $f$ splits as 
\[
f=\sum_{U_1,U_2=E^(\nu),\widetilde{U}}f_{U_2,U_1}.
\] Since $f$ is self-adjoint with respect to $C$, we have $f_{E^{( \nu) }, E^{( \nu) }}$ and $f_{\widetilde {U}, \widetilde {U}}$ are self-adjoint with respect to $C$.

 Now consider the further splitting  $\widetilde{U}=U\oplus\kappa(U)$ and accordingly 
 \[
 f_{\widetilde{U},\widetilde{U}}=\sum_{U_1,U_2=U,\kappa(U)}f_{U_2,U_1}.\] Note that the restrictions of $C$ to $U$ and $\kappa(U)$ are $0$ since $U$ is $h$-orthogonal to $\kappa(U)$. Then, we can easily verify that $f_{U,U}+f_{\kappa(U),\kappa(U)}$ is self-adjoint with respect to $C$. Thus, we get $\widetilde{f}^{(\nu)}$ is self adjoint with respect to $C$.

 Now, we can verify this equality. Note that $\widetilde{f}^{(\nu)}=\widetilde{\theta}^{(\nu)}+\widetilde{\qty(\theta^{\dagger})}^{(\nu)}$ and $V$ is $C$-orthogonal to $W$. For any $v\in V,w\in W$, we have 
 \begin{align*}
     C_V\qty(\widetilde{\theta}^{(\nu)}\qty(w),v)
     = C\qty(\widetilde{f}^{(\nu)}\qty(w),v) 
     = C\qty(w,\widetilde{f}^{(\nu)}\qty(v)) 
     = C_W\qty(w,\widetilde{\qty(\theta^{\dagger})}^{(\nu)}\qty(v)).
 \end{align*}
\end{proof}

We have the following lemma.
\begin{lemma}\cite[Lemma~2.19]{li2023higgs}
    We have $\abs{f-\widetilde{f}^{(\nu)}}_h\leq\nu^{-1}(10n)^3\abs{\comm{f}{s(h,h^{\prime})}}_h$.
\end{lemma}

Recalling the positive constant $\epsilon_0(n,A,\rho)$ in Lemma~\ref{quasicyclic}, we set
\[
\epsilon_1(n,A,\rho) \coloneqq \frac{1}{2}(10n)^{-3}\epsilon_0(n,A,\rho),\quad C_1(n,A,\rho) \coloneqq 2n\epsilon_1(n,A,\rho)^{-1}.
\]
For $0< \epsilon< \epsilon_1( n, A, \rho)$, we assume $\abs{\left[s(h,h^{\prime}),f\right]}_h\leq\epsilon$. Let 
\[
\nu_1 \coloneq \epsilon_1(n,A,\rho)^{-1}\epsilon<1.
\]
If $\nu_1\leq \operatorname* { max} S( h, h^{\prime}) - 1, $ we obtain
\[
\abs{f-\widetilde{f}^{(\nu_1)}}_h\leq\nu_1^{-1}(10n)^3\abs{[f,s(h,h')}_h\leq\epsilon_1(n,A,\rho)(10n)^3\leq\epsilon_0(n,A,\rho).
\]
By Lemma~\ref{quasicyclic}, there exists an $\widetilde{f}^{(\nu_1)}$-quasi-cyclic vector. Since $U^{\qty(\nu_1)} \neq 0$, we get $\dim U^{\qty(\nu_1)}= \kappa \qty(U^{\qty(\nu_1)})=1$ according to Corollary~\ref{0}. Hence $U^{\qty(\nu_1)}=E_a,\kappa\qty(U^{\qty(\nu_1)})=E_{a^{-1}}$ for some $a > 1$.
Moreover, $E_a \oplus E_{a^{-1}}$ is the eigenspace associated with an eigenvalue $\lambda$ of $\widetilde{f}^{(\nu_1)}$.
The splittings of $h$ and $C$ ensure the splitting of $\kappa$. Therefore, either $E_a,E_{a^{-1}} \subset V$, or $E_a,E_{a^{-1}} \subset W$.

Without loss of generality, we can assume that $E_a,E_{a^{-1}} \subset V$. For a nonzero vector $v \in E_a \subset V$, we have
\[
\lambda v= \widetilde{f}^{(\nu_1)}(v)=\widetilde{\theta}^{(\nu_1)}(v)+\qty(\widetilde{\theta}^{(\nu_1)})^{\dagger}(v)=\qty(\widetilde{\theta}^{(\nu_1)})^{\dagger}(v) \in W.
\]
The second equality follows from Lemma ~\ref{theta}.
Hence $\lambda=0$ and $\qty(\widetilde{\theta}^{(\nu_1)})^{\dagger}(v)=0$.

Note that a linear transformation and its adjoint share the same spectrum. Hence there also exists a nonzero vector $w \in W$ such that $\widetilde{\qty(\theta)}^{(\nu_1)}(w)=0$, i.e., $E_{\widetilde{f}^{(\nu_1)},0}=E_a \oplus E_{a^{-1}} \nsubseteq V$, a contradiction.

Hence, we obtain $\max\mathcal{S}(h,h^{\prime})-1<\nu_1$. Therefore, $\abs{s- \id}_h\leq 2n\nu_1\leq C_1( n, A, \rho) \epsilon$.
\end{proof}
    \subsection{Proof of Theorem ~\ref{uniqueness in mutually bounded case}}
The following is identical
to the proof of \cite[Theorem 5.1]{li2023higgs}. For the convenience of readers, we include it here.
\begin{proof}
[Proof of Theorem ~\ref{uniqueness in mutually bounded case}]
For notational simplicity, we denote $\theta(\symbfit{q})$ by $\theta$.
Let $s$ be the automorphism of $\mathbb{K}_{\mathrm{SO}_0(n,n)}$ defined by $h_2 = h_1 \cdot s$. We recall the following inequality, originally due to Simpson \cite[Lemma 3.1]{simpson1988constructing} (a detailed discussion appears in \cite[Section 3.1.3]{li2023higgs}):
\[
\sqrt{-1}\Lambda_{g_X}\bar{\partial}\partial\operatorname{tr}(s)\leq-\big|\bar{\partial}(s)s^{-1/2}\big|_{h_1,g_X}^2-\big|\big[s,\theta\big]s^{-1/2}\big|_{h_1,g_X}^2.
\]
Here $g_X$ denotes the unique complete K\"ahler hyperbolic metric on $X$ and $\Lambda_{g_{X}}$ is the contraction operator with respect to the associated K\"ahler form $\omega$ of $g_X$. 

By the Omori–Yau maximum principle, there exists a sequence of points ${p_m}$ in $X$ such that
\[
\operatorname{tr}(s)(p_m)\geq\sup_X\operatorname{tr}(s)-\frac1m,\quad\sqrt{-1}\Lambda_{g_X}\bar{\partial}\partial\operatorname{tr}(s)\qty(p_m)\geq-\frac1m.
\]

Recall $h_1$ and $h_2$ are mutually bounded with $h_X$ and $\theta$ is bounded with respect to $h_X$ and $g_X$. Hence there exists $C_1>0$ satisfying
\[
\abs{\comm{s}{\theta}}_{h_1,g_X}^2 \qty(p_m) \leq \frac{C_1}{m}.
\]

Let $\tau_m$ be a frame of the canonical line bundle of $X$ at $p_m$ such that $|\tau_m|_{g_X}=1$. Set
\begin{align*}
    e_{m,k}= 
        \begin{cases}
             \tau_m^{n-k}, &1 \leq k \leq n, \\
              1 & k=n+1, \\
               \tau_m^{n-k+1}, &
              n+2 \leq k \leq 2n.
        \end{cases}
\end{align*}
Then we obtain a frame $\symbfit{e}=\qty{e_{m,1},\dots,e_{m,2n}}$ of $\mathbb{K}_{\mathrm{SO}_0(n,n)}|_{p_m}$. Since both $h_i$ are mutually bounded with $h_X$, we can find constants $B>1$ and $\rho$ such that $B^{-1} \leq |\tau_m^{n-1}|_{h_i}\le B$ and $\abs{e_{m,1}\wedge \dots \wedge e_{m,2n-1}}_{h_i} \geq \rho\abs{e_1}^{2n-1}$ for all $m$ and $i$. 

Let $f_m$ be the endomorphism of $\mathbb{K}_{\mathrm{SO}_0(n,n)}|_{p_m}$ defined by $\theta|_{p_m}=f_m\tau_m$. The boundedness of $\theta$ with respect to $h_i$ and $g_X$ implies the existence of $C_2>0$, independent of $m$, such that $|f_m|_{h_i}\le C_2$. Theorem~\ref{estimate for s} provides a constant $C_3>0$, independent of $m$, such which
\[
\abs{s-\mathrm{id}}_{h_1}\qty(p_m) \leq \frac{C_3}{\sqrt{m}}.
\]
Then there exists $C_4>0$ independently from $m$ such that
\[
2n\leq \tr(s) \leq \sup_X\operatorname{tr}(s)\leq\operatorname{tr}(s)(p_m)+\frac{1}{m}\leq 2n+\frac{C_4}{\sqrt{m}}+\frac{1}{m}.
\]
We obtain that $\tr(s)$ is constantly $2n$. Using the Cauchy-Schwarz inequality, we have $s=\mathrm{id}$.
\end{proof}
\section{Applications}\label{Applications}

Let $X$ be a non-compact hyperbolic Riemann surface and $K$ be its canonical line bundle. 
The unique complete K\"ahler metric with Gau{\ss} curvature $-1$ is denoted by $g_X$. It induces a diagonal harmonic metric $h_X$ of $(\mathbb{K}_{\mathrm{SO}_0(n,n)}, \theta (\symbf{0}))$ (see Proposition~\ref{h_X}).
We have proved the following theorem.
\begin{theorem} \label{existence and uniqueness of harmonic metric}
For any $\symbfit{q} = (q_{1},\dots, q_{n-1}, q_{n}) \in \mathop{\bigoplus} \limits_{i=1}^{n-1} H^{0}(X,K^{2i}) \bigoplus H^{0}(X,K^{n}) $, there exists a harmonic metric $h$ of   $(\mathbb{K}_{\mathrm{SO}_0(n,n)}, \theta (\symbfit{q}))$ such that 
    \begin{itemize}
        \item $h$ is compatible with the $\mathrm{SO}_0(n,n)$-structure,
        \item $h$ weakly dominates $h_X$.
    \end{itemize}
Moreover, when $\symbfit{q}$ are all bounded with respect to $g_X$, the harmonic metric which satisfies the above two conditions is unique.
\end{theorem}
We present two applications of this theorem and the weak domination property.
\subsection{Compact case}
 We can reprove the existence and uniqueness of the harmonic metric $h$ of $(\mathbb{K}_{\mathrm{SO}_0(n,n)}, \theta (\symbfit{q}))$ over a compact hyperbolic Riemann surface. Moreover, this also proves the domination property of $h$.
 \begin{theorem} \label{existence and uniqueness in compact case}
     Let $X$ be a compact hyperbolic Riemann surface. Then for any $\symbfit{q} = (q_{1},\dots, q_{n-1}, q_{n}) \in \mathop{\bigoplus} \limits_{i=1}^{n-1} H^{0}(X,K^{2i}) \allowbreak  \bigoplus H^{0}(X,K^{n}) $, there uniquely exists a harmonic metric $h$ of $(\mathbb{K}_{\mathrm{SO}_0(n,n)}, \theta (\symbfit{q}))$ such that 
     \begin{itemize}
    \item $h$ is compatible with the $\mathrm{SO}_0(n,n)$-structure,
    \item $h$ weakly dominates $h_X$, where $h_X$ is the Hermitian metric in Proposition ~\ref{h_X}.
\end{itemize}
 \end{theorem}
 \begin{proof}
     It' s well known that $X=\mathbb{D}/\Gamma$ for some $\Gamma < \operatorname{Aut}\qty(\mathbb{D})= \mathrm{PSL}\qty(2,\mathbb{R})$. 
     Then lift $(\mathbb{K}_{\mathrm{SO}_0(n,n)}, \theta (\symbfit{q}))$ to $\mathbb{D}$. We get the Higgs bundle $(\mathbb{K}_{\mathrm{SO}_0(n,n)}, \theta (\hat{\symbfit{q}})$ over $\mathbb{D}$. The Higgs bundle $\theta (\hat{\symbfit{q}})$ is invariant under $\Gamma$ and $\theta (\hat{\symbfit{q}})$ is bounded with respect to $g_{\mathbb{D}}$. By Theorem ~\ref{existence and uniqueness of harmonic metric}, the harmonic metric $\hat{h}$ exists and is invariant under $\Gamma$. Hence $\hat{h}$ descends to a harmonic metric on $(\mathbb{K}_{\mathrm{SO}_0(n,n)}, \theta (\symbfit{q}))$ over $X$. 
  
     The uniqueness is a direct result of Theorem ~\ref{uniqueness in bounded differential case} and the second claim is clear from Theorem~\ref{existence and uniqueness of harmonic metric} and the uniqueness.
 \end{proof}
 \subsection{Energy domination}
 Let $X$ be a Riemann surface and $K$ be its canonical line bundle. The unique complete K\"ahler metric with Gau{\ss} curvature $-1$ is denoted by $g_X$.
 
 For each $\symbfit{q} = (q_{1},\dots, q_{n-1}, q_{n}) \in \mathop{\bigoplus} \limits_{i=1}^{n-1} H^{0}(X,K^{2i}) \bigoplus H^{0}(X,K^{n})$, we can construct an $\mathrm{SO}_0(n,n)$-Higgs bundle in the Hitchin section, which is denoted by $(\mathbb{K}_{\mathrm{SO}_0(n,n)}, \theta (\symbfit{q}))$. Let $h$ be a harmonic metric which is compatible with the $\mathrm{SO}_0(n,n)$-structure.
The holonomy representation of the flat connection $\mathbb{D}_{h}=\nabla_h + \theta + \theta^{*h}$ is denoted by $\rho$. 
 Consider the universal cover of $X$, i.e., the Poincar\'e disk $\mathbb{D}$. Then the pair $\qty(\mathbb{K}_{\mathrm{SO}_0(n,n)},\mathbb{D}_{h})$ is isomorphic to $\mathbb{D} \times_{\rho} \mathbb{C}^{2n}$ equipped with the natural flat connection. A harmonic metric $h$, which is compatible with the $\mathrm{SO}_0(n,n)$-structure, is equivariant to a $\rho$-equivariant harmonic map $f$ from $\mathbb{D}$ to the symmetric space $\mathrm{SO}_{0}\qty(n,n)/\mathrm{SO}\qty(n)\times\mathrm{SO}\qty(n)$. Here we regard $\mathrm{SO}_{0}\qty(n,n)/\mathrm{SO}\qty(n)\times\mathrm{SO}\qty(n)$ as a totally geodesic submanifold of $\mathrm{SO}\qty(2n,\mathbb{C})/\mathrm{SO}\qty(2n,\mathbb{R})$ and we equip $\mathrm{SO}\qty(2n,\mathbb{C})/\mathrm{SO}\qty(2n,\mathbb{R})$ with the Riemannian metric which is induced by the Killing form $\operatorname{B}\qty(X,Y)=(2n-2)\tr\qty(XY)$. Note that $\mathrm{SO}_{0}\qty(1,1)/\mathrm{SO}\qty(1)\times\mathrm{SO}\qty(1)$ is just a point. From \cite{Li_2019}, the energy density of $f$ is 
\begin{align*}
     e\qty(f)=(2n-2)\abs{\theta\qty(\symbfit{q})}^2_{h,g_X}
     =(2n-2)
    \sqrt{-1}\Lambda_{g_{X}}\tr(\theta\qty(\symbfit{q})\wedge\theta\qty(\symbfit{q})^{*h}).
\end{align*}

Recall that $h_X$ in Proposition ~\ref{h_X}, which is induced by $g_X$, is a natural diagonal harmonic metric of $(\mathbb{K}_{\mathrm{SO}_0(n,n)}, \theta (\symbf{0}))$. We call the representation $\rho_0$, which is induced by $(\mathbb{K}_{\mathrm{SO}_0(n,n)}, \theta (\symbf{0}), h_X)$, is the base $n$-Fuchsian representation of $\qty(X,g_X)$.
\begin{theorem}\label{dominate Higgs field}
    Let $f_0$ be the $\rho_0$-equivariant harmonic map induced by $(\mathbb{K}_{\mathrm{SO}_0(n,n)}, \theta (\symbf{0}))$ and $h_X$. 
    If $h$ weakly dominates $h_X$, then $e\qty(f) \geq e\qty(f_0)$. Moreover, the equality holds at a point if and only if $h=h_X,\symbfit{q}=0$, i.e., $\rho=\rho_0,f=f_0$. 
\end{theorem}
\begin{proof}
    For $n=1$, this is trivial since $\mathrm{SO}_{0}\qty(1,1)/\mathrm{SO}\qty(1)\times\mathrm{SO}\qty(1)$ is just a point. So we assume $n \geq 2$.

    In the proof of Proposition~\ref{h dominates}, 
    we decompose $\mathbb{K}_{\mathrm{SO}_0(n,n)}$ as $L_1 \oplus L_2 \oplus \dots \oplus L_{2n}$ such that 
    \begin{itemize}
        \item the metric $h$ is given by 
        \begin{align*}
            h=
            \begin{pNiceMatrix}
            h_1 &          &     \\
            &  \Ddots  &      \\    
            &           & h_{2n}
            \end{pNiceMatrix},
        \end{align*}
        where $h_i$ is the restriction of $h$ on $L_i$ and $h_i=\det\qty(h|_{F_i})/\det\qty(h|_{F_{i-1}})$ (assume $\det\qty(h|_{F_0})=1$);
        \item  the holomorphic structure on $\mathbb{K}_{\mathrm{SO}_0(n,n)}$ is given by the $\bar{\partial}$-operator
        \begin{align*}
            \bar{\partial}_{\mathbb{K}_{\mathrm{SO}_0(n,n)}}=
            \begin{pNiceMatrix}
            \bar{\partial}_1 & \beta_{1,\,2} & \beta_{1,\,3} & \dots & \beta_{1,\,2n} \\
            & \bar{\partial}_2 & \beta_{2,\,3} & \dots & \beta_{2,\,2n} \\
            & & \bar{\partial}_3 & \dots & \beta_{3,\,2n} \\
            & & & \Ddots & \Vdots \\
            & & & & \bar{\partial}_{2n}
            \end{pNiceMatrix},
        \end{align*}
        where $\bar{\partial}_k$ is the $\bar{\partial}$-operators\ defining the holomorphic structure on $L_k$, and $\beta_{i,\,j} \in \Omega^{0,1}\left(\operatorname{Hom}\left(L_j, L_i \right)\right)$; 
        \item the Higgs field is of the form
        \begin{align*}
            \theta(\symbfit{q})=
            \begin{pNiceMatrix}[columns-width=0.2cm]
                a_{1,\,1} & \NotEmpty& \NotEmpty  & \NotEmpty &\NotEmpty  & \NotEmpty &\NotEmpty & \NotEmpty   &a_{1,\,2n} \\
                \theta_1 & \NotEmpty     &  &  &  &  &  &   & \NotEmpty\\
                &      &  &  &  &  &  &   &\NotEmpty\\
                &      & \theta_{n-1}&  &  &  &  &   & \NotEmpty\\
                &      &  & \gamma &  &  &  &  & \NotEmpty \\
                &      &  & \theta_n & \gamma' &   &  & &\NotEmpty\\
                &      &  &  &  & \theta_{n+1} &  &  & \NotEmpty\\
                &      &  &  &  &  &   &    &\NotEmpty \\
                &      &  &  &  &  &   & \theta_{2n-2}& a_{2n,\,2n} 
                \CodeAfter 
                \line{2-1}{4-3} \line{7-6}{9-8} 
                \line{1-1}{9-9} \line{1-2}{8-9}
                \line{1-3}{7-9} \line{1-4}{6-9}
                \line{1-5}{5-9} \line{1-6}{4-9} 
                \line{1-7}{3-9} \line{1-8}{2-9} 
                \line{1-1}{1-9} 
                \line{1-9}{9-9}
                \end{pNiceMatrix},
        \end{align*}
        where $a_{i,\,j} \in \Omega^{1,0}\left(\operatorname{Hom}\left(L_j, L_i\right)\right)$, $\theta_k$ is the identity morphism, $\gamma \in \Omega^{1,0}\left(\operatorname{Hom}\left(L_n, L_{n+1}\right)\right)$ and $\gamma'\in \Omega^{1,0}\left(\operatorname{Hom}\left(L_{n+1}, L_{n+2}\right)\right)$.
        \end{itemize}

 Set $u_k=\log \frac{h_k}{h_X|_{L_k}}$ and $v_k=\sum_{i=1}^{k}u_i= \log \frac {\det ( h|_{F_k})} {\det ( h_X|_{F_k}) }$ for $1\leq k\leq n$. Then 
\begin{align*}
    \frac{e(f)}{2n-2}&=\abs{\theta\qty(\symbfit{q})}^2_{h,g_X} \\
        &= \sum_{k=1}^{n-1}\frac{h_{k+1}}{h_k}/g_X + \frac{h_{n+2}}{h_n}/g_X + \sum_{k=n+2}^{2n-1}\frac{h_{k+1}}{h_k}/g_X + \sum_{i,j}\abs{a_{i,j}}^2_{h,g_X} + \abs{\gamma}^2_{h,g_X}  + \abs{\gamma'}^2_{h,g_X} \\
        &\geq  \sum_{k=1}^{n-1}\frac{h_{k+1}}{h_k}/g_X + \frac{h_{n+2}}{h_n}/g_X + \sum_{k=n+2}^{2n-1}\frac{h_{k+1}}{h_k}/g_X \\
        &= \sum_{k=1}^{n-1}\frac{h_X|_{L_{k+1}}}{h_X|_{L_k}}/g_X \cdot e^{u_{k+1}-u_{k}} +\frac{h_X|_{L_{n+2}}}{h_X|_{L_n}}/g_X \cdot  e^{u_{n+2}-u_{n}}+ \sum_{k=n+2}^{2n-1}\frac{h_X|_{L_{k+1}}}{h_X|_{L_k}}/g_X \cdot  e^{u_{k+1}-u_{k}} \\
        &=\sum_{k=1}^{n-1}\frac{k(2n-1-k)}{2}\cdot e^{u_{k+1}-u_{k}} + \frac{n(n-1)}{2}\cdot  e^{u_{n+2}-u_{n}}+\sum_{k=n+1}^{2n-2}\frac{k(2n-1-k)}{2}\cdot e^{u_{k+2}-u_{k+1}} \\
        &\geq\qty(\sum_{k=1}^{2n-2}\frac{k(2n-1-k)}{2}) 
        \exp(\frac{w_n}{\sum_{k=1}^{2n-2}\frac{k(2n-1-k)}{2}})\\ 
        &=\frac{n(n-1)(2n-1)}{3}\exp(\frac{w_n}{\frac{n(n-1)(2n-1)}{3}}).\\
        \intertext{Here,}
        w_{n}
        &=\sum_{k=1}^{n-1}\frac{k(2n-1-k)}{2}\qty(u_{k+1}-u_k)+ \frac{n(n-1)}{2}\qty(u_{n+2}-u_n)+ 
        \sum_{k=n+1}^{2n-2}\frac{k(2n-1-k)}{2}\qty(u_{k+2}-u_{k+1}) \\
        &=\sum_{k=1}^{n-1} (k-n)u_k +  \sum_{k=n+1}^{2n-1}(k-n)u_{k+1} \\
        &=2\sum_{k=1}^{n-1} (k-n)u_k\quad (u_k=-u_{2n+1-k} \text{ by equality \eqref{h equality}})\\
        &=-2\sum_{k=1}^{n-1}v_{k} \\
        &\geq 0.
\end{align*}

Therefore, $e(f) \geq \frac{2n(n-1)^2(2n-1)}{3}=e(f_0)$ and if the equality holds at a point $p$, then $v_{1}(p)=\dots=v_{n-1}(p)=0$.

Now, using the similar method in the proof of Proposition ~\ref{h dominates}, we obtain 
$\symbfit{q}=0, h=h_X$, i.e., $\rho=\rho_0,f=f_0$.
\end{proof}

When $X$ is compact, we have proved the unique harmonic metric $h$, which is compatible with the $\mathrm{SO}_0(n,n)$-structure, weakly dominates $h_X$. So we obtain the following corollary:
\begin{corollary} \label{energy domination}
    Suppose that $S$ is a closed orientable surface with genus $g\geq2$ and $g_0$ is a hyperbolic metric on $X$. Consider the universal cover $\widetilde{S}$ of $S$ and the lifted hyperbolic metric $\widetilde{g_0}$. Let $\rho$ be a Hitchin representation for $\mathrm{SO}_0(n,n)$ and $f$ be the unique $\rho$-equivariant harmonic map from $\qty(\widetilde{S},\widetilde{g_0})$ to $\mathrm{SO}_{0}\qty(n,n)/\mathrm{SO}\qty(n)\times\mathrm{SO}\qty(n)$. Then its energy density $e(f)$ satisfies 
    \begin{align*}
        e(f) \geq \frac{2n(n-1)^2(2n-1)}{3}.
    \end{align*}
    Moreover, the equality holds at a point if and only if $\rho$ is the base $n$-Fuchsian representation of $\qty(S,g_0)$.
\end{corollary}
\section*{Declarations}
 \subsection*{Conflict of interest} 
 The author states that there is no conflict of interest.
 \subsection*{Data Availability} Data sharing is not applicable to this article as no datasets were generated or analyzed.
\bibliographystyle{alpha}
\bibliography{reference}

@misc{garciaprada2012hitchinkobayashi,
  title         = {{The Hitchin-Kobayashi correspondence, Higgs pairs and surface group representations}},
  author        = {Oscar Garcia-Prada and Peter B. Gothen and Ignasi Mundet i Riera},
  year          = {2012},
  eprint        = {0909.4487},
  archiveprefix = {arXiv},
  primaryclass  = {math.DG},
  url           = {https://arxiv.org/abs/0909.4487}
}

@article{li2023higgs,
  author  = {Li, Qiongling and Mochizuki, Takuro},
  title   = {Higgs bundles in the Hitchin section over non-compact hyperbolic surfaces},
  journal = {Proceedings of the London Mathematical Society},
  volume  = {129},
  number  = {6},
  pages   = {e70008},
  doi     = {https://doi.org/10.1112/plms.70008},
  url     = {https://londmathsoc.onlinelibrary.wiley.com/doi/abs/10.1112/plms.70008},
  eprint  = {https://londmathsoc.onlinelibrary.wiley.com/doi/pdf/10.1112/plms.70008},
  year    = {2024}
}

@article{li2019harmonic,
  title         = {{Harmonic maps for Hitchin representations}},
  author        = {Li, Qiongling},
  journal       = {Geometric and Functional Analysis},
  volume        = {29},
  pages         = {539--560},
  year          = {2019},
  publisher     = {Springer},
  eprint        = {1806.06884},
  archiveprefix = {arXiv},
  primaryclass  = {math.DG}
}

@article{hitchin1992lie,
  title     = {{Lie groups and Teichm{\"u}ller space}},
  author    = {Hitchin, Nigel J},
  journal   = {Topology},
  volume    = {31},
  number    = {3},
  pages     = {449--473},
  year      = {1992},
  publisher = {Elsevier}
}

@article{hitchin1987self,
  author  = {Hitchin, N. J.},
  title   = {The Self-Duality Equations on a Riemann Surface},
  journal = {Proceedings of the London Mathematical Society},
  volume  = {s3-55},
  number  = {1},
  pages   = {59-126},
  doi     = {https://doi.org/10.1112/plms/s3-55.1.59},
  url     = {https://londmathsoc.onlinelibrary.wiley.com/doi/abs/10.1112/plms/s3-55.1.59},
  eprint  = {https://londmathsoc.onlinelibrary.wiley.com/doi/pdf/10.1112/plms/s3-55.1.59},
  year    = {1987}
}

@article{simpson1988constructing,
  title   = {{Constructing variations of Hodge structure using Yang-Mills theory and applications to uniformization}},
  author  = {Simpson, Carlos T},
  journal = {Journal of the American Mathematical Society},
  volume  = {1},
  number  = {4},
  pages   = {867--918},
  year    = {1988}
}

@article{simpson1990harmonic,
  title   = {{Harmonic bundles on noncompact curves}},
  author  = {Simpson, Carlos T},
  journal = {Journal of the American Mathematical Society},
  volume  = {3},
  number  = {3},
  pages   = {713--770},
  year    = {1990}
}

@article{biquard2004wild,
  title         = {{Wild non-abelian Hodge theory on curves}},
  author        = {Biquard, Olivier and Boalch, Philip},
  journal       = {Compositio Mathematica},
  volume        = {140},
  number        = {1},
  pages         = {179--204},
  year          = {2004},
  publisher     = {London Mathematical Society},
  eprint        = {math/0111098},
  archiveprefix = {arXiv},
  primaryclass  = {math.DG}
}

@article{Mochizuki2021good,
  author        = {Takuro Mochizuki},
  title         = {Good wild harmonic bundles and good filtered Higgs bundles},
  journal       = {Symmetry, Integrability and Geometry: Methods and Applications (SIGMA)},
  volume        = {17},
  year          = {2021},
  pages         = {068},
  note          = {66 pages},
  doi           = {10.3842/SIGMA.2021.068},
  archiveprefix = {arXiv},
  eprint        = {1902.08298},
  primaryclass  = {math.DG}
}

@article{Li2023generically,
  title     = {{Harmonic metrics of generically regular semisimple Higgs bundles on noncompact Riemann surfaces}},
  volume    = {5},
  issn      = {2576-7658},
  url       = {http://dx.doi.org/10.2140/tunis.2023.5.663},
  number    = {4},
  journal   = {Tunisian Journal of Mathematics},
  publisher = {Mathematical Sciences Publishers},
  author    = {Li, Qiongling and Mochizuki, Takuro},
  year      = {2023},
  month     = nov,
  pages     = {663–711}
}

@article{donaldson1992boundary,
  title     = {{Boundary value problems for Yang—Mills fields}},
  author    = {Donaldson, Simon K},
  journal   = {Journal of geometry and physics},
  volume    = {8},
  number    = {1-4},
  pages     = {89--122},
  year      = {1992},
  publisher = {Elsevier}
}

@article{maximum,
  author  = {Boyan Sirakov},
  title   = {Some estimates and maximum principles for weakly coupled systems of elliptic PDE},
  journal = {Nonlinear Analysis},
  volume  = {70},
  number  = {8},
  pages   = {3039--3046},
  year    = {2009}
}

@book{jost2012partial,
  title     = {Partial Differential Equations},
  author    = {Jost, J.},
  isbn      = {9780387493183},
  lccn      = {20069361},
  series    = {Graduate Texts in Mathematics, v.214},
  url       = {https://books.google.com/books?id=gpN2cYmZg2UC},
  year      = {2013},
  publisher = {Springer New York}
}

@article{Li_2019,
  title     = {{An Introduction to Higgs Bundles via Harmonic Maps}},
  issn      = {1815-0659},
  url       = {http://dx.doi.org/10.3842/SIGMA.2019.035},
  doi       = {10.3842/sigma.2019.035},
  journal   = {Symmetry, Integrability and Geometry: Methods and Applications},
  publisher = {SIGMA (Symmetry, Integrability and Geometry: Methods and Application)},
  author    = {Li, Qiongling},
  year      = {2019},
  month     = may
}

@article{fujioka2024harmonic,
  title     = {Harmonic metrics for Higgs bundles of rank 3 in the Hitchin section},
  author    = {Fujioka, Hitoshi and others},
  journal   = {SIGMA. Symmetry, Integrability and Geometry: Methods and Applications},
  volume    = {20},
  pages     = {111},
  year      = {2024},
  publisher = {SIGMA. Symmetry, Integrability and Geometry: Methods and Applications}
}

@article{donaldson1987twisted,
  author  = {Donaldson, S. K.},
  title   = {Twisted Harmonic Maps and the Self-Duality Equations},
  journal = {Proceedings of the London Mathematical Society},
  volume  = {s3-55},
  number  = {1},
  pages   = {127-131},
  doi     = {https://doi.org/10.1112/plms/s3-55.1.127},
  url     = {https://londmathsoc.onlinelibrary.wiley.com/doi/abs/10.1112/plms/s3-55.1.127},
  eprint  = {https://londmathsoc.onlinelibrary.wiley.com/doi/pdf/10.1112/plms/s3-55.1.127},
  year    = {1987}
}

@article{corlette1988flat,
  title     = {{Flat $ G $-bundles with canonical metrics}},
  author    = {Corlette, Kevin},
  journal   = {Journal of differential geometry},
  volume    = {28},
  number    = {3},
  pages     = {361--382},
  year      = {1988},
  publisher = {Lehigh University}
}

@article{Mochizuki_2016,
  title     = {{Asymptotic behaviour of certain families of harmonic bundles on Riemann surfaces}},
  volume    = {9},
  issn      = {1753-8424},
  url       = {http://dx.doi.org/10.1112/jtopol/jtw018},
  doi       = {10.1112/jtopol/jtw018},
  number    = {4},
  journal   = {Journal of Topology},
  publisher = {Wiley},
  author    = {Mochizuki, Takuro},
  year      = {2016},
  month     = sep,
  pages     = {1021–1073}
}

@book{mochizuki2010wild,
  author    = {Mochizuki, Takuro},
  title     = {{Wild harmonic bundles and wild pure twistor $D$-modules}},
  series    = {Ast\'erisque},
  publisher = {Soci\'et\'e math\'ematique de France},
  number    = {340},
  year      = {2011},
  mrnumber  = {2919903},
  zbl       = {1245.32001},
  language  = {en},
  url       = {https://arxiv.org/abs/0803.1344}
}

@article{biquard2020parabolic,
  title         = {{Parabolic Higgs bundles and representations of the fundamental group of a punctured surface into a real group}},
  author        = {Biquard, Olivier and Garc{\'\i}a-Prada, Oscar and i Riera, Ignasi Mundet},
  journal       = {Advances in Mathematics},
  volume        = {372},
  pages         = {107305},
  year          = {2020},
  publisher     = {Elsevier},
  eprint        = {1510.04207},
  archiveprefix = {arXiv},
  doi           = {https://doi.org/10.1016/j.aim.2020.107305}
}

@article{biswas1997parabolic,
  title         = {{Parabolic Higgs bundles and Teichm{\"u}ller spaces for punctured surfaces}},
  author        = {Biswas, Indranil and Ar{\'e}s-Gastesi, Pablo and Govindarajan, Suresh},
  journal       = {Transactions of the American Mathematical Society},
  volume        = {349},
  number        = {4},
  pages         = {1551--1560},
  year          = {1997},
  eprint        = {alg-geom/9510011},
  archiveprefix = {arXiv}
}

@article{li2020complete,
  title     = {Complete solutions of Toda equations and cyclic Higgs bundles over non-compact surfaces},
  author    = {Li, Qiongling and Mochizuki, Takuro},
  journal   = {International Mathematics Research Notices},
  volume    = {2025},
  number    = {7},
  pages     = {rnaf081},
  year      = {2025},
  publisher = {Oxford University Press}
}

@incollection{Garcia-Prada_2009,
  author    = {García-Prada, Oscar},
  title     = {Higgs bundles and surface group representations},
  booktitle = {Moduli Spaces and Vector Bundles},
  editor    = {Brambila-Paz, Leticia and Newstead, Peter and Bradlow, Steven B. and García-Prada, Oscar},
  publisher = {Cambridge University Press},
  series    = {London Mathematical Society Lecture Note Series},
  volume    = {359},
  year      = {2009},
  pages     = {265--310},
  address   = {Cambridge}
}

@article{PozzettiSambarinoWienhard+2021+1+51,
  url     = {https://doi.org/10.1515/crelle-2020-0029},
  title   = {Conformality for a robust class of non-conformal attractors},
  author  = {Maria Beatrice Pozzetti and Andrés Sambarino and Anna Wienhard},
  pages   = {1--51},
  volume  = {2021},
  number  = {774},
  journal = {Journal für die reine und angewandte Mathematik (Crelles Journal)},
  doi     = {doi:10.1515/crelle-2020-0029},
  year    = {2021}
}

@phdthesis{Arroyo2009TheGO,
  title  = {{The geometry of $\mathrm{SO}(p,q)$-Higgs bundles}},
  author = {Marta Aparicio Arroyo},
  school = {Universidad de Salamanca},
  year   = {2009},
  type   = {PhD thesis},
  url    = {https://api.semanticscholar.org/CorpusID:115808298}
}
\end{document}